\documentclass[10pt, reqno]{amsart} 
\usepackage[utf8]{inputenc}
\usepackage [english]{babel}
\usepackage{a4wide}
\usepackage{amssymb}
\usepackage{amsmath}
\usepackage{amsthm}
\usepackage{amsbsy}
\usepackage{mathrsfs}
\usepackage{amsfonts}
\usepackage{mathtools}
\usepackage{verbatim}
\usepackage{tikz}
\usepackage{tikz-cd}
\usetikzlibrary{cd}
\usepackage{amscd}
\usepackage{hyperref}
\usepackage{xcolor}
\hypersetup{
	colorlinks,
	linkcolor={red!50!black},
	citecolor={blue!50!black},
	urlcolor={blue!80!black}
}
\usepackage[normalem]{ulem}
\usepackage{cleveref}   
\usepackage[all]{xy}
\usepackage{xcolor}
\theoremstyle{plain}
\usepackage[margin=1 in, a4paper]{geometry}
\usepackage{relsize}
\usepackage{lipsum}
\usepackage[autostyle]{csquotes}
\usepackage{relsize}
\usepackage{enumitem}  
\usepackage{titlesec}
\usepackage{setspace}
\usepackage{adjustbox}
\usepackage{crossreftools}
\allowdisplaybreaks
\titleformat{\section}
{\centering\large\sc}{\thesection.}{1em}{}
\titleformat{\subsection}
{\centering\normalfont\bfseries}{\thesubsection.}{1em}{}
\titleformat{\subsubsection}[runin]
{\normalfont\normalsize\bfseries}{\thesubsubsection.}{1em}{}

\newcommand{\id}{\operatorname{id}}
\newcommand{\gl}{\operatorname{GL}}
\newcommand{\gal}{\operatorname{Gal}}

\newcommand{\res}{\operatorname{res}}

\newcommand{\Ker}{\operatorname{Ker}}

\newcommand{\N}{\mathbb{N}}
\newcommand{\Z}{\mathbb{Z}}
\newcommand{\Q}{\mathbb{Q}}
\newcommand{\C}{\mathbb{C}}

\newcommand{\lra}{\longrightarrow}

\newcommand{\Hom}{\mathrm{Hom}}

\newcommand{\cyc}{\mathrm{cyc}}

\newtheorem{theorem}{Theorem}[section]
\newtheorem*{theorem*}{Theorem}
\newtheorem{corollary}[theorem]{Corollary}
\newtheorem{lemma}[theorem]{Lemma}
\newtheorem*{lemma*}{Lemma}
\newtheorem{proposition}[theorem]{Proposition}
\newtheorem{conjecture}[theorem]{Conjecture}

\theoremstyle{definition}
\newtheorem{rem}[theorem]{Remark}
\newtheorem{remark}[theorem]{Remark}
\newtheorem{definition}[theorem]{Definition}

\newtheorem{example}[theorem]{Example}

\numberwithin{figure}{section}			
\numberwithin{table}{section}				

\begin{document}
	\title{Iwasawa theory for Rankin-Selberg product at an Eisenstein prime} 
	\author{Somnath Jha,  Sudhanshu Shekhar, Ravitheja Vangala} 
	\address{Somnath Jha, Department of Mathematics and Statistics, IIT Kanpur,  Kanpur-208016, India}\email{jhasom@iitk.ac.in} 
	\address{Sudhanshu Shekhar, Department of Mathematics and Statistics, IIT Kanpur,  Kanpur-208016, India}\email{sudhansh@iitk.ac.in}
	\address{Ravitheja Vangala, Department of Mathematics, IISc Bangalore,  Bangalore-560012, India}\email{ravithejav@iisc.ac.in}
	\thanks{\noindent AMS subject classification: 11F33, 11F67, 11R23.}
	\thanks{Keywords: Iwasawa main conjecture, $ p $-adic $ L $-functions, Rankin-Selberg product,  Selmer group.}
	\begin{abstract} {Let $p$ be an odd prime, $ f$ be a $ p $-ordinary newform of weight $ k $ and  $ h $ be a normalized cuspidal $ p $-ordinary Hecke eigenform of weight $ l < k$. In this article, we study the $p$-adic $ L $-function and $ p^{\infty} $-Selmer group of the Rankin-Selberg product of $f$ and $h$ under the assumption that $ p $ is an Eisenstein prime for $ h $ i.e. the residual Galois representation of $ h $ at $ p $ is reducible. We show that   the $ p $-adic $ L $-function  and the characteristic ideal of the $p^\infty$-Selmer group of the Rankin-Selberg product  of $f, h$ generate the same ideal modulo $ p $ in the Iwasawa algebra i.e. the Rankin-Selberg Iwasawa main conjecture  for $f \otimes h$ holds mod $p$. As an application to our results, we explicitly describe a few examples where the above congruence holds.}
	\end{abstract}
	\maketitle
	\section*{Introduction}
	The study of Rankin-Selberg product of two modular forms is of considerable interest in number theory from arithmetic, analytic and automorphic point of view. In this article, we study Iwasawa theory for  Rankin-Selberg product of two $ p $-ordinary modular forms $f$ and $h$ where the weight of $ f $ is strictly greater than the weight of $ h $ and $ p $ is an Eisenstein prime for $ h $.
	
	Iwasawa main conjecture for elliptic curves and modular forms (at an ordinary prime) has been established (in a large number of cases)  by the fundamental works of  Kato \cite{kato} and Skinner-Urban \cite{SU}. As a next step, it is natural to study the  Iwasawa theory of Rankin-Selberg $L$-function. Let $ f $ and $ h$ be two $ p $-ordinary modular forms.
	In this setting, important works have been done by
	Lei-Loeffler-Zerbes \cite{llz}, Kings-Loeffler-Zerbes \cite{klz} and X. Wan \cite{Wan}. In \cite[Theorem 7.5.4]{llz} and  \cite[Corollary D]{klz}, one side divisibility in the Iwasawa main conjecture was shown 
	(i.e. the characteristic ideal  of the $ p^\infty $-Selmer group of $ f \otimes h$ divides the $ p $-adic $ L $-function of $ f \otimes h$)  
	provided the residual Galois representation of $f \otimes h$ at $p$ is irreducible and assuming additional hypotheses.  In  \cite[Theorem 1.2]{Wan}, X. Wan proves the other divisibility in the main conjecture 
	under  suitable assumptions which include $ f $ is a CM form and the residual Galois representation of $h$  at $p$, say $  \bar{\rho}_h $ is irreducible. Thus following the works of \cite{llz}, \cite{klz} and \cite{Wan}, the Iwasawa main conjecture is still open in many cases; for example when  (i) the residual Galois representation of $f \otimes h$ at $p$  is reducible or (ii) $ f $ is a non-CM eigenform.
	
	In our setting, $ p $ is an Eisenstein prime for $ h $ i.e. $ \bar{\rho}_h $ is reducible, so our assumptions are complimentary to those mentioned in  \cite{llz}, \cite{klz} and \cite{Wan}. In fact, our result in this article can be interpreted in terms of  congruence of modular forms. Indeed, the congruences of modular forms (cf. \cite{Hidaadjoint}) is an important topic of study in the arithmetic of elliptic curves, modular forms and Iwasawa theory for congruent modular forms has been extensively studied  following the works of  \cite{gv}, \cite{Vatsal},  \cite{epw} etc. 
	In this  article, we combine  ideas from these two topics; Iwasawa theory of the Rankin-Selberg convolution  and the congruence of modular forms to arrive at our result. 
	
	We now briefly recall the terminology needed to introduce our main theorem. The precise definitions are given later  at appropriate places. Let  $ p $ be an odd prime. Let $ \mathbb{Q}_{\cyc} $ be the cyclotomic $ \mathbb{Z}_{p} $-extension of $ \mathbb{Q} $ and $ \Gamma := \mathrm{Gal}(\mathbb{Q}_{\cyc}/\mathbb{Q}) $.  Let $ f 
	\in S_{k}(\Gamma_{0}(N),\eta)$ be a $ p $-ordinary primitive form 
	with $ p \nmid N $ and $ h \in  S_{l}(\Gamma_{0}(I),\psi)$ be a normalised $ p $-ordinary Hecke eigenform 
	such  that   $2 \leq  l < k$.  Write $I =I_0 p^\alpha$ with $p \nmid I_0$. 
	Fix a number field $ K $ containing  the Fourier coefficients of $ f,h $ and the  values of $ \eta , \psi$.  Let $ K_\mathfrak{p}  $ denote the completion of $ K $ at a prime $ \mathfrak{p} $ lying above $ p $ induced by the embedding $ i_p: \bar{\mathbb{Q}}  \rightarrow \bar{\mathbb{Q}}_p $ and $ \pi $ be a uniformizer of the ring of integers $ \mathcal{O} $ of $ K_\mathfrak{p} $.  For $ \mathfrak{g} \in \{f,h\} $, let $ \rho_{\mathfrak{g}} : G_{\Q} \rightarrow \mathrm{Aut}(V_{\mathfrak{g}}) $  be the $ p $-adic Galois representation attached to $ \mathfrak{g} $ and  $ T_{\mathfrak{g}} $ be an $ \mathcal{O} $-lattice inside $ V_{\mathfrak{g}}$.  Set $ A_f := V_f/T_f$ and $ A_f^{-}  $ be the unramified quotient of $ A_f $ defined using the $ p $-ordinarity of $ f $.  The crucial assumption in this article is: 
	\begin{small}		
		\begin{align}\label{intro rho_h}
			\text{We have an exact sequence of } G_{\Q} \text{ modules, } \quad   0 \rightarrow \bar{\xi}_1 \rightarrow  T_h/\pi \rightarrow \bar{\xi}_2 \rightarrow 0  . \tag{$\bar{\rho}_h$-red}
		\end{align}
	\end{small}%
	As $ h $ is $ p $-ordinary, we may assume $ \bar{\xi}_{2} $ is unramified.  We choose appropriate (Teichm\"uller) lifts  $\xi_1$ and $\xi_2$ of $\bar{\xi}_1$ and $\bar{\xi}_2$ respectively as  described after Lemma~\ref{lemma lifting characters}.  Let $ \chi_{\cyc}$ and $\omega_{p} $ be the $ p $-adic cyclotomic character and Teichm\"uller character respectively. 
	Let  $ A_j $ be the discrete $ \mathcal{O} $-lattice $ \big( (V_f \otimes V_h)/ (T_f \otimes T_h) \big) \otimes \chi_{\cyc}^{-j} $. Using the $ p $-ordinarity and following \cite{gr1}, we define the  $ p^{\infty}$- Selmer groups $  S_{\mathrm{Gr}}(A_f/\Q_\cyc)$,  $  S_{\mathrm{Gr}}(A_j/\Q_\cyc)$ attached to $ A_f $ and $ A_j $ respectively.   In this setting, Hida \cite{Hidarankin2} has constructed the $ p $-adic $ L $-function $ \mu_{p, f \times h,j}   $  attached to the Rankin-Selberg convolution $ f \otimes h $  (see Theorem~\ref{padic rankin} for the definition).  Let $ \mu_{p,f, \xi_i,j} $ be the $ p $-adic $ L $-function attached to $ f \otimes \xi_i$ as constructed  in \cite{MTT}  with an appropriate choice of periods,   as explained in  Theorems~\ref{period and integral measure} and \ref{c(f) and petterson}.
	
	Let $\Sigma$ be a finite set of primes of $\Q$  containing primes  dividing $pNI\infty $ and  $\Q_\Sigma$ be the  {maximal} algebraic extension of $\Q$ unramified outside $\Sigma$. Define $ M_0 $ to be the prime to $ p $-part of the conductor of $ \bar{\rho}_h$ and $ \Sigma_0:=\{\ell \text{ prime}: \ell \mid  I_0/M_0 \text{ and } \ell^2\mid M_0\}$.  Put $  m := \prod_{\ell \in \mathcal{M}}\ell $ and $ \Sigma_{0}^{\infty} := \{ w  \text{ is a prime in } \Q_\cyc: w \mid m \}$. Let   $\mu^{\Sigma_0}_{p, f\times h,j}$, $\mu^{\Sigma_0}_{p,f, \xi_1,j}$ and $\mu^{\Sigma_0}_{p,f, \xi_2,j}$    be the $\Sigma_0$-imprimitive $p$-adic $L$-function as defined in \eqref{def: p-adic L-function f x h} and \eqref{def: p-adic L-function f} respectively. Let $  S^{\Sigma_0}_{\mathrm{Gr}}(A_f/\Q_\cyc)$,  $  S^{\Sigma_0}_{\mathrm{Gr}}(A_j/\Q_\cyc)$ be $\Sigma_0$-imprimitive Selmer group attached to $ A_f $ and $ A_j $ respectively. Let $G_p \subset G_{\Q}$ be the decomposition subgroup at $p$ and for a prime $w$ in $\Q_{\cyc}$ let $G_{\Q_{\cyc},w}$ and $I_{\cyc,w}$ be the decomposition and inertia subgroup at $w$. 
	Let $ \iota_{m} $ be  the trivial character of modulus $m$ i.e. 
	$
	\iota_m(n)  = 
	1   \text{ if }  (m,n) =1$ and $
	\iota_m(n)=0   \text{ if }  (m,n) >1
	$.
	We now state some hypotheses:
	\begin{enumerate}
		\item[{\crtcrossreflabel{(irr-$f$)}[irr-f]}] The residual representation $ \bar{\rho}_f $ is an irreducible $G_\Q$-module. 
		\item[{\crtcrossreflabel{($p$-dist)}[p-dist]}]  $f$    is   $p$-distinguished i.e. the restriction of  $\bar{\rho}_f$ to  $G_p$ satisfies 
		$ \bar{\rho}_{f}\vert_ {G_{p}} \cong \begin{psmallmatrix}
			\epsilon_{p} & \ast \\ 0 &  \delta_{p}
		\end{psmallmatrix}$  with  $\epsilon_{p}  \neq  \delta_{p}$. 
		\item[{\crtcrossreflabel{(Sel-tors)}[Stors]}] 
		$S_{\mathrm{Gr}}( A_j/\Q_\cyc)^\vee$ is a finitely generated \textit{torsion} module over the Iwasawa algebra $\mathcal{O}[[\Gamma]]$.
	\end{enumerate}	
	The following Theorems~\ref{thm: congrunence analytic intro}, \ref{thm: congrunence algebraic intro} and Corollary~\ref{congruence intro} are our main results (also see Examples~\ref{example 1}-\ref{example 3}): 
	
	\begin{theorem}\label{thm: congrunence analytic intro} [Theorem~\ref{analytic final}]
		Let $ p$ be an odd prime. Let $ f \in S_{k}(\Gamma_{0}(N),\eta)$ and $ h \in S_{l}(\Gamma_{0}(I),\psi)$ be normalised $ p $-ordinary newforms with $ p \nmid N $ and $ 2 \leq l < k$.  Let $j$ be an integer with $ l-1 \leq j \leq k-2 $. Assume $h$ satisfies \eqref{intro rho_h} and $f$ satisfies  \ref{irr-f} and  \ref{p-dist}.  If $p=3$, then we further assume that  $ \Gamma_{0}(N)/\{\pm 1\} $ has no non-trivial torsion elements.  
		Then  the following congruence of $ \Sigma_0$-imprimitive  $p $-adic $ L $-functions holds in $ \mathcal{O}[[\Gamma]] $:
		\begin{small}
			\begin{align}\label{intro congruence of p-adic L-functions}
				(\mu^{\Sigma_{0}}_{p,f \times  h,j}) \equiv (\mu^{\Sigma_{0}}_{p,f,\xi_1, j}) (\mu^{\Sigma_{0}}_{p,f,\xi_2, j}) \mod \pi.
			\end{align}
		\end{small}%
	\end{theorem}
	
	\begin{theorem}\label{thm: congrunence algebraic intro} [Theorem~\ref{thm:congruence of ideals}] Let $ p$ be an odd prime. Let $ f \in S_{k}(\Gamma_{0}(N),\eta)$ and $ h \in S_{l}(\Gamma_{0}(I),\psi)$ be normalised $ p $-ordinary newforms with $ p \nmid N $ and $ 2 \leq l < k$. Let $j$ be an integer with $ l-1 \leq j \leq k-2 $. Assume $h$ satisfies \eqref{intro rho_h} and $f$ satisfies  \ref{irr-f} and  \ref{p-dist}.  In addition, we make following hypotheses:
		\begin{enumerate}
			\item[$\mathrm{(i)}$] $ (N,M_0) =1 $ i.e. $N$ is co-prime to $M_0$.
			\item[$\mathrm{(ii)}$]  $ \psi|_{ I_{cyc,w}}$ has order co-prime to $ p $ for every prime $w$ in $\Q_{\cyc}$ dividing $pI$. 
			\item[$\mathrm{(iii)}$] 
				The dual Selmer group $S_{\mathrm{Gr}}( A_j/\Q_\cyc)^\vee$ satisfies \ref{Stors} and moreover the assumption \eqref{H2van} holds i.e. $H^2(\Q_\Sigma/\Q_{\cyc}, A_f(\xi_1\omega_p^{-j})[\pi]) $ vanishes. 
			\item[$\mathrm{(iv)}$]  Let $ w$  be the  prime in  $\Q_{\cyc}$ dividing $ p$.  If   $ (p-1) \mid j $ 
			and  $ H^{0}(G_{\Q_{\cyc,w}}, A_{f}^{-}[\pi],  (\bar{\xi}_2)) \neq 0$, then we assume \eqref{assumption} holds i.e. $ \bar{\rho}_h|I_{\cyc,w} \cong \bar{\xi}_1 \oplus \bar{\xi}_2 $.
		\end{enumerate}  
		Then the following congruence of the characteristic ideals of Selmer groups holds in $ \mathcal{O}[[\Gamma]] $:
			\begin{small}
				\begin{align}\label{intro congruence of char ideals}
					C_{\mathcal{O}[[\Gamma]]} \Big(S^{\Sigma_0}_{\mathrm{Gr}}( A_j/\Q_\cyc)^\vee\Big)  \equiv C_{\mathcal{O}[[\Gamma]]} \Big(S^{\Sigma_0}_{\mathrm{Gr}}(A_f(\xi_1 \omega_{p}^{-j})/\Q_\cyc)^\vee\Big) C_{\mathcal{O}[[\Gamma]]} \Big(S^{\Sigma_0}_{\mathrm{Gr}}( A_f(\xi_2 \omega_{p}^{-j})/\Q_\cyc)^\vee\Big) \mod \pi.
				\end{align}
			\end{small}%

		\end{theorem}
			\begin{corollary}\label{congruence intro}
				Let $j$ be an integer with $ l-1 \leq j \leq k-2 $. We keep the assumptions of both Theorems~\ref{thm: congrunence analytic intro},~\ref{thm: congrunence algebraic intro}. 
				Furthermore, assume that the Iwasawa main conjecture holds for $ f\otimes \xi_{1} \omega_p^{-j}$ and $ f \otimes  \xi_{2} \omega_p^{-j}$ over $ \Q_{\cyc} $, that is,  
				\begin{align*}
					(\mu_{p,f,\xi_i, j}) = C_{\mathcal{O}[[\Gamma]]}\Big(S_{\mathrm{Gr}}( A_f(\xi_i \omega_{p}^{-j})/\Q_\cyc)^\vee\Big) 
				\end{align*}
				holds for  $i=1,2$. Then  we have the following congruence of ideals in the Iwasawa algebra $ \mathcal{O}[[\Gamma]] $
				\begin{small}
					\begin{align}\label{intro IMC}
						(\mu_{p,f\times h,j}) \equiv C_{\mathcal{O}[[\Gamma]]}\Big(S_{\mathrm{Gr}}( A_j/\Q_\cyc)^\vee\Big) \mod \pi.
					\end{align}
				\end{small}%
				In particular, $ \mu_{p,f\times h,j} $ is a unit in $ \mathcal{O}[[\Gamma]] $ if and only if  $ C_{\mathcal{O}[[\Gamma]]}\Big(S_{\mathrm{Gr}}( A_j/\Q_\cyc)^\vee\Big) $ is a unit in $ \mathcal{O}[[\Gamma]] $. 
			\end{corollary}

			\begin{remark} The Iwasawa main conjecture for modular form (as stated in Conjecture~\ref{IWC for modular form}), which is needed to deduce congruence \eqref{intro IMC} from the congruences \eqref{intro congruence of p-adic L-functions} and \eqref{intro congruence of char ideals}, is known for  large class of modular forms by the results of \cite{kato} and  \cite{SU}. More precisely, to deduce \eqref{intro IMC} we  only need Iwasawa main conjecture modulo $\pi$ for $ f \otimes \xi_{i} \omega_p^{-j}$ i.e. the congruence $(\mu_{p,f,\xi_i, j} ) \equiv C_{\mathcal{O}[[\Gamma]]}\big(S_{\mathrm{Gr}}( A_f(\xi_i \omega_{p}^{-j})/\Q_\cyc)^\vee\big) \mod \pi$ holds for $ i =1,2 $. 

			\end{remark}
			\begin{rem}\label{remark mu invariant and H2} We discuss the hypotheses \ref{Stors} and \eqref{H2van} appearing as condition (iii) in  Theorem~\ref{thm: congrunence algebraic intro}.
				\begin{enumerate}
					\item[(i)] Recall  the definition of \textit{fine Selmer group} which was defined and studied by Coates-Sujatha (cf. \cite{CS}): $$R(A_{f}(\xi_1\omega_p^{-j}))/\mathbb{Q}_{\cyc})=  \mathrm{Ker}(H^1(\mathbb{Q}_{\Sigma}/\mathbb{Q}_{\cyc}, A_{f}(\xi_1\omega_p^{-j})[\pi]) \rightarrow \oplus_{w \in \Sigma} H^1(\mathbb{Q}_{\cyc,w}, A_{f}(\xi_1 \omega_p^{-j})[\pi]) ).$$ It was observed independently by Greenberg and Sujatha that the vanishing of $\mu$-invariant, that is, $\mu ( R(A_{f}/\mathbb{Q}_{\cyc})^\vee) =0$ is equivalent to $H^2(\Q_\Sigma/\Q_{\cyc}, A_f(\xi_1\omega_p^{-j}) [\pi]) =0$ (cf. \cite[ (40)]{js}). 
					Following \cite[Conjecture A]{CS}, it is expected that  $\mu ( R(A_{F}/\mathbb{Q}_{\cyc})^\vee) =0$  always holds for any cuspidal Hecke eigenform $F$ and any  prime $p$ (also see \cite{js}). Indeed, as explained in \cite[Corollary 3.6]{CS},  the vanishing of the $\mu$-invariant $\mu ( R(A_{F}/\mathbb{Q}_{\cyc})^\vee) =0$ is essentially a reformulation of  the classical $\mu =0$ conjecture of Iwasawa. Thus in our setting of Theorem~\ref{thm: congrunence algebraic intro}, the assumption \eqref{H2van} is expected to hold.
					
					\item[(ii)] Observe that in the setting of \cite{klz}, \cite{llz} it is known that  $S_{\mathrm{Gr}}( A_j/\Q_\cyc)^\vee$ is a finitely generated \textit{torsion} $\mathcal{O}[[\Gamma]]$-module using an Euler system argument. 
					
					\item[(iii)]  Note that if  $S_{\mathrm{Gr}}(A_f(\xi_i \omega_p^{-j})/\Q_\cyc)^\vee$ has $\mu$-invariant equal to $0$ i.e. $S_{\mathrm{Gr}}(A_f(\xi_i \omega_p^{-j})/\Q_\cyc)^\vee$ is a finitely generated $\Z_p$-module for both $i=1,2$,  then from the exact sequence \eqref{split-selmer-03}, it follows that $S_{\mathrm{Gr}}( A_j/\Q_\cyc)^\vee/\pi$ is finite.  Hence by the structure theorem of $\mathcal{O}[[\Gamma]]$-modules, it follows that $S_{\mathrm{Gr}}( A_j/\Q_\cyc)^\vee$ is a finitely generated torsion $\mathcal{O}[[\Gamma]]$-module. Further, as $R(A_{f}(\xi_1\omega_p^{-j}))/\mathbb{Q}_{\cyc}) \subset S_{\mathrm{Gr}}(A_f(\xi_1 \omega_p^{-j})/\Q_\cyc)$,  $\mu(S_{\mathrm{Gr}}(A_f(\xi_1 \omega_p^{-j})/\Q_\cyc)^\vee)=0$ implies $\mu(R(A_f(\xi_1 \omega_p^{-j})/\Q_\cyc)^\vee)=0$ which in turn implies $H^2(\Q_\Sigma/\Q_\cyc, A_f(\xi_i \omega_p^{-j})[\pi])$ vanishes. 
					
					\item[(iv)] In particular, the hypotheses \ref{Stors} and \eqref{H2van} appearing in Theorem~\ref{thm: congrunence algebraic intro} are satisfied whenever 
					$S_{\mathrm{Gr}}( A_f(\xi_i\omega_p^{-j})/\Q_\cyc)^\vee$ is a finitely generated  $\Z_p$-module for  $i=1,2$ or equivalently the $\mu$-invariant of $S_{\mathrm{Gr}}(A_f(\xi_i\omega_p^{-j})/\Q_\cyc)^\vee$ vanishes for $i=1,2$.  In fact, under \ref{irr-f}, following \cite[Conjecture 1.11]{gr}, it is  expected that 
					$\mu(S_{\mathrm{Gr}}( A_f(\xi_i\omega_p^{-j})/\Q_\cyc)^\vee) =0$ for $i=1,2$.  
				\end{enumerate}    
			\end{rem}

			As an application of Corollary~\ref{congruence intro},  we obtain several explicit examples (See \S\ref{section: examples} for details). 
			
			\begin{example}(Example~\ref{example 1})\label{thm: IMC example}
				Let $ p=11 $, $ \Delta \in S_{12}(\mathrm{SL}_2(\mathbb{Z})) $ be the Ramanujan Delta function, $ h \in S_{2}(\Gamma_{0}(23)) $ be the newform of label LMFDB $23.2.a$ and $ \chi_K $ be the quadratic character of $ \mathbb{Q}(\sqrt{-23})$. Put $ f = \Delta \otimes \chi_K  $. Then all the assumptions in Theorems~\ref{thm: congrunence analytic intro}, \ref{thm: congrunence algebraic intro} and Corollary~\ref{congruence intro} are satisfied and we deduce that the  Iwasawa main conjecture holds for $ f \otimes h $ modulo $ \pi $, 
				that is, the following congruences hold in $ \mathcal{O}[[\Gamma]] $:
				\begin{small}
					\begin{equation*}
						(\mu_{p, f \times h, j})  \equiv C_{\mathcal{O}[[\Gamma]]}\Big((S_{\mathrm{Gr}}(A_j)/\mathbb{Q}_{\cyc})^\vee\Big) \mod \pi \quad \mathrm{for} ~ 1 \leq j \leq 10.
					\end{equation*}
				\end{small}%
				Further,   
				$$(\mu_{p, f \times h, j}) = C_{\mathcal{O}[[\Gamma]]}(S_{\mathrm{Gr}}(A_j)/\mathbb{Q}_{\cyc})^{\vee} =  \mathcal{O}[[\Gamma]] \quad \mathrm{for } ~1 \leq j \leq 10 ~\mathrm{ and }~ j \neq 4,5. $$
			\end{example}
			
			\begin{example}(Example~\ref{example 3})\label{thm: IMC example 1}
				Let $ p=5 $. Let $E$ be the elliptic curve given by $y^2+y=x^3+x^2-9x-15$ and $f \in S_{6}(\Gamma_0(19),\iota_{19})$ be the newform congruent to $E$ mod $\pi$. Let $ h \in S_{2}(\Gamma_{0}(11)) $ be the newform of label LMFDB label $11.2.a.a$. Then all the assumptions in Theorems~\ref{thm: congrunence analytic intro}, \ref{thm: congrunence algebraic intro} and Corollary~\ref{congruence intro} are satisfied and we deduce that the  Iwasawa main conjecture holds for $ f \otimes h $ modulo $ \pi $. In fact, in this case, $\mu_{p, f \times h, j}$ is a $p$-adic unit for every $j$ and hence we  deduce the full Iwasawa main conjecture, that is,
				\begin{small}
					\begin{equation*}
						(\mu_{p, f \times h, j}) = C_{\mathcal{O}[[\Gamma]]}\Big((S_{\mathrm{Gr}}(A_j)/\mathbb{Q}_{\cyc})^\vee\Big) = \mathcal{O}[[\Gamma]] \quad \mathrm{for} ~ 1 \leq j \leq 4.
					\end{equation*}
				\end{small}%
			\end{example}

			As stated in Corollary~\ref{congruence intro} the congruence \eqref{intro IMC} is obtained via the congruences  \eqref{intro congruence of p-adic L-functions} and \eqref{intro congruence of char ideals}. We proceed in the analytic side and algebraic  side separately to establish the congruences  \eqref{intro congruence of p-adic L-functions} and \eqref{intro congruence of char ideals} respectively.
			
			On the analytic side, we  start by lifting the characters $ \bar{\xi}_{1}$ and $\bar{\xi}_{2} $ to the Dirichlet characters $ \xi_{1} $ and  $ \xi_{2} $. This enables us to  define a  suitable $ p $-ordinary and primitive Eisenstein series $ g $ such that  the residual Galois representation $\bar{\rho}_g \cong \bar{\xi}_{1} \oplus \bar{\xi}_2 $ and $ g|\iota_{m} \equiv h|\iota_{m} \mod \pi$. As $\tilde{f}:=f|\iota_m$ is not necessarily primitive, there is some technical difficulty in the construction of $p$-adic $L$-function of $f|\iota_m \times h|\iota_m $ (See Remark~\ref{twist and untwist remark} for details).   To circumvent this, we  choose an auxiliary character $\chi$ of $m$-power conductor, such that $f|\chi, g|\bar{\chi}$ and $h|\bar{\chi}$ remain primitive.  Further, we show that the $p$-adic Rankin-Selberg $ L $-function of $ f|\chi \otimes h|\bar{\chi} $ is equal to  the $\Sigma_0$-imprimitive $p$-adic Rankin-Selberg $ L $-function of $ f \otimes h $ up to multiplication by a $p$-adic unit (See Proposition~\ref{analytic final cuspform}). Now the congruence between the Fourier coefficients of $g$ and $h$ implies the $p$-adic Rankin-Selberg $ L $-functions of $ f|\chi \otimes h|\bar{\chi} $ and $ f|\chi \otimes g|\bar{\chi} $ are congruent.    Thus it suffices to show that 
			the congruence \eqref{intro congruence of p-adic L-functions} holds with the $\Sigma_0$-imprimitive $ p $-adic Rankin-Selberg $ L $-function of $ f \otimes h $ replaced by the $ p $-adic Rankin-Selberg $ L $-function of $ f|\chi \otimes g|\bar{\chi} $. 
			
			To establish this analogue of congruence  \eqref{intro congruence of p-adic L-functions} for the $ p $-adic $ L $-function of $ f|\chi \otimes g|\bar{\chi} $, we suitably adopt the strategy of \cite{Vatsal} from the setting of elliptic modular forms to our Rankin-Selberg product setting.  This can be explained in two steps. First, we express the special values of the Rankin-Selberg $ L $-function of $ f|\chi \otimes g|\bar{\chi}  $ as a product of  the special values of $ L $-functions of $ \tilde{f}_0 \otimes \xi_1 $ and $ \tilde{f}_0 \otimes \xi_2 $ (see Lemma~\ref{twist and untwist Eis}). 
			From this to arrive at the required  congruence of  the $ p $-adic $ L $-functions involves a delicate choice of periods $ \Omega^{\pm}_{\tilde{f}} $ such that (i) The  $p$-adic $ L $-functions $ \mu^{\Sigma_{0}}_{p,f, \xi_{i},j} $ attached to $ \tilde{f} \otimes \xi_{i} $ are   $ \mathcal{O} $-valued for $ i=1,2 $ and $ 0 \leq j \leq k-2 $ and (ii) The $ \mathcal{O} $-valued  $ p $-adic $ L $-function  $\mu^{\Sigma_{0}}_{p, f \times g, j}$ differs from the product $ \mu^{\Sigma_{0}}_{p,f, \xi_{1},j} \times \mu^{\Sigma_{0}}_{p,f, \xi_{2},j} $ by a $ p $-adic unit. In fact, to achieve (i),  we choose periods using certain parabolic cohomology groups and then show that with this choice of period, the required $ p $-adic $ L $-functions becomes $ \mathcal{O} $-valued (see Theorems~\ref{choice of period and petterson innerproduct}, \ref{period and integral measure}). On the other hand, we use results from \cite[Chapter 5]{Hida3} and \cite{Hidaadjoint} on the adjoint $ L $-function of a modular form  to show that the condition (ii) is satisfied. Building up on these results, we  obtain the congruence of $ p $-adic $ L $-functions in equation \eqref{intro congruence of p-adic L-functions} in  Theorems~\ref{analytic final1}, \ref{analytic final}.  

			On the algebraic side, from \ref{irr-f} and $f$ is $p$-ordinary, it follows that $S_{\mathrm{Gr}}^{\Sigma_{0}}(A_f(\xi_{i}\omega_p^{-j})/\Q_{\cyc})^\vee $ is a torsion $ \mathcal{O}[[\Gamma]] $-module for $ i=1,2 $  (\cite{kato}). 
			The key ingredient  in the congruence \eqref{intro congruence of char ideals} is the following exact sequence, which is proved in Lemma~\ref{lem: dual selmer group sequence}:
			\begin{small}
				\begin{align}\label{intro exact sequence of dual selmer}
					0 \rightarrow S_{\mathrm{Gr}}^{\Sigma_{0}}(A_f(\bar{\xi}_1 \omega_{p}^{-j})[\pi]/\Q_{\cyc})^{\vee}  \rightarrow S_{\mathrm{Gr}}^{\Sigma_{0}}(A_j[\pi]/\Q_{\cyc})^{\vee} \rightarrow  S_{\mathrm{Gr}}^{\Sigma_{0}}(A_f(\bar{\xi}_2 \omega_{p}^{-j})[\pi]/\Q_{\cyc})^\vee \rightarrow 0.
				\end{align}
			\end{small}%
			Under \ref{irr-f}, we show that $ S_{\mathrm{Gr}}^{\Sigma_{0}}(B_j[\pi]/\Q_{\cyc}) = S_{\mathrm{Gr}}^{\Sigma_{0}}(B_j/\Q_{\cyc})[\pi] $ for $ B_j \in \{ A_f(\xi_1 \omega_{p}^{-j}), A_j, A_f(\xi_2 \omega_{p}^{-j})  \}$. 
			Further, for $ B_j \in \{ A_f(\xi_1 \omega_{p}^{-j}), A_j, A_f(\xi_2 \omega_{p}^{-j})\}$, we show that  $S_{\mathrm{Gr}}^{\Sigma_{0}}(B_j/\Q_{\cyc})^\vee $   has no non-zero pseudo-null submodule using a result of \cite{we}.  The last fact enables us to show that the base change holds for the characteristic ideals i.e. $C_{\mathcal{O}[[\Gamma]]}(S_{\mathrm{Gr}}^{\Sigma_{0}}(B_j/\Q_{\cyc})^{\vee})/ \pi \mathcal{O}[[\Gamma]] = C_{\mathbb{F}[[\Gamma]]}\big((S_{\mathrm{Gr}}^{\Sigma_{0}}(B_j/\Q_{\cyc})[\pi]\big)^{\vee})$ with $\mathbb{F} = \mathcal{O}/\pi$.  
			Putting all these together,  the congruence in  \eqref{intro congruence of char ideals} follows (See Theorem~\ref{thm:congruence of ideals}). 
			We now describe the key ideas involved to derive the exact sequence  \eqref{intro exact sequence of dual selmer}: Tensoring the exact sequence  \eqref{intro rho_h} with $ A_f \otimes \omega_{p}^{-j} $, we obtain 
			\begin{small}
				\begin{align*}
					0 \rightarrow A_f(\bar{\xi}_1 \omega_{p}^{-j})[\pi] \rightarrow A_j[\pi] \rightarrow  A_f(\bar{\xi}_2 \omega_{p}^{-j})[\pi] \rightarrow 0. 
				\end{align*}   
			\end{small}%
			A priori, the Selmer group $S^{\Sigma_{0}}_{\mathrm{Gr}}(B_j[\pi]/\Q_\cyc)$ depends on $B_j$.  Under hypotheses (i) and (ii) of Theorem~\ref{thm: congrunence algebraic intro}, we show that the local conditions appearing in the Selmer group $  S^{\Sigma_{0}}_{\mathrm{Gr}}(B_j[\pi]/\Q_\cyc)$ depend 
			only on $B_j[\pi]$. With this explicit description,  $  S^{\Sigma_{0}}_{\mathrm{Gr}}(B_j[\pi]/\Q_\cyc)$ is  essentially determined by the corresponding residual representation. Using  this  and the 
			hypotheses (iii), (iv) of  Theorem~\ref{thm: congrunence algebraic intro}, we deduce \eqref{intro exact sequence of dual selmer}. 

			The structure of article is as follows: In \S\ref{sec: Prelims and setup}, we recall the  basics of $p$-adic modular forms and  the $ p $-adic Rankin-Selberg $ L $-function due to Hida \cite{Hidarankin2}. In \S\ref{section: simplyfying rankin},  we define the Eisenstein series $ g $ and show that the $p$-adic $ L $-functions of $ f|\chi \otimes g|\bar{\chi} $ and  $ f|\chi \otimes h|\bar{\chi} $ are  congruent modulo $ \pi $. We also express the special values of the Rankin-Selberg $ L $-function of $ f|\chi \otimes g|\bar{\chi} $ as a product of the special values of the $ L $-functions attached to $ \tilde{f} \otimes \xi_1$ and $ \tilde{f} \otimes \xi_2$. Next in \S\ref{sec: periods and congruence}, we make an appropriate choice of periods $ \Omega_{\tilde{f}}^{\pm} $ and use it to obtain that the ideal generated by the $ p $-adic $ L $-function of $ f|\chi \otimes h|\bar{\chi}  $ in the Iwasawa algebra is congruent to the ideal generated by the product of the $ p $-adic $ L $-functions associated to $ \tilde{f} \otimes \xi_1 $ and $  \tilde{f} \otimes \xi_2 $ (Theorem~\ref{analytic final}). We also show that the ideal generated by the $p$-adic $ L $-function of $ f|\chi \otimes h|\bar{\chi} $ is equal to the $\Sigma_0$-imprimitive $p$-adic $L$-function of $f \otimes h$ in the Iwasawa algebra. 
			In \S\ref{sec: Selmer groups}, we recall the $ p^\infty $-Selmer groups of the Rankin-Selberg convolution $ f \otimes h $ and the modular form $ f \otimes\xi_{i} $ and further give explicit descriptions of these Selmer groups in terms of residual representations. In \S\ref{sec: Congruences of char ideals}, we show that the characteristic ideal of the $p^\infty $-Selmer group of $ f \otimes h $ is congruent to the product of the characteristic ideals attached to the  $p^\infty $-Selmer groups of  $ f \otimes \xi_1 $ and $ f \otimes \xi_2$.  In \S\ref{sec: IMC}, we prove our main theorem (Theorem~\ref{congruence main conjecture }) establishing the Iwasawa main conjecture modulo $ \pi $ for $ f \otimes h$. We  discuss a few concrete examples illustrating the result of Theorem~\ref{congruence main conjecture } in \S\ref{section: examples}. 
			
			\thanks{{\bf Acknowledgement:} S. Jha acknowledges the support of SERB MTR/2019/000996 grant.  A part of this work was done at ICTS-TIFR and we acknowledge the program ICTS/ecl2022/8. We thank Aribam Chandrakant for the help with the computations on $ \mu $-invariants. R. Vangala  acknowledges the support of IPDF  of IIT Kanpur.} 

		\section{Preliminaries and Setup}\label{sec: Prelims and setup}
		In this section, we begin by recalling $p$-adic modular forms, measures, Rankin-Selberg $ L $-function  and $ p $-adic Rankin-Selberg $ L $-function.  Throughout this section $ J $ denotes an arbitrary positive integer.
		\subsection{\texorpdfstring{$p$}{}-adic modular forms}
		In this subsection, we briefly recall the definition of the space of $p$-adic modular forms in the sense of Serre.
		For more details we refer the reader to \cite{Hidarankin1}, \cite{Hidarankin2}. We  fix embeddings   $ i_{p}: \bar{\mathbb{Q}} \rightarrow 
		\bar{\mathbb{Q}}_{p} $ and $ i_{\infty}:  \bar{\Q} \hookrightarrow \C $.  
		Let $ M_{k}  (\Gamma_{0}(J), \psi) $ 
		(resp. $ M_{k}  (\Gamma_{1}(J)) $) is the space of modular 
		forms with coefficients in $ \mathbb{C} $, nebentypus  $ \psi $ and the congruence subgroup $\Gamma_{0}(J)$ (resp. $ \Gamma_{1}(J) $).
		For a subring  $ R \subset \bar{\mathbb{Q}} $, let 
		$M_{k}(\Gamma_{0}(J), \psi ; R)$ (resp. 
		$M_{k}(\Gamma_{1}(J); R) $) denote the subspace of $ M_{k}  (\Gamma_{0}(J), \psi) $ 
		(resp. $ M_{k}  (\Gamma_{1}(J)) $)  consisting of all modular forms with Fourier coefficients in $ R $.
		Also, let  $ S_{k}  (\Gamma_{0}(J), \psi) $ 
		(resp. $S_{k}  (\Gamma_{1}(J)) $) be the space of cusp forms with coefficients in $ \mathbb{C} $, nebentypus  $ \psi $ and the congruence subgroup $ \Gamma_{0}(J) $ (resp. $ \Gamma_{1}(J) $).  Put
		\begin{small}{
				\begin{align*}
					S_{k}(\Gamma_{0}(J), \psi ; R) :=
					M_{k}(\Gamma_{0}(J), \psi ; R)  
					\cap S_{k}  (\Gamma_{0}(J), \psi),  \quad 
					S_{k}(\Gamma_{1}(J); R) := M_{k}(\Gamma_{1}(J); R) \cap
					S_{k}  (\Gamma_{1}(J)). 
			\end{align*}}
		\end{small}%
		For every modular form $ f(z) = \sum_{n\geq 0}^{} a(n,f) q^n $ with  Fourier coefficients in $\bar{\mathbb{Q}}$, we define a $p$-adic norm $ |f|_{p} $ by  $|f|_{p}  : = \sup_{n} |a(n,f)|_{p}$.
		For  a finite extension $ K$ of $ \mathbb{Q} $, let 
		$K_{\mathfrak{p}}$ be the closure of $ K $ in $ \mathbb{C}_{p} $ induced by the embedding $\iota_p: \bar{\Q} \rightarrow \bar{\Q}_p$, where $ \mathbb{C}_{p}  $ is the completion of $ \bar{\mathbb{Q}}_{p} $  and $ \mathcal{O} $ be the ring of integers of $ K_{\mathfrak{p}} $. Let $ M_{k}(\Gamma_{0}(J), \psi ; K_{\mathfrak{p}}) $
		(resp. $M_{k}(\Gamma_{1}(J) ; K_{\mathfrak{p}})$) denote the completion of
		$ M_{k}(\Gamma_{0}(J), \psi ; K) $
		(resp. $M_{k}(\Gamma_{1}(J) ; K)$) 
		with respect to the norm $ |\cdot|_{p} $ in $ K_{\mathfrak{p}}[[q]] $.
		Then it is known that (see \cite[Pg 170]{Hidarankin1})
		\begin{small}
			\begin{align*}
				M_{k}(\Gamma_{0}(J), \psi ; K_{\mathfrak{p}})  = 
				M_{k}(\Gamma_{0}(J), \psi ; K) \otimes_{K} K_{\mathfrak{p}}, \quad
				M_{k}(\Gamma_{1}(J) ; K) = M_{k}(\Gamma_{1}(J) ; K)
				\otimes_{K} K_{\mathfrak{p}}.
			\end{align*}
		\end{small}%
		Put  $M_{k}(\Gamma_{0}(J), \psi ; \mathcal{O})  :=  \mathcal{O}[[q]]  \cap  M_{k}(\Gamma_{0}(J), \psi ; K_{\mathfrak{p}})$  and  $M_{k}(\Gamma_{1}(J) ; \mathcal{O}) :=  \mathcal{O}[[q]]   \cap M_{k}(\Gamma_{1}(J) ; K_{\mathfrak{p}})$. For $ A \in \{ K_{\mathfrak{p}}, \mathcal{O} \}$,  set
		\begin{small}
			\begin{align*}
				M_{k}(J;A) = \cup_{n=0}^{\infty} 
				M_{k}(\Gamma_{1}(Jp^n);A)  \quad  \mathrm{and} \quad
				M_{k}(J, \psi ;A) = \cup_{n=0}^{\infty} 
				M_{k}(\Gamma_{0}(Jp^n), \psi;A). 
			\end{align*}
		\end{small}%
		Let $\overline{M}_{k}(J,\psi;A)$ (resp. $\overline{M}_{k}(J;A)$) denote the completion of $M_{k}(J,\psi;A)$ (resp. $M_{k}(J;A)$) under the norm $| \cdot |_{p}$ in $K_{\mathfrak{p}}[[q]]$. Any element of the space $\overline{M}_{k}(J;A) $ will be called a $p$-adic modular form. Similarly, one can define the spaces $ \overline{S}_{k}(J, \psi ;A) $ and $\overline{S}_{k}(J;A)$).

		Next, we consider the Hecke algebras of the space of $p$-adic modular forms. Denote by  $ H_{k}(\Gamma_{0}(J), \psi ; R) $
		(resp. $H_{k}(\Gamma_{1}(J) ; R)$), the Hecke algebra corresponding to the space of modular forms $ M_{k}(\Gamma_{0}(J), \psi;R) $  (resp. $ M_{k}  (\Gamma_{1}(J);R) $). Let $ h_{k}(\Gamma_{0}(J), \psi ; R) $ (resp. $h_{k}(\Gamma_{1}(J) ; R)$) denote the Hecke algebra corresponding to the space of cusp forms $ S_{k}  (\Gamma_{0}(J), \psi;R) $ (resp. $ S_{k}  (\Gamma_{1}(J);R) $). Set
		\begin{small}
			\begin{alignat*}{2}
				H_{k}(J, \psi; \mathcal{O} ) &= \varprojlim_{n} H_{k}(\Gamma_{0}(Jp^n), \psi; \mathcal{O} )
				\quad \text{ and } \quad
				H_{k}(J; \mathcal{O} ) &&= \varprojlim_{n} H_{k}(\Gamma_{1}(Jp^n);\mathcal{O} ), \\
				h_{k}(J, \psi; \mathcal{O} ) &= \varprojlim_{n} h_{k}(\Gamma_{0}(Jp^n), \psi; \mathcal{O} )
				\quad \text{ and } \quad
				h_{k}(J; \mathcal{O} ) &&= \varprojlim_{n} h_{k}(\Gamma_{1}(Jp^n);\mathcal{O} ). 
			\end{alignat*}
		\end{small}%
		\begin{definition}\label{idempotent defn.} 
			\textbf{(Idempotent)}(\cite[\S 2]{Hidarankin2})
			Let \begin{small}$ U(p) $\end{small} be the $ p^{\text{th}} $-Hecke operator in \begin{small}$H_{k}(\Gamma_{0}(Jp^n), \psi; \mathcal{O})$\end{small} and \begin{small}$H_{k}(\Gamma_{1}(Jp^n);\mathcal{O} )$\end{small}. For every $ n $, let $ e_n $ be the idempotent in
			\begin{small}$H_{k}(\Gamma_{0}(Jp^n), \psi; \mathcal{O})$\end{small} and  \begin{small}$H_{k}(\Gamma_{1}(Jp^n);\mathcal{O})$\end{small} defined by  \begin{small}$ \lim\limits_{m \rightarrow \infty} U(p)^{m!} $\end{small}. The idempotent operator $e$ in \begin{small}$ H_{k}(J, \psi; \mathcal{O}) $\end{small} and \begin{small}$ H_{k}(J; \mathcal{O}) $\end{small} is defined as $ \varprojlim_{n} e_{n} $.
		\end{definition}	
		\subsection{Measures  and \texorpdfstring{$p$}{}-adic Rankin-Selberg Convolution}
		In this subsection, we briefly recall the construction of $ p $-adic Rankin-Selberg $L$-function due to Hida \cite{Hidarankin2}. For a topological ring $ A $, let $ C(\mathbb{Z}_{p}^{\times}; A  ) $ and $ LC(\mathbb{Z}_{p}^{\times}; A ) $  denote the space of continuous (resp. locally constant) functions on  $ \mathbb{Z}_{p}^{\times} $ with values in  $ A $. Let $ g(z) = \sum_{n=0}^{\infty} a(n,g) q^n \in M_{l}(\Gamma_{0}(J),\psi;\mathcal{O})$. Then we consider the arithmetic measure (see \cite[Page 36, Example b]{Hidarankin2}) $ \mu_{g} $ of weight $l$ (see \cite[(5.1a)]{Hidarankin2}), defined by
		\begin{small}
			\begin{align}\label{Def measure mu star}
				\mu_{g}(\phi)= \sum_{n =1}^{\infty} \phi(n) a(g, n)q^n, \quad \forall ~ \phi \in  C(\mathbb{Z}_{p}^{\times};\mathcal{O}),
			\end{align}
		\end{small}%
		where  $ \phi  \equiv 0 $  on $ \mathbb{Z}_{p} \setminus \mathbb{Z}_{p}^\times$. Note that $\mu_g(\phi) \in S_l(\Gamma_0(J);\mathcal{O})$ for $\phi \in LC(\mathbb{Z}_{p}^{\times}; \mathcal{O})$. Let $ J_{0} $ be the prime to $ p $-part of $ J $.
		Let $L$ be a positive integer such that $J_{0} \mid L$.  Then we have a modified arithmetic measure defined as $ \mu_{g}^{L}(\phi) := \mu_{g}(\phi)\vert [L/J_{0}]$, where $[L/J_{0}] : \bar{S}_{l}(J_{0};\mathcal{O}) \rightarrow \bar{S}_{l}(L;\mathcal{O})$ is the linear map defined by  $(f \vert [L/J_{0}])(z) = f(zL/J_{0}), ~ \forall ~ f \in \bar{S}_{l}(N_{0};\mathcal{O})$. Put $Z_{L} = \mathbb{Z}_{p}^{\times} \times (\mathbb{Z}/L\mathbb{Z})^{\times}$. For $z \in Z_{L} $ we denote its component in $ \mathbb{Z}_{p}^{\times} $ by $z_{p}$. We consider the action (depending on the weight and the nebentypus of $g$) of the group $Z_L$ on $C(\mathbb{Z}_{p}^{\times}; \mathcal{O})$ by the formula $(z \ast \phi)(x) := \psi(z) z_{p}^{l}\phi(z_{p}^2 x)$ for $z \in Z_{L}$ and $\phi \in C(\mathbb{Z}_{p}^{\times}; \mathcal{O})$. We also consider the arithmetic measure of weight one defined by
		\begin{small}
			\begin{align*}
				2 E(\phi ) = \sum_{\substack{n=1\\ (n,p)=1}}^{\infty}
				\sum_{\substack{d \mid n \\ (d,L)=1}} \mathrm{sgn}(d) \phi(d) q^{n}
				\in \mathbb{Z}_{p} [[q]].
			\end{align*}
		\end{small}%
		For a given integer $k>l$ and a finite order character $\eta: Z_{L} \rightarrow  \mathbb{Z}_{p}^{\times}$, we consider the arithmetic measure $(\mu_{g}^{L} \star E)_{\eta,k} : C(\mathbb{Z}_{p}^{\times};\mathcal{O})\rightarrow \bar{S}_{k}(L;\mathcal{O})$ of weight $k$ and character $ \eta $ defined by convolution of $ \mu_{g}^{L}$ and $ E $ as follows:
		\begin{small}
			\begin{align*}
				(\mu_{g}^{L}\star E)_{\eta,k}(\phi) := \int_{\mathbb{Z}_{p}^{\times} }\int_{Z_{L}} \eta(z) z_{p}^{k-1} (z^{-1} \ast \phi)(x) dE(z) d\mu_{g}^{L}(x).
			\end{align*}   
		\end{small}%
		By \cite[(9.3)]{Hidarankin2} (see also \cite[Section 2]{Bouganis}), for a finite order character $ \phi \in C(Z_{L}; \mathcal{O})  $,
		we have 
		\begin{small}
			\begin{equation}\label{convolution mu and Eisenstien measure}
				\begin{aligned}	
					(\mu_{g}^{L} \star E)_{\eta,k} (x_p^{j}\phi)  & =  \int_{\mathbb{Z}_{p}^{\times} } \int_{Z_{L}}
					\eta(z) z_{p}^{k-1} (z^{-1} \ast x_p^{j}\phi)(x) dE(z) d\mu_{g}^{L}(x) \\
					& = \int_{\mathbb{Z}_{p}^{\times} } \int_{Z_{L}}
					\eta(z) z_{p}^{k-1} \psi(z)^{-1}  z_{p}^{-l} z_{p}^{-2j} \phi(z_{p}^{-2})
					\phi( x) dE(z) d\mu_{g}^{L}(x) \\
					& = \mu_{g}^{L}(x_{p}^{j}\phi)  \cdot E(\eta \cdot \psi^{-1} \cdot  (\phi_{p}^{-2}) \mathfrak{Z}_{p}^{k-l-2j-1}),
				\end{aligned} 	
			\end{equation}		
		\end{small}%
		where $ \mathfrak{Z}_{p}(z)=z_{p} $ and $ \phi_{p}(z) = \phi(\mathfrak{Z}_{p}(z)) $.  
		
		Let $ f(z) = \sum_{n=1}^{\infty} a(n,f) q^n$ be a  normalized 	Hecke eigenform of weight  $ k > l$, level $ N_f $ and  nebentypus $ \eta $.  We assume  $ f $ is $p$-ordinary  i.e. $ | i_p(a(n,f)) |_{p} =1 $.  We define  the $ p $-stabilization of $ f $ by  
		\begin{small}
			\begin{align}\label{p-stabilization of f}
				f_{0}(z) = 
				\begin{cases}
					f(z) &\text{ if } p \mid N_{f},\\
					f(z) - \beta_{f} f(pz) & \text{ otherwise},
				\end{cases}
			\end{align}
		\end{small}%
		where $ \beta_{f} $ is the unique root of $ X^2 - a(p,f) X + \eta(p) p^{k-1} = 0 $ with $ |\beta_{f} |_{p} < 1 $. It is well-known  that the level of $f_0$, $N_{f_{0}} = N_{f} $ if $ p \mid N_{f} $ and $ N_{f_{0}} = pN_{f} $ if $ p \nmid N_{f} $ and nebentypus of $ f_{0} $ is $ \eta $. Note that $u_f:=a(p,f_0)$ is the unique $p$-adic unit root of the Hecke polynomial of $f$ at $p$.  For simplicity, by a slight abuse of notation, we  will denote $ N_f $ by $N$. We assume that  the Fourier coefficients of $f$ and the values of $ \eta $ lie in  $ \mathcal{O}$. Hence we obtain a surjective homomorphism $\phi_{f}: h_{k}(\Gamma_{0}(Np), \eta; \mathcal{O}) \rightarrow \mathcal{O}$ induced by $ T(n) \rightarrow a(n,f_{0}) $. Assume that the map $ \phi_{f} $ induces the decomposition 
		\begin{align}\label{splitting of hecke algebra}
			h_{k}(\Gamma_{0}(Np), \eta ; K_{\mathfrak{p}}) = K_{\mathfrak{p}} \oplus A.
		\end{align}
		It is known that the above splitting holds if $f$ is primitive. We will need the decomposition in  \eqref{splitting of hecke algebra} for a  twist of certain specific modular form and it will be made explicit in \S\ref{section: simplyfying rankin}. Let $ 1_{f_0} $ be the  idempotent attached to the first summand. 
		We fix a constant $ c(f_0) \in \mathcal{O}$ such that 
		$ c(f_0) 1_{f_0} \in 
		h_{k}(\Gamma_{0}(Np), \eta ;  \mathcal{O})$.  The idempotent $ 1_{f_0} $ induces a map 
		\begin{small}
			\[ 
			l_{f_{0}} :   e \bar{S}_{k}(N, \eta ; K_{\mathfrak{p}})
			\rightarrow K_{\mathfrak{p}}
			\]
		\end{small}%
		defined by $l_{f_0}(ex) = a(1, x \vert e \vert 1_{f_0})  $ where $e \in H_k(\Gamma_{0}(N), \eta ;\mathcal{O})$ is the idempotent operator defined in 
		Definition~\ref{idempotent defn.} and $ x \in 
		\bar{S}_{k}(N, \eta ; K_{\mathfrak{p}}) $.  It follows from 
		\cite[Proposition 4.1]{Hidarankin1} that $ e \bar{S}_{k}(N, \eta ; K_{\mathfrak{p}}) \subset S_{k}(\Gamma_{0}(Np), \eta ; K_{\mathfrak{p}}) $ and the map $ l_{f_0} $ 
		is well-defined. Let $L$ be the least common multiple of
		$ J_{0} $ and $ N $ where $ J_{0} $ denotes prime to $ p $ part of $ J $. Let $ \omega_p:  (\mathbb{Z}/{p}\mathbb{Z})^{\times}  \rightarrow \mathbb{Z}_{p}^{\times}$ denote the Teichm\"uller character  and $ \chi $ be a Dirichlet character of $ (\mathbb{Z}/Np\mathbb{Z})^{\times} $. For an integer $ r \geq 1 $ and a finite order character $ \epsilon : (1+p\mathbb{Z}_{p})  \rightarrow
		\mathcal{O}^{\times} $ of conductor  $p^r$, let
		$Tr_{L/N}: M_{k}(\Gamma_{0}(Lp^r), \epsilon \chi ; \mathcal{O}) \rightarrow  M_{k}(\Gamma_{0}(Np^r), \epsilon \chi ; \mathcal{O})
		$
		denote the trace operator considered in \cite[VI.(1.7)]{Hidarankin2}. For $ g \in  M_{l}(\Gamma_{0}(J), \psi)$, put  $ g^{\rho}(z) =  \sum_{n=0}^{\infty} \overline{a(n,g)} q^n  \in M_{l}(\Gamma_{0}(J), \psi^{-1})$. Take  $ \epsilon \chi = \eta $ and define the measure  $ \mu_{f \times g} $ as follows: 
		\begin{align}\label{Def of convolution}
			\mu_{f \times g}(\phi) :=  
			c(f_0)  \circ l_{f_0} \circ Tr_{L/N} \circ 
			e((\mu_{g^{\rho}}^{L} \star E)_{\eta,k}(\phi^{-1})), ~~ \forall \text{ finite order character }\phi \in C(\mathbb{Z}_{p}^{\times},\mathcal{O}), 
		\end{align}
		where  $\circ $ denotes the composition of maps. By the choice of $c(f_0)$, it follows that $\mu_{f \times g}$ is an $\mathcal{O}$-valued measure i.e. $\mu_{f \times g}(\phi) \in \mathcal{O}$, for all $\phi \in LC(\mathbb{Z}_p^\times, \mathcal{O})$.
		
		The Rankin-Selberg $ L $-function of $ f \otimes g $ is defined by 
		\begin{small}
			\begin{align*}
				D_{JN}(s,f,g) := L_{JN}(2s+2-k-l, \psi \eta)	
				L(s,f,g) :=  L_{JN}(2s+2-k-l, \psi \eta)	
				\sum_{n=1}^{\infty} a(n,f) a(n,g) n^{-s},
			\end{align*}
		\end{small}%
		where $L_{JN}(2s+2-k-l, \psi \eta)$ denotes the 
		Dirichlet $L$-function of $ \psi \eta $ with the Euler factors at the primes dividing $JN$ omitted from its Euler product (see \cite[\S 1]{Hidarankin1}). For $ f, h \in S_k(N,\chi) $, define the Petersson inner product of
		$f$ and $h$  by \begin{small}$   \langle f,h \rangle_N := \int_{\mathbb{H}/\Gamma_{0}(N)} f(z)\overline{h(z)} y^{k-2} ~ dz$,\end{small}
		where $ y = \text{Im}(z) $. 
		
		For integers $N$ and $J$, let $[N,J]$ denote the least common multiple of $N$ and $J$. With the notation as above we recall the following theorem:
		\begin{theorem}\label{padic rankin}$($\cite[Theorem 5.1 and Section 8]{Hidarankin2}, \cite[Theorem 2.9]{Bouganis}$)$ Let $ f \in S_{k}(\Gamma_{0}(N), \eta) $ be a normalised $ p $-ordinary eigenform with $ p \nmid N $ and $ g \in M_{l}(\Gamma_{0}(Jp^{\alpha}),\psi) $ with $ k>l $ and $ (J,p) =1 $.	
			Assume that \eqref{splitting of hecke algebra} holds.
			For every finite order character   $ \phi \in C(\mathbb{Z}_{p}^{\times} ,
			\bar{\mathbb{Q}}_p) $ and $ 0 \leq j \leq k-l-1 $, the $\mathcal{O}$-valued measure $\mu_{f \times g}$ satisfies the following interpolation property:
			\begin{small}
				\begin{align*}
					\mu_{f \times g}(x_p^{j} \phi ) = \int_{\mathbb{Z}_{p}^{\times} } x_p^{j} \phi ~ d\mu_{f \times g}
					=  c(f_0) t   p^{\beta l/2} p^{\beta j} p^{(2-k)/2}
					a(p,f_0)^{1-\beta}  \frac{D_{JNp}(l+j,f_{0}, 
						\mu_{g^{\rho}}(\phi^{-1})|_{l}\tau_{Jp^{\beta}})}{(2i)^{k+l+2j} \pi^{l+2j+1} 
						{ \langle f_{0}^{\rho}|_{k}\tau_{Np} , f_{0} \rangle_{Np}}},       		
				\end{align*}
			\end{small}%
			where $\beta $ is the smallest positive integer such that $ \mu_{g^{\rho}}(\phi^{-1})|\tau_{Jp^{\beta}}
			\in M_{l}(\Gamma_{1}(Jp^{\beta})),  \tau_{Jp^{\beta}} = 
			\begin{psmallmatrix}  0 & -1 \\ Jp^{\beta} & 0 \end{psmallmatrix} $
			and $ t = [N,J] N^{k/2} J^{(l+2j)/2} \Gamma(l+j)\Gamma(j+1)  $.
		\end{theorem}
		\subsection{Simplifying the \texorpdfstring{$L$}{}-function}

		In this subsection, we express the Rankin-Selberg $L$-function  $D_{JNp}(l+j,f_{0}, 
		\mu_{g^{\rho}}(\phi^{-1})|_{l}\tau_{Jp^{\beta}})$ appearing in Theorem~\ref{padic rankin} in terms of $D_{JNp}(l+j,f_0,\mu_{g}(\phi))$. 
		
		For $g \in M_{l}(Jp^\alpha,\psi)$ with $p\nmid J$, we  denote the characteristic polynomial of $ g $ at $ p $ by $ P_{p}(g,T) $. More precisely, let 
		\begin{small}
			\begin{align}\label{Euler factor at p}
				P_{p}(g,T)  = \begin{cases}
					\psi(p) p^{l-1} T^2 - a(p,g) T + 1 & \text{ if } \alpha \neq 0, \\
					1 - a(p,g) T  & \text{ if } \alpha =0.
				\end{cases}
			\end{align}
		\end{small}%
		Note that $P_p(g,p^{-s})$ gives the Euler factor at $p$ appearing in the $L$-function $L(s,g)$.  For a newform $g \in S_{l}(Jp^\alpha,\psi)$, let $W(g)$ be the root number of $g$ i.e. $g|\tau_{Jp^\alpha} = W(g) g^{\rho}$ (cf. \cite[Page 37]{Hidarankin2} and \cite[Theorem 4.6.15]{Miyake}). We first treat the case $\phi$ is the trivial character of $\mathbb{Z}_p^\times$ and $g$ is a cusp form.  
		
		\begin{lemma}\label{special value f,cusp and trivial character}
			Let $f \in S_{k}(\Gamma_{0}(N),\eta)$ be a normalised $p$-ordinary eigenform with $p\nmid N$. Let $g \in S_{l}(\Gamma_{0}(J p^\alpha),\psi)$ be a $p$-ordinary newform with $p\nmid J$ and $2\leq l<k$. Let $\beta  $ be as defined  in Theorem~\ref{padic rankin}. For $ 0 \leq j \leq k-l-1 $, we have 	$\beta = \max\{\alpha,1\} +1 $ and
			\begin{small}
				\begin{align*}
					p^{\beta(l+2j)/2} D_{JNp}(l+j,f_0, 
					\mu_{g^\rho}(\iota_{p})
					|\tau_{Jp^{\beta}})  =   p^{\alpha(l+2j)/2} u_{f}^{\beta-\alpha} W(g^\rho)   P_{p}(g,p^j u_{f}^{-1}) 
					D_{JNp}(l+j,f_0,g).
				\end{align*} 
			\end{small} %
		\end{lemma} 	
		\begin{proof}
			Since $ g^\rho \in S_{l}(J p^\alpha,\bar{\psi})$ is a Hecke eigenform,  we have  (see \cite[Line -1, Page 46]{Hidarankin2})
			\begin{align}\label{iota fourier cusp}
				g^\rho|\iota_{p} = 
				\begin{cases}
					g^\rho - a(p,g^\rho) g^\rho|[p]+\bar{\psi}(p)p^{l-1} g^\rho|[p^2] 
					& \text{ if } \alpha \neq 0, \\
					g^\rho - a(p,g^\rho) g^\rho|[p] & \text{ if } \alpha = 0.
				\end{cases}
			\end{align}
			As $g$ is primitive, for every $ i \geq 0 $,  we have
			\begin{small}
				\begin{equation}\label{Atkin involution and d}
					\begin{aligned}
						g^\rho|[p^i]|\tau_{Jp^\beta} & = p^{-il/2} 
						g^\rho|\begin{psmallmatrix}
							p^i & 0 \\ 0 &1 \end{psmallmatrix} 
						\begin{psmallmatrix}
							0 & -1 \\ J p^{\beta} &0 \end{psmallmatrix} \\
						& =p^{-il/2} 
						p^{(\beta-\alpha-i)l/2}
						(g^{\rho}|\tau_{Jp^\alpha})| [p^{\beta-\alpha-i}]  \\
						& = W(g^\rho) p^{(\beta-\alpha-2i)l/2}
						g|[p^{\beta-\alpha-i}]. 
					\end{aligned}
				\end{equation}
			\end{small}%
			Since $a(p,g) \neq 0$ and $g$ is primitive,  we have $g$ is  not super-cuspidal at $p$ and $g$ is $p$-minimal (cf. remark after \cite[Lemma 10.1]{H6}).
			
			\noindent \underline{Case $ \alpha = 0$}:    By \cite[Lemma 5.2(i)]{Hidarankin2}, we have $ \beta =2  = \max\{\alpha,1\}+1$. Therefore by \eqref{iota fourier cusp} and \eqref{Atkin involution and d}, we have
			\begin{small}
				\begin{align*}
					& D_{JNp}(l+j,f_0, 
					\mu_{g^\rho}(\iota_{p})
					|\tau_{Jp^{\beta}})  
					\\ 
					& = D_{JNp}(l+j,f_0, 
					g^\rho
					|\tau_{Jp^{\beta}})  - a(p,g^\rho) D_{JNp}(l+j,f_0, g^\rho|[p]|\tau_{J p^{\beta}})+\bar{\psi}(p) p^{l-1} D_{JNp}(l+j,f_0,g^\rho)|[p^2]|\tau_{J p^{\beta}}) 
					\\ 	      
					& =W(g^\rho) \Big(p^l D_{JNp} (l+j,f_0, g|[p^{\beta}]) - a(p,g^\rho) D_{JNp} (l+j,f_0, (g)|[p^{\beta-1}])+ \bar{\psi}(p)
					p^{-1} D_{JNp} (l+j,f_0, (g)|[p^{\beta-2}]) \Big).  	
				\end{align*}
			\end{small}%
			Since  $ L(l+j,f_{0}, g|[p^i]) = p^{-i(l+j)}  u_{f}^{i} L(l+j,f_{0}, g)$ and $\beta =2$, we get
			\begin{small}
				\begin{align}\label{cusp atkin and trivial char}
					\begin{split}
						D_{JNp}(l+j,f_{0}, \mu_{g^{\rho}}(\iota_{p}) |\tau_{Jp^{\beta}})  
						& =  W(g^{\rho})  D_{JNp}(l+j,f_{0},g)  
						\Big( p^{-l-2j}  u_{f}^2 - p^{-l-j}a(p,g^\rho) u_{f} + p^{-1} \bar{\psi}(p) \Big)  \\
						&=  p^{-l -2j} u_{f}^{2} W(g^{\rho})   P_{p}(g^{\rho},p^{j}  u_{f}^{-1}) 
						D_{I_{0}Np}(l+j,f_{0}, g) .
					\end{split}
				\end{align} 
			\end{small}%
			
			\noindent\underline{Case $ \alpha \neq 0$}:  
			By \cite[Lemma 5.2(i)]{Hidarankin2},  $\beta = \alpha+1$.  
			Then it follows from \eqref{iota fourier cusp} and \eqref{Atkin involution and d} that
			\begin{small}
				\begin{align*}
					D_{JNp}&(l+j,f_0, 
					\mu_{g^\rho}(\iota_{p})
					|\tau_{J p^{\beta}})  
					= W(g^{\rho})
					\Big( p^{ l/2} D_{JNp}(l+j,f_{0},
					g|[p]) -  p^{-l/2} 
					a(p,g^{\rho}) D_{JNp}(l+j,f_{0}, g) \Big).
				\end{align*}
			\end{small}%
			By a similar argument as in \eqref{cusp atkin and trivial char}, we have
			\begin{small}
				\begin{align*}
					D_{JNp}(l+j,f_0, 
					\mu_{g^\rho}(\iota_{p})
					|\tau_{J p^{\beta}})  
					& =     p^{- l/2}W(g^{\rho}) D_{JNp}(l+j,f_{0},
					g) \Big( p^{-j} u_{f}
					-  a(p,g^{\rho}) \Big) \\
					& =  p^{-j-l/2} u_{f} W(g^{\rho})    P_{p}(g^{\rho},p^j u_{f}^{-1})
					D_{JNp}(l+j,f_{0}, g).       	
				\end{align*}
			\end{small}%
			Noting that $\beta = \alpha +1$ finishes the proof.
		\end{proof}
		
		Let $v_{p}(\cdot)$ be the $p$-adic valuation on $\mathbb{Q}$ with $v_{p}(p)=1$. For every Dirichlet character $\psi$, let cond$(\psi)$ denote the conductor of $ \psi $ and $\psi_p$ denote the $p$-part of $\psi$ i.e. $\psi = \psi_{p}\psi'$ where $\psi_{p}$ is a character of $p$-power conductor and the conductor of $\psi'$ is co-prime to $p$. For a Dirichlet character $ \chi $ and a modular form $ g(z) = \sum_{n \geq 0} a(n,g) q^n$, let $ (g\vert \chi) (z) := \sum_{n \geq 1} \chi(n) a(n,g) q^n$ denote the twist of $ g $ by $ \chi$.  We next treat the case $\phi$ is a non-trivial  finite order character of $\mathbb{Z}_p^\times$ and $g$ is a cusp form.
		\begin{lemma}\label{special value f,cusp and non-trivial character}
			Let $f \in S_{k}(\Gamma_{0}(N),\eta)$ be a normalised $p$-ordinary eigenform with $p\nmid N$. Let $g \in S_{l}(\Gamma_{0}(J p^\alpha),\psi)$ be a $p$-ordinary newform with $p\nmid J$ and $2\leq l<k$. Let $\beta  $ be as defined  in Theorem~\ref{padic rankin}. Then for every non-trivial finite order  character $\phi$ of $\mathbb{Z}_p^\times$ with 
			$ \phi \neq \bar{\psi}_{p} $ and $0 \leq j\leq k-l-1$, we have $\beta = v_p(\mathrm{cond}(\phi)) + v_p( \mathrm{cond}(\phi \bar{\psi}))$ and 
			\begin{align*}
				D_{JNp}(l+j,f_{0}, 
				\mu_{g}(\bar{\phi})|\tau_{J p^{\beta}})  =  W(g^\rho|\bar{\phi})
				D_{JNp}(l+j,f_0, g|\phi).
			\end{align*}
		\end{lemma}
		\begin{proof}
			As observed in Lemma~\ref{special value f,cusp and trivial character} we have $g$ is $p$-minimal and not super-cuspidal. By the proof of \cite[Lemma 5.2(i)]{Hidarankin2} we have $g^\rho|\bar{\phi}$ is primitive and $\beta = v_p(\mathrm{cond}(\phi)) + v_p( \mathrm{cond}(\phi \bar{\psi}))$. Now the lemma follows.
		\end{proof}
		The following result is an immediate consequence  of Theorem~\ref{padic rankin} and Lemmas~\ref{special value f,cusp and trivial character}, \ref{special value f,cusp and non-trivial character}.
		\begin{corollary}\label{special value f,cusp}
			Let $f \in S_{k}(\Gamma_{0}(N),\eta)$ be a normalised $p$-ordinary newform with $p\nmid N$. Let $g \in S_{l}(\Gamma_{0}(J p^\alpha),\psi)$ be a $p$-ordinary newform with $p\nmid J$ and $2\leq l<k$. Let $\beta  $ be as defined  in Theorem~\ref{padic rankin}. Set $t(f,g)=[N,J] N^{k/2} J^{(l+2j)/2} \Gamma(l+j)\Gamma(j+1)$. Then we have
			\begin{enumerate}
				\item[$\mathrm{(i)}$] For $ 0 \leq j \leq k-l-1 $ and $\phi = \iota_{p}$, we have $\beta = \max\{\alpha,1\} +1 $ and 
				\begin{small}
					\begin{align*}
						\mu_{f \times g}(x_p^{j} \iota_p ) = c(f_0) t(f,g) p^{\alpha(l+2j)/2} u_{f}^{1-\alpha}   p^{(2-k)/2} W(g^\rho)   P_{p}(g,p^j u_{f}^{-1})
						\frac{D_{JNp}(l+j,f_0,g)}{(2i)^{k+l+2j} \pi^{l+2j+1} 
							{ \langle f_{0}^{\rho}|_{k}\tau_{Np} , f_{0} \rangle_{Np}}} .
					\end{align*} 
				\end{small} %
				\item[$\mathrm{(ii)}$] For $ 0 \leq j \leq k-l-1 $ and $\phi \neq \iota_{p}, \bar{\psi}_{p}$, we have  $\beta = v_p(\mathrm{cond}(\phi)) + v_p( \mathrm{cond}(\phi \bar{\psi}))$ and 
				\begin{small}
					\begin{align*}
						\mu_{f \times g}(x_p^{j} \phi ) = c(f_0) t(f,g) p^{\beta(l+2j)/2} u_{f}^{1-\beta}   p^{(2-k)/2} W(g^\rho|\bar{\phi})   
						\frac{D_{JNp}(l+j,f_0,g|\phi)}{(2i)^{k+l+2j} \pi^{l+2j+1} 
							{ \langle f_{0}^{\rho}|_{k}\tau_{Np} , f_{0} \rangle_{Np}}} .
					\end{align*} 
				\end{small} %
			\end{enumerate}
			
		\end{corollary}
		We also need the analogue of the above result in the case $g$ is an Eisenstein series. For any  Dirichlet character $ \psi $,  let $ \psi_0 $ denote the primitive character associated  to $ \psi $ and $ G(\psi) := \sum_{a =1}^ {\mathrm{cond}(\psi) } \psi_0(a) e^{2 \pi i a/\mathrm{cond}(\psi) }$ denote the Gauss sum of $ \psi $. Further, for a prime $ r $,  let $\mathrm{cond}_{r}(\psi) $ denote the $ r $-part of $ \mathrm{cond}(\psi)  $. 
		Let $ \chi_{\cyc} : G_{\mathbb{Q}} \rightarrow \mathbb{Z}_{p}^{\times}$ denote  the $ p $-adic cyclotomic character. We now define the Eisenstein series considered in this article.

		\begin{lemma}\label{Eisenstein series Hida}
			Let $\theta $ and $ \varphi $ be primitive Dirichlet characters modulo $ u $ and $ v $ respectively with $ \theta\varphi(-1) = (-1)^{l} $.  Then 
			\begin{small}{
					\begin{align*}
						E_{l}(\theta,\varphi)(z) = \delta_1(u) L(0,\varphi) +
						\delta(v) L(1-l, \theta) +\delta_{2}(u,v) \frac{i}{2 \pi (z-\bar{z})}
						+  \sum_{n=1}^{\infty} \sum_{0 < d \mid n}  \theta(d)
						\varphi(n/d) d^{l-1} q^{n} \in M_{l}(uv,\theta\varphi),
					\end{align*}
			} \end{small}
			\begin{scriptsize}
				\begin{align*}
					\mathrm{where} ~~	\delta_1(u) = \begin{cases}
						2^{-1} & \text{ if } l= u=1,\\
						0 & \text{ otherwise},
					\end{cases} \qquad 
					\delta_2(u,v) = \begin{cases}
						2^{-1} & \text{ if } l=2 \text{  and } u=v=1,\\
						0 & \text{ otherwise},
					\end{cases} \qquad 
					\delta(v) &= \begin{cases}
						2^{-1} & \text{ if }  v=1,\\
						0 & \text{ otherwise}.
					\end{cases}
				\end{align*}
			\end{scriptsize}
			Further, we have 
			\begin{enumerate}[label=$\mathrm{(\roman*)}$]
				\item $(E_{l}(\theta,\varphi)\vert \tau_{uv})(z) =
				(uv^{-1})^{l/2} \varphi(-1) G(\varphi)/G(\bar{\theta})
				E_{l}(\bar{\varphi}, \bar{\theta})(z)$, where $ \tau_{uv}= \begin{psmallmatrix} 0 & -1 \\  uv & 0
				\end{psmallmatrix} $. 
				\item $L(s,E_{l}(\theta,\varphi)) = L(s-l+1,\theta)L(s,\varphi)$.
				\item The Galois representation
				$ \rho_{E_{l}(\theta,\varphi)}: G_{\mathbb{Q}} \rightarrow \mathrm{GL}_{2}(\bar{\mathbb{Q}}_{p})
				$
				is isomorphic to $	\theta \chi_{\cyc}^{l-1} \oplus \varphi $.  Furthermore, the representation $\rho_{E_{l}(\theta,\varphi)}$ is unramified at all primes  not dividing $ puv$.
			\end{enumerate}
			
		\end{lemma}
		\begin{proof}
			The main assertion and part (i) are immediate from \cite[Lemma 5.3]{Hidarankin2}.  Part (ii) can be proved following the argument in  \cite[Theorem 4.7.1]{Miyake}.
			The part (iii) follows from \cite[Theorem 9.6.6]{DS05}. 
		\end{proof}
		
		Let $\theta $ and $ \varphi $ be primitive Dirichlet characters  with $ \theta\varphi(-1) = (-1)^{l} $. Set the \textquote{root number} of the $E_l(\theta,\varphi)$
		\begin{small}
			\begin{equation}\label{root number of g definition}
				W(E_l(\theta,\varphi)) =  (\mathrm{cond}(\theta)/\mathrm{cond}(\varphi) )^{l/2} \varphi(-1) G(\varphi)/G(\bar{\theta}).
			\end{equation}
		\end{small}%
		We now prove the following analogue of Lemma~\ref{special value f,cusp and trivial character} in the case $g$ is an  Eisenstein series. 
		\begin{lemma}\label{special value f,Eis and trivial character}
			Let $f \in S_{k}(\Gamma_{0}(N),\eta)$ be a normalised $p$-ordinary eigenform with $p\nmid N$. Let $\theta$ and $\varphi$ be primitive Dirichlet characters of conductors $u,v$ respectively with $\theta\varphi(-1)=(-1)^l$ and $2 \leq l < k$. Set  $uv = Jp^s$ with $p\nmid J$. Let $g  = E_{l}(\theta, \varphi) \in M_{l}(\Gamma_{0}(Jp^s), \theta\varphi)$ and  $ g'= E_{l} (\varphi,\theta)\in M_{l}(\Gamma_{0}(Jp^s), \theta\varphi)$. Assume that $g$ is ordinary at $p$. Let $\beta  $ be as defined  in Theorem~\ref{padic rankin}. For $ 0 \leq j \leq k-l-1 $, we have 
			\begin{small}
				\begin{align*}
					p^{\beta(l+2j)/2} D_{JN p}(l+j,f_0, 
					\mu_{(g^\rho)}(\iota_{p})
					|\tau_{Jp^{\beta}})  =   p^{s(l+2j)/2} u_{f}^{\beta-s} W(g^\rho)   P_{p}(g^\rho,p^j u_{f}^{-1}) 
					D_{JNp}(l+j,f_{0}, g')
				\end{align*} 
			\end{small} %
			where $\beta = \max\{ s,1\} +1$ is as defined  in Theorem~\ref{padic rankin} and $W(g^\rho) = (u/v)^{l/2} \varphi(-1) G(\bar{\varphi})/ G(\theta)$.
		\end{lemma}
		\begin{proof}
			Since $ g^\rho $ is a Hecke eigenform,  we have  (see \cite[Page 46]{Hidarankin2})
			\begin{align}\label{iota fourier Eis}
				g^\rho|\iota_{p} = 
				\begin{cases}
					g^\rho - a(p,g^\rho) g^\rho|[p]+\bar{\theta}\bar{\varphi}(p)p^{l-1} g^\rho|[p^2] 
					& \text{ if } s \neq 0, \\
					g^\rho - a(p,g^\rho) g^\rho|[p] & \text{ if } s = 0.
				\end{cases}
			\end{align}
			For every $ i \geq 0 $, by Lemma~\ref{Eisenstein series Hida} we obtain
			\begin{small}
				\begin{equation}\label{Atkin involution and p, g}
					\begin{aligned}
						(g^\rho)|[p^i]|\tau_{J p^{\beta}} & = p^{-il/2} 
						(g^\rho)|\begin{psmallmatrix}
							p^i & 0 \\ 0 &1 \end{psmallmatrix} 
						\begin{psmallmatrix}
							0 & -1 \\ J p^{\beta} &0 \end{psmallmatrix} \\
						& =p^{-il/2} 
						p^{(\beta-s-i)l/2}
						(g^{\rho}|\tau_{Jp^s})| [p^{\beta-s-i}]  \\
						& = W(g^\rho) p^{(\beta-s-2i)l/2}
						g'|[p^{\beta-s-i}]. 
					\end{aligned}
				\end{equation}
			\end{small}%
			\underline{Case $ s = 0$}:   We have $p \nmid u,v$. By \cite[Lemma 4.6.1]{Miyake} and \eqref{iota fourier Eis}, we have $\beta \leq 2$. We claim that $\beta =2$. If $\beta =0$, then by \eqref{iota fourier Eis} we have $(a(p,g^\rho) g^\rho+\bar{\theta}\bar{\varphi}(p)p^{l-1} g^\rho|[p])|[p] = a(p,g^\rho) g^\rho|[p]+\bar{\theta}\bar{\varphi}(p)p^{l-1} g^\rho|[p^2]  \in M_l(\Gamma_{0}(J),\bar{\theta}\bar{\varphi}) \subset M_l(\Gamma_{0}(Jp),\bar{\theta}\bar{\varphi})$. By \cite[Theorem 4.6.4(1)]{Miyake}, we have $\bar{\theta}\bar{\varphi}(p)p^{l-1} g^\rho|[p] \in M_l(\Gamma_{0}(J),\bar{\theta}\bar{\varphi})$. Since $\bar{\theta}\bar{\varphi}(p) \neq 0$ in this case, we get $g^\rho|[p] \in M_l(\Gamma_{0}(J),\bar{\theta}\bar{\varphi})$. Since $p \nmid J$, this leads to contradiction according to \cite[Lemma 4.6.1(2)]{Miyake}. Hence $\beta \geq 1$. If $\beta =1$, then $\bar{\theta}\bar{\varphi}(p)p^{l-1} g^\rho|[p^2] \in M_l(\Gamma_{0}(Jp),\bar{\theta}\bar{\varphi})$. Hence by  \cite[Theorem 4.6.4(1)]{Miyake}, we have $g^\rho|[p] \in M_l(\Gamma_{0}(J),\bar{\varphi}\bar{\theta})$. By a similar argument as before, this again leads to a contradiction.  Hence $\beta =2 $ as claimed.
			Therefore by \eqref{iota fourier Eis} and \eqref{Atkin involution and p, g}, we have
			\begin{small}
				\begin{align*}
					& D_{JNp}(l+j,f_0, 
					\mu_{g^\rho}(\iota_{p})
					|\tau_{Jp^{\beta}})  
					\\ 
					& = D_{JNp}(l+j,f_0, 
					g^\rho
					|\tau_{Jp^{\beta}})  - a(p,g^\rho) D_{JNp}(l+j,f_0, g^\rho|\tau_{J p^{\beta}}|[p])+\bar{\theta}\bar{\varphi}(p) p^{l-1} D_{JNp}(l+j,f_0,g^\rho)|[p^2]\tau_{J p^{\beta}}) 
					\\ 	      
					& =W(g^\rho) \Big(p^l D_{JNp} (l+j,f_0, g|[p^{\beta}]) - a(p,g^\rho) D_{JNp} (l+j,f_0, g|[p^{\beta-1}]) + \bar{\theta}\bar{\varphi}(p)
					p^{-1} D_{JNp} (l+j,f_0, g|[p^{\beta-2}]) \Big).  	
				\end{align*}
			\end{small}%
			Since  $ L(l+j,f_{0}, g|[p^i]) = p^{-i(l+j)}  u_{f}^{i} L(l+j,f_{0}, g)$ and $\beta =2$, we get
			\begin{small}
				\begin{align}\label{Eis atkin and trivial char}
					\begin{split}
						D_{JNp}(l+j,f_{0}, 
						\mu_{g^{\rho}}(\iota_{p})
						|\tau_{Jp^{\beta}})  
						& =   W(g^{\rho})  D_{JNp}(l+j,f_{0},g)  \Big(
						p^{-l-2j}  u_{f}^2 - p^{-l-j}a(p,g^\rho) u_{f} +  p^{-1} \bar{\theta}\bar{\varphi}(p) \Big)  \\
						&=  p^{-l -2j} u_{f}^{2} W(g^{\rho})   P_{p}(g^{\rho},p^{j}  u_{f}^{-1}) 
						D_{JNp}(l+j,f_{0}, g) .
					\end{split}
				\end{align} 
			\end{small}%
			
			\noindent \underline{Case $ s \neq 0$}:  Since $g$ is $p$-ordinary, we have we have  $a(p,g) \neq 0$. Thus $p \nmid u$ or $p \nmid v$. 
			Again by \cite[Lemma 4.6.1]{Miyake}, we have  $\beta \leq s+1$. If $\beta \leq s$, then by \eqref{iota fourier Eis} and $a(p,g)\neq 0$ we have $g^\rho|[p] \in M_l(\Gamma_0(Jp^s), \bar{\theta}\bar{\varphi})$. By \cite[Theorem 4.6.4(2)]{Miyake}, we get $g=0$ which is a contradiction. Hence $\beta = s+1$.   
			Then it follows from \eqref{iota fourier Eis} and \eqref{Atkin involution and p, g} that 
			\begin{small}
				\begin{align*}
					D_{JNp}&(l+j,f_0, 
					\mu_{g^\rho}(\iota_{p})
					|\tau_{J p^{\beta}})  
					= W(g^{\rho})
					\Big( p^{ l/2} D_{JNp}(l+j,f_{0},
					g|[p]) -  p^{-l/2} 
					a(p,g^{\rho}) D_{JNp}(l+j,f_{0}, g) \Big).
				\end{align*}
			\end{small}%
			By a similar argument as in \eqref{Eis atkin and trivial char}, we have
			\begin{small}
				\begin{align*}
					D_{JNp}(l+j,f, 
					\mu_{g^\rho}(\iota_{p})
					|\tau_{J p^{\beta}})  
					& =     p^{- l/2}W(g^{\rho}) D_{JNp}(l+j,f_{0},
					g) \Big( p^{-j} u_{f}
					-  a(p,g^{\rho}) \Big) \\
					& =  p^{-j-l/2} u_{f} W(g^{\rho})    P_{p}(g^{\rho},p^j u_{f}^{-1})
					D_{JNp}(l+j,f_{0}, g).       	
				\end{align*}
			\end{small}%
			Noting that $\beta = s +1$ finishes the proof.
		\end{proof}
		We next treat 
		the analogue of Lemma~\ref{special value f,cusp and non-trivial character} in the case $g$ is an  Eisenstein series.
		\begin{lemma}\label{special value f,Eis and nontrivial character}
			Let $f,g $ and $g'$ be as in Lemma~\ref{special value f,Eis and trivial character}. Let $\phi$ be a non-trivial finite order character of $\mathbb{Z}_p^\times$ with  $\phi \neq \iota_p, \bar{\theta}_p, \bar{\varphi}_p$.  Let $\beta  $ be as defined  in Theorem~\ref{padic rankin}. For every $ 0 \leq j \leq k-l-1 $, we have 
			\begin{small}
				\begin{align*}
					D_{JN p}(l+j,f_0, 
					\mu_{(g^\rho)}(\bar{\phi})
					|\tau_{Jp^{\beta}})  =   W(g^\rho|\phi)   
					D_{JNp}(l+j,f_{0}, g'|\phi)
				\end{align*} 
			\end{small} %
			where $\beta = v_{p}(\mathrm{cond}(\theta \phi)) + v_{p}(\mathrm{cond}(\varphi \phi))$  and $W(g^\rho|\bar{\phi}) = (\mathrm{cond}(\theta \phi)/\mathrm{cond}(\varphi \phi))^{l/2} \varphi\phi(-1) G(\bar{\varphi}\bar{\phi})/ G(\theta \phi)$.
		\end{lemma}
		\begin{proof}
			Recall that  $(\bar{\theta}\bar{\phi})_0$ and $(\bar{\varphi}\bar{\phi})_0$ is the primitive character associated to $\bar{\theta}\bar{\phi}$ and $\bar{\varphi}\bar{\phi}$ respectively. We claim that  $a(n,g^\rho|\bar{\phi}) = a(n, E_{l}((\bar{\theta}\bar{\phi})_0,(\bar{\varphi}\bar{\phi})_0)$. If $p\nmid n$, then it is easy to check that above equality holds. If $p\mid n$, then  $a(n,g^\rho|\bar{\phi})=0$. Since $\phi \neq \bar{\theta}_p, \bar{\varphi}_p$, it follows that $(\bar{\theta}\bar{\phi})_0(p) ,(\bar{\varphi}\bar{\phi})_0(p) =0$. Hence $g^\rho|(\bar{\phi}) = E_{l}((\bar{\theta}\bar{\phi})_0,(\bar{\varphi}\bar{\phi})_0)$. Noting that $G(\bar{\varphi}\bar{\phi})= G((\bar{\varphi}\bar{\phi})_0)$ and $G(\theta\phi) = G((\theta\phi)_0)$, the lemma now follows from Lemma~\ref{Eisenstein series Hida}.
		\end{proof}
		We now deduce the following result which relates the $p$-adic Rankin-Selberg $L$-function $\mu_{f \times E_{l}(\theta, \varphi)} $ with the product of $L$-function attached to modular form.
		\begin{corollary}\label{special value f,Eis}
			Let $f \in S_{k}(\Gamma_{0}(N),\eta)$ be a normalised $p$-ordinary newform with $p\nmid N$. Let $\theta$ and $\varphi$ be primitive Dirichlet characters of conductors $u,v$ respectively with $\theta\varphi(-1)=(-1)^l$ and $2 \leq l < k$. Set  $uv = Jp^s$ with $p\nmid J$. Let $g  = E_{l}(\theta, \varphi) \in M_{l}(\Gamma_{0}(Jp^s), \theta\varphi)$ and  $ g'= E_{l} (\varphi,\theta)\in M_{l}(\Gamma_{0}(Jp^s), \theta\varphi)$. Assume that $g$ is ordinary at $p$. Let $\beta  $ be as defined  in Theorem~\ref{padic rankin}. Set $t(f,g)=[N,J] N^{k/2} J^{(l+2j)/2} \Gamma(l+j)\Gamma(j+1)$. Then we have
			\begin{enumerate}
				\item[$\mathrm{(i)}$] For $ 0 \leq j \leq k-l-1 $ and $\phi = \iota_{p}$, we have $\beta = \max\{s,1\} +1 $ and 
				\begin{small}
					\begin{align*}
						\mu_{f \times g}(x_p^{j} \iota_p ) = c(f_0) t(f,g) p^{s(l+2j)/2} u_{f}^{1-s}   p^{(2-k)/2} W(g^\rho)   P_{p}(g,p^j u_{f}^{-1})
						\frac{L(j+1,f_0,\varphi)    
							L(l+j,f_0,\theta)}{(2i)^{k+l+2j} \pi^{l+2j+1} 
							{ \langle f_{0}^{\rho}|_{k}\tau_{Np} , f_{0} \rangle_{Np}}} ,
					\end{align*} 
				\end{small} %
				where $W(g^\rho) = (\mathrm{cond}(\theta)/\mathrm{cond}(\varphi ))^{l/2} \varphi(-1) G(\bar{\varphi})/ G(\theta )$.
				\item[$\mathrm{(ii)}$] For $ 0 \leq j \leq k-l-1 $ and $\phi \neq \iota_{p}, \bar{\theta}_p, \bar{\varphi}_p$, we have  $\beta = v_{p}(\mathrm{cond}(\theta \phi)) + v_{p}(\mathrm{cond}(\varphi \phi))$ and 
				\begin{small}
					\begin{align*}
						\mu_{f \times g}(x_p^{j} \phi ) = c(f_0) t(f,g) p^{\beta(l+2j)/2} u_{f}^{1-\beta}   p^{(2-k)/2} W(g^\rho|\bar{\phi})   
						\frac{L(j+1,f_0,\varphi\phi)    
							L(l+j,f_0,\theta\phi)}{(2i)^{k+l+2j} \pi^{l+2j+1} 
							{ \langle f_{0}^{\rho}|_{k}\tau_{Np} , f_{0} \rangle_{Np}}} ,
					\end{align*} 
				\end{small} %
				where $W(g^\rho|\bar{\phi}) = (\mathrm{cond}(\theta \phi)/\mathrm{cond}(\varphi \phi))^{l/2} \varphi\phi(-1) G(\bar{\varphi}\bar{\phi})/ G(\theta \phi)$. 
			\end{enumerate}
		\end{corollary}
		\begin{proof}
			By Lemma~\ref{Eisenstein series Hida}(ii), we get $ L(l+j,g'|\phi) = 
			L(l+j, \theta \phi) L(j+1, \varphi \phi) $.
			Thus by Rankin-Selberg method (See \cite[Lemma 1]{Shimura1}), we obtain 
					$
					D_{JNp}(l+j,f_{0},g'|\phi) 
					= L(l+j,f_{0},\theta\phi)  L(j+1,f_{0}, \varphi \phi)$.  Now the corollary follows from Theorem~\ref{padic rankin}, Lemma~\ref{special value f,Eis and trivial character} and Lemma~\ref{special value f,Eis and nontrivial character}.
				\end{proof}
				\section{Towards the congruence of the p-adic L-functions}\label{section: simplyfying rankin}
				
				Fix $f(z) = \sum_{n=1}^{\infty} a(n,f) q^n \in S_{k}(\Gamma_0(N), \eta)$ a primitive form.  We shall assume throughout this 
				article  that $ p \nmid N $ and  $ f $ is $p$-ordinary.  
				Let $ f_0 $ be the $ p $-stabilization of $ f $ as in \eqref{p-stabilization of f}. 	
				For an integer $ J $, let $ J_0 $ denote the prime to $ p $-part of  $ J$. For a normalised Hecke eigenform $ g $ with  nebentypus $\varphi$, let   
				$K_{g}$  be the number field generated by the Fourier coefficients of $g$  and the values of $\varphi$.  

				We also fix  $ h(z) =   \sum_{n=1}^{\infty} a(n,h) q^n
				\in M_{l}(\Gamma_{0}(I), \psi)$, 
				a normalized  newform of weight 
				$ 2 \leq l < k$. Assume that  $ h $ is $ p $-ordinary. Put $ I = I_{0}p^{\alpha} $
				where $ I_0 $ is co-prime to $ p $.  Let $ K $ be a number field containing $ K_f$ and $ K_h $. Let $K_{\mathfrak p}$ denote the completion of $ K $ at a prime  $\mathfrak{p}$ lying above $p$ induced by the embedding $i_p:\bar{\Q} \rightarrow \bar{\Q}_p$ and $ \mathcal{O}_{K_{\mathfrak{p}}} $ be the ring of integers of $K_{\mathfrak p}$ and $ \pi $ be a uniformizer of $ \mathcal{O}_{K_{\mathfrak{p}}} $.  Let $ (\rho_{h},V_{h}) $ be the $ p $-adic Galois
				representation attached to the modular form $h$ (see Theorem~\ref{rhof}). Choose a 
				rank two $ \mathcal{O}_{K_{\mathfrak{p}}} $-submodule
				$T_h$ of $V_h$ which is invariant under the action of
				the absolute Galois group $ G_{\mathbb{Q}} := \gal (\bar{\mathbb{Q}}|\mathbb{Q}) $.  Let $\bar{\rho}_{h}: G_{\Q} \lra \mathrm{GL_2}(T_h/\pi)$ be the reduction of $ \rho $ modulo $ \pi $. Assume that the semi-simplification of $ T_{h}/\pi $ is isomorphic to $  \bar{\xi}_{1} \oplus \bar{\xi}_{2}$, i.e.
				\begin{align}\label{eq: splitting rho_h}
					0 \rightarrow \bar{\xi}_{1} \rightarrow T_{h}/ \pi 
					\rightarrow \bar{\xi}_{2} \rightarrow 0.
				\end{align}
				
				In this section,  we use the residual characters  $ \bar{\xi}_1 $ and $ \bar{\xi}_{2} $ to construct an Eisenstein series $ g $ congruent to $ h $. We then express the  special values of the Rankin-Selberg $ L $-function $ f_0 \otimes g $ as a product  of the special values of $ L $-function associated to $ f_0 \otimes \xi_{1} $ and $ f_0 \otimes \xi_{2}  $. In the next section, using these results we shall show that the $ p$-adic Rankin-Selberg $ L $-function associated to $ f \otimes g $ is congruent to the product of $ p $-adic $ L $-functions of $ f \otimes \xi_1 $ and $ f \otimes \xi_2 $.   
				We  begin by constructing  the Dirichlet characters $ \xi_i $ whose reduction is $ \bar{\xi}_{i} $ for $ i =1,2 $.
				
				\begin{lemma}\label{lemma lifting characters}
					Let $ \bar{\xi}: (\mathbb{Z}/v\Z)^{\times} \rightarrow \mathbb{F}_{p^r}^{\times} $ be a character.
					Then there exists a Dirichlet character $ \xi: (\mathbb{Z}/v \Z)^{\times} \rightarrow \mu_{p^a -1}$ such that  the reduction of $ \xi $ equals $ \bar{\xi} $, where $\mu_{p^a -1} : = \{ z \in \mathbb{C} : z^{p^a -1} = 1\}$. Further,  if the conductor of $ \bar{\xi} $ equals $ v_0 p^a $ with $ p \nmid v_0 $, then the conductor of $ \xi $ is  $ v_0 p^{a'} $ with $ a' = \min\{a,1\} $.    	 
				\end{lemma}
				\begin{proof}
					Without any loss of generality, assume that $ \bar{\xi} $ is primitive.  Since $ \bar{\mathbb{F}}_{p^r} $ doesn't contain non-trivial $ p^{\mathrm{th}} $ root of unity,
					we get that the image of $ p $-Sylow subgroup of
					$(\mathbb{Z}/v\Z)^{\times}$ is trivial under $ \bar{\xi} $. As $ \bar{\xi} $ is primitive, we obtain $ p^2 \nmid v $ and $ v = v_0 p^a $ with $ a \in \{ 0,1 \} $ and $ p \nmid v_0 $. Choose a finite extension $ L/\Q_p $ such that $ (\mathcal{O}_L/\pi)^{\times} \cong  \mathbb{F}_{p^r}^{\times} $. Composing $ \bar{\xi} $  with the Teichm\"uller character $ \omega:  (\mathcal{O}_L/\pi)^{\times} \rightarrow \mu_{p^r -1}$ we obtain a lift $ \xi: (\mathbb{Z}/v \mathbb{Z})^{\times} \rightarrow  \mu_{p^r -1} \hookrightarrow \mathbb{C}^{\times}$ of $ \bar{\xi} $. If the conductor of $ \xi $ is strictly less than $ v $, then it would imply that the conductor of $ \bar{\xi} $ is strictly less than $ v $. Thus the conductor of  $ \xi $ is $  v_0 p^a $. This finishes the proof.
				\end{proof}
				Since $ h $ is $ p $-ordinary, either $ \bar{\xi}_{1} $ or $ \bar{\xi}_{2} $ is unramified at $ p $ (see Theorem~\ref{rhof}(ii)).   Without any loss of generality, assume $ \bar{\xi}_{2} $ is unramified at $p$. Since $ \bar{\xi}_{2}(\mathrm{Frob}_p)  \neq 0$, where $ \mathrm{Frob}_p $ is the arithmetic Frobenius at $ p $, we get $p$ doesn't divide the conductor of $\bar{\xi}_2$. From Lemma~\ref{lemma lifting characters}, it follows that there exists a Dirichlet character $ \xi_2 $ whose reduction equals $ \bar{\xi}_2 $ and the conductor of $ \xi_{2} $ is  $ M_2  $ with  $ p \nmid M_2$.  Similarly, we can lift $ \bar{\xi}_{1} \bar{\omega}_p^{1-l}$ to a Dirichlet character $ \xi_{1} \omega_p^{1-l}  $ such that  the conductor of $ \xi_{1} \omega_p^{1-l} $ is of the form $ M_1 p^{s} $	with  $ (M_1, p)  = 1$ and $ s \in \{ 0,1\} $.  Since the conductor of $ (T_h/\pi) $ divides the conductor of $ T_{h} $, we obtain $ M_1 M_2 \mid I_{0} $. Set $ M := M_1 M_2 p^{s}$.  Then $ M_{0} = M_{1} M_{2} $ and  $ M_{0} \mid I_{0} $.  Put
				$\Sigma_0 = \{ r \text{ is prime} : r \mid (I_0/ M_0) \text{ or } r^2 \mid M_0 \}$ and set
				\begin{align}\label{definition m}
					m := \prod_{r \in \Sigma_0} r. 
				\end{align}
				Note that $ \bar{\xi}_{1} \bar{\xi}_{2} \bar{\omega}_p^{1-l}(-1) = \bar{\psi}(-1) = (-1)^{l} $.  As $p$ is  odd,   
				$  \xi_{1} \xi_{2} \omega_p^{1-l} (-1) = (-1)^{l} $. Define 
				\begin{align}\label{definition of Eis g}
					g (z) := E_{l}(\xi_{1} \omega_p^{1-l}, \xi_2)(z) .
				\end{align} 
				From Lemma~\ref{Eisenstein series Hida}, we have $ g (z)  \in M_{l}(\Gamma_{0}(M), \xi_{1} \xi_{2} \omega_p^{1-l}) $ and $ g(z) = \sum_{n \geq 0} a(n,g) q^n $, with $ a(n,g) = \sum_{0 < d \mid n}^{} \xi_1 \omega_p^{1-l}(d) \xi_{2}(n/d) d^{l-1}$ for all $n \geq 1$. Further, we have  $ \bar{\rho}_g \simeq \bar{\xi}_1 \oplus \bar{\xi}_{2} \simeq (T_h/\pi)^{ss}$. As $ \xi_2(p) \neq 0$, it follows that $ a(p,g) =  \xi_{2}(p) + p^{l-1}  \xi_1 \omega_p^{1-l}(p)$ is a $ p $-adic unit. Hence $g$ is $p$-ordinary. Recall that for an integer $ J $,  $ \iota_{J} $ denotes the trivial character of modulus $ J $. We now show that $ g|\iota_{mp}$ and $ h|\iota_{mp}$ are congruent  modulo $ \pi $.  
				
				\begin{lemma}\label{lem: congruence of g,h}
					With the notation as above, we have $ h |\iota_{pm} \equiv g |\iota_{pm}  \mod \pi$.
				\end{lemma}
				\begin{proof}
					Let $ r $ be a prime. Since $( T_{h}/\pi)^{ss} $ and $ (T_{g}/\pi)^{ss} $ are the same, we get $  a(r,h|\iota_{pm}) =  a(r,h)  
					\equiv a(r,g) = a(r,g|\iota_{pm}) \mod \pi$, $ \forall$ $ r  \nmid pmM_{0}$. If $ r =p $ or $ r \mid m $, then $ a(r,h|\iota_{pm}) = 0 = a(r,g|\iota_{pm})  $. Next, consider $ r \mid  M_0 $ and $ r \nmid pm $. This forces that $ r \nmid I_0/M_0 $, $ r \| M_{0} $ and $ r \|I_{0} $. Since $ \psi $ and $ \xi_{1} \xi_{2} \omega_p^{1-l} $ have the  same reduction modulo $ p $, we also get that $ r $ divides  the conductor of $ \psi $. As $ \text{cond}(\psi) \mid I $ and $ r \| I $, we obtain  $ r \| \mathrm{cond}(\psi) $. Thus by \cite[Theorem 3.26]{Hida3},  we have $ a(r, h) = \chi(\text{Frob}_{r}) $, where $ \chi $ is the unique unramified  character appearing in the restriction of $ \rho_{h} $  to the decomposition subgroup at $r$. Since $ r \| M_0 $  and $ M_0 =  M_1 M_2$, we get either $ r \| M_1,  r \nmid M_2 $ or $ r \| M_2, r \nmid M_1$. Suppose we are in the situation where $r \|M_{1}$ and $ r \nmid M_2 $. Then $ \xi_2 $ is unramified at $ r $ and $ \xi_1 $ is ramified at $ r $. Also $ a(r,g) = \xi_{2}(r) = \xi_{2}(\text{Frob}_{r})$. Since $ (T_{h}/\pi)^{\mathrm{ss}} = \bar{\xi}_{1} \oplus \bar{\xi}_{2} $ and  $ \bar{\xi}_{1} $ is ramified at $ r $, we obtain $ \bar{\chi} = \bar{\xi}_{2} $. Thus $ a(r,h) = \chi(\text{Frob}_{r}) \equiv \bar{\xi}_{2}(\text{Frob}_{r}) \equiv a(r,g) \mod \pi$. If $ r \| M_{2} $ and $ r \nmid M_1 $, then $ a(r,g) = \xi_{1}(r) \omega_p^{1-l}(r) r^{l-1} \equiv \xi_{1}(r) \mod p$. Again by a similar argument, we deduce  $ a(r,h) = \chi(\text{Frob}_{r}) \equiv \bar{\xi}_{1}(\text{Frob}_{r}) \equiv a(r,h) \mod \pi$. Thus in  either case, we have $ a(r,h|\iota_{pm}) = a(r,h) \equiv a(r,g) =  a(r,g|\iota_{pm})  \mod \pi$. This proves the lemma.
				\end{proof}
				We  note the following observation  made while proving in the lemma above:
				\begin{lemma}\label{lem: exactly divides}
					If $ r \mid I_0 $ and $ r \nmid pm $, then $ r \| I_0 $, $r\|M_0$ and $ r \| \mathrm{cond}(\psi) $. 
				\end{lemma} 
				For every prime $r$, let $v_r(\cdot)$ be the $r$-adic valuation on $\Q$ with $v_r(r) =1$. Recall that $f \in S_k(\Gamma_{0}(N),\eta)$ 
				and $h \in S_l(\Gamma_{0}(I_0p^\alpha), \psi)$ are primitive with $p\nmid NI_0$. Also $g = E_l(\xi_1\omega_p^{1-l},\xi_2) \in M_l(\Gamma_0(M_0p^s), \xi_1\xi_2\omega_p^{1-l})$ with $p \nmid M_0$. For every prime $r\mid m$, set $n_r = \max\{v_r(N), v_r(I_0), v_r(M_0)\} + 2$ and $n = \max\{n_r\}$. For every prime $r \mid m$, choose a primitive character $\chi_r$ of conductor $r^{n}$ and set $\chi = \prod_{r \mid m} \chi_r$ be a Dirichlet conductor $m^n$. Using that $v_{r}(\mathrm{cond}(\chi)) = v_r(\mathrm{cond}(\chi_r)) > \max\{v_r(N), v_r(I_0), v_r(M_0)\} > \max \{ v_r(\mathrm{cond}(\eta)), v_r(\mathrm{cond}(\psi)), v_r(\mathrm{cond}(\xi_1\omega_p^{1-l})), v_{r}(\mathrm{cond}(\xi_2))\}$, it is easy to check that the Dirichlet character $\chi$ satisfies the following conditions: 
				\begin{enumerate}
					\item[(i)] $\chi$ is primitive and conductor of $\chi$ equals $m^n$ with $n \geq 2$.
					\item[(ii)] For every prime $r \mid m$, we have $v_r(\mathrm{cond}(\chi\eta)) = v_r(\mathrm{cond}(\chi)) $ and $v_r(\mathrm{cond}(\chi)) > v_{r}(N)$.
					\item[(iii)] For every prime $r \mid m$, we have $v_r(\mathrm{cond}(\bar{\chi}\psi)) = v_r(\mathrm{cond}(\chi)) $ and $v_r(\mathrm{cond}(\bar{\chi})) > v_{r}(I_0p^\alpha)$.
					\item[(iv)] For every prime $r \mid m$, we have $v_r(\mathrm{cond}(\bar{\chi}\xi_2)) = v_r(\mathrm{cond}(\bar{\chi})) $, $v_r(\mathrm{cond}(\bar{\chi}\xi_1 \omega_{p}^{1-l})) = v_r(\mathrm{cond}(\bar{\chi})) $ and $v_r(\mathrm{cond}(\bar{\chi})) > v_{r}(M_0p^s)$.
				\end{enumerate}
				Throughout this article we denote the level of $f|\chi$, $g|\bar{\chi}$, $h|\bar{\chi}$ by $N_{f|\chi}$, $N_{g|\bar{\chi}}$ and $N_{h|\bar{\chi}}$ respectively. We describe the values of $N_{f|\chi}$, $N_{g|\bar{\chi}}$ and $N_{h|\bar{\chi}}$ in the following lemma. 
				\begin{lemma}\label{lem: primitive after twist}
					We have $f|\chi \in S_{k}(N_{f|\chi}, \eta \chi^2)$ and  $h|\bar{\chi} \in S_{l}(N_{h|\bar{\chi}}, \psi \bar{\chi}^2)$ are newforms  with $N_{f|\chi} = [N,\mathrm{cond}(\chi)]\mathrm{cond}(\chi)$ and  $N_{h|\bar{\chi}} = [I_0p^\alpha,\mathrm{cond}(\chi)]\mathrm{cond}(\chi)$. Also, we have $g|\bar{\chi} = E_{l}(\xi_{1}\omega_{p}^{1-l}\bar{\chi}, \xi_{2}\bar{\chi}) \in M_{l}(N_{g|\bar{\chi}},\xi_{1}\xi_{2}\omega_{p}^{1-l}\bar{\chi}^2)$ with $N_{g|\bar{\chi}} = [M_0p^s,\mathrm{cond}(\chi)]\mathrm{cond}(\chi)$. Furthermore, we have $[N_{f|\chi}, N_{h|\bar{\chi}}p^{-\alpha}] = [N_{f|\chi}, N_{g|\bar{\chi}}p^{-s}]$.
				\end{lemma}
				\begin{proof}
					The first assertion follows from \cite[Theorem 4.1]{Atkin}. The second assertion follows from the property (iv) listed above. To verify the last assertion we do case by case analysis and show the $r$-adic valuations of $[N_{f|\chi}, N_{g|\bar{\chi}}]$ and $[N_{f|\chi}, N_{h|\bar{\chi}}]$ are the same. As $M_0 p^s\mid I_0 p^\alpha$, we need to consider the cases (a) $r\mid m$, (b) $r\nmid m$ but $r \mid I_0$ and (c) $r\nmid I_0$ and $r\mid N$. If $r$ is prime and $r\mid m$, then by property (ii)-(iv) listed above, we have $v_r([N_{f|\chi}, N_{g|\bar{\chi}}]) = 2 v_r(\mathrm{cond}(\chi)) = v_r([N_{f|\chi}, N_{h|\bar{\chi}}]) $. If $r$ is a prime $r\nmid m$ but $r \mid I_0$, then by Lemma~\ref{lem: exactly divides} we have $v_r([N_{f|\chi}, N_{g|\bar{\chi}}]) = v_r([N_{f|\chi},r])= v_r([N_{f|\chi}, N_{h|\bar{\chi}}]) $. If $r$ is a prime $r\nmid I_0$ but $r \mid N$, then $v_r([N_{f|\chi}, N_{g|\bar{\chi}}]) = v_r(N_{f|\chi})= v_r([N_{f|\chi}, N_{h|\bar{\chi}}]) $. This finishes the proof of the last assertion.
				\end{proof}
				\begin{remark}\label{twist and untwist remark}
					If  $F \in S_{k}(\Gamma_{0}(N), \phi;K)$ is primitive, then the  Hecke algebra $h_{k}(\Gamma_0(N), \chi;K_{\mathfrak{p}})$ 
					splits (as described  in  \eqref{splitting of hecke algebra}). This splitting of Hecke algebra is essential  in the construction of $p$-adic Rankin-Selberg $L$-function $\mu_{F \times \star}$ by Hida (See Theorem~\ref{padic rankin}). As $f|\iota_m$ is not necessarily primitive, a priori one needs to assume the splitting of Hecke algebra $h_k(\Gamma_0(N_{\tilde{f}}), \eta \iota_m;K_{\mathfrak{p}})$ to guarantee the existence of the $p$-adic $L$-function $\mu_{f|\iota_m \times h|\iota_m}$.  As  $f|\chi$ is primitive, the map $\phi_{f|\chi}$ induces the splitting of the corresponding Hecke algebra. This ensures the existence of the $p$-adic $L$-functions  $\mu_{f|\chi \times h|\bar{\chi}}$ and $\mu_{f|\chi \times g|\bar{\chi}}$ and Theorem~\ref{c(f) and petterson} holds for $f|\chi$. We later show  that the ideal generated by \textquote{twisted} $p$-adic  $L$-function $\mu_{f|\chi \times \star|\chi}$ is same as the ideal generated by $\Sigma_0$-imprimitive $p$-adic  $L$-function $\mu^{\Sigma_0}_{f \times \star}$ in the Iwasawa algebra.
				\end{remark}
				
				Using the congruence  in Lemma~\ref{lem: congruence of g,h} we deduce the following congruence of $p$-adic measures:
				
				\begin{lemma}\label{congruence of measures}
					Let $f \in S_{k}(\Gamma_{0}(N),\eta)$ be a $p$-ordinary newform and $g,h$ be as before. 
					Let $\chi$ be a primitive Dirichlet character be as above. Then for every finite order character $ \phi \in C(\mathbb{Z}_{p}^{\times} ; \bar{\mathbb{Q}}_p) $ and $ 0 \leq j \leq k-l-1 $, we have
					\[ 
					\mu_{f|\chi \times g|\bar{\chi}} (x_p^j\phi)  \equiv 
					\mu_{f|\chi \times h|\bar{\chi}} (x_p^j \phi) \mod \pi.
					\]
				\end{lemma}
				\begin{proof}
					By Lemma~\ref{lem: congruence of g,h},  we have $a(n,h|\chi)  \equiv a(n,g|\chi)  \mod \pi  $ for all $p\nmid n$. As $ \phi(p) = 0 $, for every finite order character $\phi$ of $\mathbb{Z}_p^\times$, it follows from (\ref{Def measure mu star}) that $\mu_{g^{\rho}|\chi}(\phi) \equiv \mu_{h^{\rho}|\chi}(\phi)  \mod \pi$. By  Lemma~\ref{lem: primitive after twist}, we have $L = \mathrm{lcm}(N_{f|\chi},N_{g|\bar{\chi}}) = \mathrm{lcm}(N_{f|\chi},N_{h|\bar{\chi}}) $.
					Hence $ \mu^{L}_{g^{\rho}|\chi}(\phi) \equiv \mu^{L}_{h^{\rho}|\chi}(\phi)  
					\mod \pi $. Since $ \psi  \equiv \xi_{1} \xi_{2} \omega_p^{1-l}  \mod \pi$,  it follows that
					$
					E(\eta\chi^2 \cdot \psi^{-1}\chi^{-2} \cdot 
					(\phi_{p}^{-2}) \mathfrak{Z}_{p}^{k-l-2j-1}) 
					\equiv E(\eta\chi^2  \cdot \xi^{-1} \chi^{-2} \cdot  (\phi_{p}^{-2}) \mathfrak{Z}_{p}^{k-l-2j-1}) \mod \pi $.
					Thus
					we deduce  from (\ref{convolution mu and Eisenstien measure}) that
					$
					(\mu_{g^{\rho}|\chi}^{L} \star E)_{\eta\chi^2,k} (x_p^{j} \phi) \equiv 
					(\mu_{h^{\rho}|\chi}^{L} \star E)_{\eta\chi^2,k} (x_p^{j} \phi) \mod \pi
					$. 
					By \cite[Proposition 7.8]{Hidarankin2}, it follows that the linear map $ c(f_0|\chi) l_{f_{0}|\chi} $ preserves integrality. Since the  trace operator $Tr_{L/N} $ also preserves  integral structures  the lemma follows from \eqref{Def of convolution}.
				\end{proof}

				\begin{remark}
					In \cite[Theorem 1.2]{Delbergo}, it was shown that $\mu_{f \times g'} \equiv \mu_{f \times h'}$ whenever $g' \equiv h' \mod \pi$. Thus Lemma \ref{congruence of measures} is a consequence of \cite[Theorem 1.2]{Delbergo} applied with $g' = g|\bar{\chi}\iota_{p}$ and $h' = h|\bar{\chi}\iota_{p}$. Using Lemma~\ref{congruence of measures}, we establish our final congruence on $ p $-adic $L$-functions in Theorem~\ref{analytic final}.
				\end{remark}
						The following lemma relates the $p$-adic $L$-functions $\mu_{f|\chi \times h|\bar{\chi}}$ and $\mu_{f\times h}$.
						\begin{lemma}\label{twist and untwist cusp}
							Let $f \in S_{k}(\Gamma_{0}(N),\eta)$ be a normalised $p$-ordinary newform with $p\nmid N$. Let $h \in S_{l}(\Gamma_{0}(I_0p^\alpha),\psi)$ be a normalised $p$-ordinary newform with $p\nmid I_0$ and $2 \leq l < k$. Let $\phi$ be a finite order character  of $\mathbb{Z}_p^\times$  and $0\leq j \leq k-l-1$, we have 
							\begin{align}\label{eq:twist and untwist cusp}
								\mu_{f|\chi \times h|\bar{\chi}}(x_p^j \phi) = t' \frac{c(f_0|\chi)}{c(f_0)} \frac{\langle (f_0^\rho| \tau_{Np}, f_0\rangle_{Np}}{\langle (f_0|\chi)^\rho | \tau_{N_{f|\chi}p}, f_0|\chi \rangle_{N_{f|\chi}p}} \frac{W(h^\rho|\phi\chi)}{W(h^\rho|\phi)} \chi(p)^{1-\alpha} \mu_{f \times h}(x_p^j\phi)  \prod_{r \in \Sigma_0} L_{r}(l+j, f, h|\phi) ,
							\end{align} 
							where $t' =  (N_{f|\chi}/N)^{k/2} (N_{h|\chi}/I_0p^\alpha)^{(l+2j)/2} [N_{f|\chi}, N_{h|\chi}]/[N,I_0p^\alpha]$ and $L_{r}(l+j, f, h)$ denotes the Euler factor of $D_{INp}(l+j, f,h)$ at $r$. 
						\end{lemma}
						\begin{proof}
							Note that the span of $\{ \text{finite order characters of } \mathbb{Z}_p^\times\} \setminus \{ \bar{\psi}_{p} \}$ is dense in $C(\mathbb{Z}_p^\times, \bar{\mathbb{Q}}_p)$. Thus it is enough to show that the statement of lemma holds for all $\phi \in \{ \text{finite order characters of } \mathbb{Z}_p^\times\} \setminus \{ \bar{\psi}_{p}\}$. By Corollary~\ref{special value f,cusp}(i), we have
							\begin{align*}
								\mu_{f|\chi \times h|\bar{\chi}}(x_p^j \phi) = c(f_0|\chi) t(f|\chi, h|\bar{\chi}) p^{\alpha(l+2j)/2} (\chi(p)u_f)^{1-\alpha} & p^{(2-k)/2} W(h^\rho|\chi) P_{p}(h|\chi, p^j u_f \chi(p)) \\ & \times \frac{D_{N_{f|\chi}N_{h|\bar{\chi}}p}(l+j, f_0|\chi, h|\bar{\chi})}{(2i)^{k+l+2j} \pi^{l+2j+1} \langle (f_0|\chi)^\rho|\tau_{N_{f|\chi}p},f_0|\chi\rangle_{N_{f|\chi}p}}.
							\end{align*}
							Since $P_{p}(h|\chi, p^j u_f \chi(p))  = P_{p}(h, p^j u_f ) $ and twisting $f,h$ by $\chi$ doesn't change the value of $\alpha$ we have 
							\begin{align*}
								\mu_{f|\chi \times h|\bar{\chi}}(x_p^j \phi)  =  &\frac{t(f|\chi, h|\bar{\chi})}{t(f,h)}  \frac{c(f_0|\chi)}{c(f_0)} \frac{\langle f_0^\rho | \tau_{Np}, f_0\rangle_{Np}}{\langle (f_0|\chi)^\rho| \tau_{N_{f|\chi}p}, f_0|\chi \rangle_{N_{f|\chi}p}} \frac{W(h^\rho|\chi)}{W(h^\rho)} \chi(p)^{1-\alpha} \frac{D_{N_{f|\chi}N_{h|\bar{\chi}}p}(l+jf_0|\chi, h|\bar{\chi})}{D_{NI_0p}(l+j,f_0, h)}\\
								& \times t(f,h) c(f_0) p^{\alpha(l+2j)/2} u_f^{1-\alpha}  p^{(2-k)/2} W(h^\rho) P_{p}(h, p^j u_f )
								\frac{D_{NI_0p}(l+j, f_0, h)}{(2i)^{k+l+2j} \pi^{l+2j+1} \langle f_0^\rho|\tau_{Np},f_0|\chi\rangle_{Np}}.
							\end{align*}
							Now the lemma follows from Corollary~\ref{special value f,cusp}(i) in the case $\chi = \iota_{p}$. The remaining  case $\phi \neq  \iota_p, \bar{\psi}_{p}$ can be proved similarly using Corollary~\ref{special value f,cusp}(ii). 
						\end{proof}
						
						In the next section, using Eq. \eqref{eq:twist and untwist cusp}, we show that $\mu_{f|\chi \times h|\bar{\chi}}$ and the $\Sigma_0$-imprimitive $p$-adic $L$-function $\mu^{\Sigma_0}_{f \times h}$ generate the same ideal in the Iwasawa algebra.
						
						To ease the we denote $\tilde{f} = f|\iota_{m}$ and $\tilde{f}_0 = f_0|\iota_m$. Since the conductor of $\chi$ is a positive power of $m$, we have  $\chi\bar{\chi} = \iota_{m}$. We now relate the $p$-adic Rankin $L$-function $\mu_{f|\chi \times g|\bar{\chi}}$ to the product of the $L$-functions attached to $\tilde{f}_0 | \xi_{1}$ and $\tilde{f}_0 | \xi_{2}$. 
						\begin{lemma}\label{twist and untwist Eis}
							Let $f \in S_{k}(\Gamma_{0}(N),\eta)$ be a normalised $p$-ordinary newform with $p\nmid N$.  Let $g  = E_{l}(\xi_{1} \omega_p^{1-l}, \xi_2)$ and $g' = E_{l}(\xi_2,\xi_{1} \omega_p^{1-l})$. Set $ t(f|\chi,g|\bar{\chi}) = [N_{f|\chi},N_{g|\bar{\chi}}] N_{f|\chi}^{k/2} (N_{g|\bar{\chi}})_0^{(l+2j)/2} \Gamma(l+j)\Gamma(j+1)$.
							Let $ \phi $ be a finite order character on $ \mathbb{Z}_p^{\times} $ and $ 0 \leq j \leq k-l-1 $.
							\begin{enumerate}
								\item[$\mathrm{(i)}$] If $\phi = \iota_{p}$, then we have $\beta = 2 $ and 
								\begin{small}
									\begin{align*}
										\mu_{f|\chi \times g|\bar{\chi}}(x_p^{j} \phi ) = c(f_0|\chi) t(f|\chi,g|\bar{\chi}) p^{s(l+2j)/2} (\chi(p)u_{f})^{1-s} &  p^{(2-k)/2} W(g^\rho|\chi)   P_{p}(g,p^j u_{f}^{-1}) \\ 
										& \frac{L(j+1,\tilde{f}_0,\xi_2)    
											L(l+j,\tilde{f}_0,\xi_1 \omega_{p}^{1-l})}{(2i)^{k+l+2j} \pi^{l+2j+1} 
											{ \langle (f_{0}|\chi)^{\rho}|_{k}\tau_{N_{f|\chi}p} , f_{0}|\chi \rangle_{N_{f|\chi}p}}} ,
									\end{align*} 
								\end{small} %
								where $W(g^\rho|\chi) = (\mathrm{cond}(\xi_1\omega_{p}^{1-l}\bar{\chi
								})/\mathrm{cond}(\bar{\xi}_2\chi))^{l/2} \bar{\xi}_2\chi(-1) G(\bar{\xi_2}\chi)/ G(\xi_1\omega_{p}^{1-l} \bar{\chi})$.
								\item[$\mathrm{(ii)}$] If $\phi \neq \iota_{p}, (\xi_1 \omega_{p}^{1-l})_p$, then we have  $\beta = v_{p}(\mathrm{cond}(\xi_{1}\omega_{p}^{1-l}\phi)) + v_{p}(\mathrm{cond}(\xi_{2}\phi))$ and 
								\begin{small}
									\begin{align*}
										\mu_{f|\chi \times g|\bar{\chi}}(x_p^{j} \phi ) = c(f_0|\chi) t(f|\chi,g|\bar{\chi}) p^{\beta(l+2j)/2} (\chi(p)u_{f})^{1-\beta}   & p^{(2-k)/2} W(g^\rho|\chi\bar{\phi})   \\
										& \frac{L(j+1,\tilde{f}_0,\xi_2 \phi)    
											L(l+j,\tilde{f}_0,\xi_1 \omega_{p}^{1-l}\phi)}{(2i)^{k+l+2j} \pi^{l+2j+1} 
											{ \langle (f_{0}|\chi)^{\rho}|_{k}\tau_{N_{f|\chi}p} , f_{0}|\chi \rangle_{N_{f|\chi}p}}} ,
									\end{align*} 
								\end{small} %
								where $W(g^\rho|\chi\bar{\phi}) = (\mathrm{cond}(\xi_1 \omega_{p}^{1-l}\bar{\chi} \phi)/\mathrm{cond}(\xi_2 \bar{\chi} \phi))^{l/2} \xi_2\bar{\chi}\phi(-1) G(\bar{\xi}_2\chi\bar{\phi})/ G(\xi_1 \omega_{p}^{1-l}\bar{\chi}\phi)$. 
							\end{enumerate}
						\end{lemma}
						\begin{proof}
							Note that $u_{f|\chi} = \chi(p) u_f$ and $\chi$ has conductor prime to $p$. Since we are twisting $f,g$ and $\bar{\xi}_{1}\bar{\omega}_{p}\phi, \bar{\xi}_{2}\bar{\phi}$ by $\chi$, the $p$-part of the level and conductor remain the same.  Further, from \eqref{Euler factor at p}, we have $P_{p}(g^{\rho}|\chi,p^ju_{f|\chi}^{-1}) = P_{p}(g^{\rho},p^ju_{f}^{-1}) $.  Now the lemma follows immediately from Corollary~\ref{special value f,Eis}.
						\end{proof}
						\section{Periods and the congruences of p-adic L-functions}\label{sec: periods and congruence}
						
						In this section, we show that the $p $-adic Rankin-Selberg $ L $-function $ \mu_{f|\chi \times g|\bar{\chi} }$ is congruent to the product of the $p$-adic $ L $-functions of $ \tilde{f}_{0} \otimes \xi_{1}$ and $ \tilde{f}_{0} \otimes \xi_{2} $. In order to do this, we make an appropriate choice of periods in \S\ref{subsection: relation periods}. We also show that  the ideal generated by the $p$-adic $ L $-function $\mu_{f|\chi \times h|\bar{\chi} }$ is same as  the ideal generated by the $\Sigma_{0}$ imprimitive $p$-adic $L$-function of $\mu_{f \times h}$ in the Iwasawa algebra.
						
						\subsection{Relation between the periods}\label{subsection: relation periods}
						We begin by recalling  an algebraicity  result due to  Shimura.  
						
						\begin{theorem}\label{Shimura special values} $\mathrm{(}$\cite[Theorem 1]{Shimura1}$ \mathrm{ ) } $ Let $ F \in S_{k}(\Gamma_{1}(N))$  be a
							normalised	Hecke eigenform. There exist  complex periods $ \Omega_{F}^{+} $ and  $ \Omega_{F}^{-} $ such that  for every Dirichlet character $ \theta $,  we have
							\begin{small}		
								\begin{align}\label{eq: special values of f}
									\frac{L(j,F,\theta)}{(2 \pi i)^{j} G(\theta)
										\Omega_{F}^{\mathrm{sgn}((-1)^j \theta(-1))}}  \in \bar{\mathbb{Q}}, \quad \text{for} \quad 1 \leq j \leq k-1.
								\end{align}
							\end{small}%
							Further, we have  $ \frac{\Omega_{F}^{+} 
								\Omega_{F}^{-} }
							{2 G(\theta) \langle F , F \rangle} \in \bar{\mathbb{Q}}$ and $ \Omega_{F}^{\pm} $  are  well-defined up to an element of $ \bar{\mathbb{Q}}^{\times} $.
						\end{theorem}
						We next recall the $ p $-adic $ L $-function associated to  an eigenform. See  \cite[\S 14]{MTT} and \cite[\S 3.4.4]{SU}. 
						\begin{theorem}\label{p-adic l-function of modular form} 
							Let $ F = \sum a(n,F) q^n \in S_{k}(\Gamma_{0}(N), \varphi)$ be a $ p $-ordinary eigenform and $\Omega^{+}_{F}, \Omega^{-}_{F}$ be  complex periods satisfying \eqref{eq: special values of f}. Then there exist a bounded measure $ \mu_F $ such that for a finite order  character $ \phi $ of $(\mathbb{Z}/L\mathbb{Z}) \times \mathbb{Z}_{p}^{\times} $ with $p \nmid L$ and $ 0 \leq j \leq k-2 $,  we have
							\begin{small}
								\begin{align}\label{eq: interpol muF}
									\mu_{F}(x_{p}^{j}\phi ) = \int_{ (\mathbb{Z}/L\mathbb{Z}) \times \mathbb{Z}_{p}^{\times}} x_{p}^{j}\phi ~d\mu_{F} = \frac{e_{p}(u_F, x_p^j \phi )}{u_{F}^{v}}  
									\frac{ \mathrm{cond}(\phi)^{j+1}}{(-2\pi i)^{j+1} }\frac{j!}{G(\phi)} 
									\frac{L(j+1,F,\phi)}{
										\Omega_{F}^{\mathrm{sgn}((-1)^j \phi(-1))}  }.
								\end{align}
							\end{small}%
							Here $ e_{p}(u_F, x_p^j \phi ) = \left( 1 - 
							\frac{\phi_0(p) \varphi(p) p^{k-2-j}}{u_{F}}\right)
							\left(  1 - \frac{\bar{\phi}_0(p)p^{j}}{u_{F}}\right) $  is the $ p $-adic multiplier, $ \phi_0 $ is the primitive character associated to $ \phi $ and $ u_F $ is the unique $ p $-adic unit  root of $ X^2 - a(p,F)X+\varphi(p) p^{k-1} $.
						\end{theorem} 
						
						Observe that  for $ l \leq j \leq k-1 $, we have  
						$
						(-1)^{j} \xi_{2}\phi(-1) (-1)^{j-l+1} \xi_{1} \omega_p^{1-l}\phi(-1)  = (-1)^{l-1} \xi_{1}\xi_{2} \omega_p^{1-l} (-1) =(-1)^{l-1}(-1)^{l} = -1$. Hence $ \mathrm{sgn}((-1)^{j}\xi_{2}\phi(-1)) = -\mathrm{sgn}((-1)^{j-l+1} \xi_{1} \omega_p^{1-l}\phi(-1) )$. 
						Recall that $ f_0 $ is the $ p $-stabilization of the $p  $-ordinary newform $ f$ and $\tilde{f}_0 := f_0|\iota_{m}$. Thus, for $ l \leq j \leq k -1$, we have 
						\begin{small}
							\begin{align*}
								\frac{L(j,\tilde{f}_0, \xi_{2} \phi)}{(2 \pi i)^{j}G(\xi_{2}\phi) 
									\Omega_{\tilde{f}_0 }^{\pm} } \cdot \frac{L(j-l+1,\tilde{f}_0, \xi_{1}^{-1} \omega_p^{1-l}\phi)}
								{(2 \pi i)^{j-l+1} G(\xi_{1} \omega_p^{1-l}\phi) 
									\Omega_{\tilde{f}_0 }^{\mp} }   \in \bar{\mathbb{Q}}.
							\end{align*}
						\end{small}%
						In view of Lemma~\ref{twist and untwist Eis}, to compare the $p$-adic Rankin $L$-function $\mu_{f|\chi \times h|\bar{\chi}}$  with the product of $p$-adic $L$-functions corresponding to $\tilde{f}_0 \otimes \xi_1$ and $\tilde{f}_0 \otimes \xi_2$, it suffices to compare   $ c(f_0|\chi)/\langle (f_0|\chi)^\rho|\tau_{N_{f|\chi}p}, f_{0}|\chi \rangle_{N_{f|\chi}p} $ and $ \Omega_{\tilde{f}_{0}}^{+} \Omega_{\tilde{f}_{0}}^{-}  $. 
						Next, we make a choice of periods $ \Omega_{\tilde{f}_0}^{\pm} $ such that $ p $-adic $ L $-function $ L_{p}(\tilde{f}_0,\cdot) $ associated to $ \tilde{f}_0 $ in Theorem~\ref{p-adic l-function of modular form} has  integral power series expansion and $ c(f_0|\chi)\Omega_{\tilde{f}_{0}}^{+} \Omega_{\tilde{f}_{0}}^{-} /  \langle (f_0|\chi)^\rho|\tau_{N_{f|\chi}p}, f_{0}|\chi \rangle_{N_{f|\chi}p}   $ is a $ p $-adic unit.

						We  now recall relevant cohomology groups. Let $ \Gamma $ be a congruence subgroup. Let $A$ be an integral domain  and $ \mathcal{M} $ be an $ A[\Gamma] $-module. 
						Let $ \widetilde{\mathcal{M}}$ denote the sheaf of locally constant sections of the covering $
						\Gamma_0(N) \backslash (\mathbb{H} \times\mathcal{ M}) \rightarrow \Gamma_0(N) \backslash \mathbb{H}
						$. If $ \Gamma/\{\pm \mathrm{I}\} $ has no torsion elements or $ A $ is a field of characteristic zero, then  by \cite[Proposition 8.1]{Shimura book} there is a canonical isomorphism  
							$H^{1}(\Gamma,\mathcal{M}) \cong H^{1}(\Gamma \backslash \mathbb{H}, \widetilde{\mathcal{M}})$.
						Set $ X_{\Gamma} :=  \Gamma \backslash \mathbb{H} $
						and let $ {Y}_{\Gamma} $ be the Borel-Serre compactification of
						$ X_{\Gamma} $. Then we have
						$ \partial Y_{\Gamma} = 
						\sqcup_{\Gamma \backslash P^{1}(\Q)} S^{1}
						$
						and 
						$
						H^{i}(\partial Y_{\Gamma} , \widetilde{\mathcal{M}} )
						= \bigoplus_{s \in \Gamma_0(N) \backslash P^{1}(\Q)}  H^{i} ( \Gamma_{s}, \mathcal{M}),
						$
						where $ \Gamma_{s} $ denotes the stabilizer of $ s $ in $ \Gamma $.
						Define the parabolic cohomology 
						\begin{small}
							\begin{align*}
								H^{i}_{P}(X_{\Gamma} , \widetilde{\mathcal{M}}) &= \Ker\Big(\mathrm{res}: H^{i}(X_{\Gamma} , \widetilde{\mathcal{M}})  \longrightarrow H^{i}(\partial Y_{\Gamma} , \widetilde{\mathcal{M}} )\Big) \\
								H^{i}_{P}(\Gamma, \mathcal{M})  &= \Ker\Big(\mathrm{res}: H^{i}(\Gamma, \mathcal{M}) \longrightarrow \bigoplus\limits_{s \in \Gamma \backslash P^{1}(\Q)}  H^{i} ( \Gamma_s, \mathcal{M})\Big),
							\end{align*}
						\end{small}%
						where $\mathrm{res}$ are the corresponding restriction maps. Under the isomorphism $ H^{1}(\Gamma,\mathcal{M}) \cong H^{1}(X_\Gamma, \widetilde{\mathcal{M}}) $, we have $ H^{i}_{P}(X_{\Gamma} , \widetilde{\mathcal{M}}) \cong  H^{i}_{P}(\Gamma, \mathcal{M})  $. 
						There is a well-defined
						action of Hecke algebra  on the cohomology groups  $H^{i}(X_{\Gamma} , \widetilde{\mathcal{M}})$ and  $H^{i}_{P}(\Gamma, \mathcal{M}) $ (see \cite[\S8.3]{Shimura book} or \cite[\S 6.3]{hida_93}). 
						
						We now introduce the modules $ \mathcal{M} $ which are of interest for us. For a non-negative integer $n$, we denote by $L(n,A)$ the symmetric polynomial algebra over $A$ of degree $n$. Thus $L(n,A)$ consist of the homogeneous polynomials $P(X,Y)$ of degree $n$ in variables $ X $ and $ Y $, with coefficients in $A$. The semigroup  $\gl_2(\Q) \cap \text{M}_2(\Z)$ acts on $L(n,A)$  by
						\begin{small}
							\begin{align*}
								(\gamma \cdot P)(X,Y) = 
								P(aX+bY, cX+dY),
								\qquad \forall ~  \gamma = \begin{psmallmatrix} a & b \\ c & d \end{psmallmatrix},
							\end{align*}
						\end{small}%
						Note that the above action is a left action and is as
						defined in \cite[(8.2.1)]{Shimura book} and \cite[\S 3.3]{Kitagawa}. Note that Hida \cite[Page 165]{hida_93} instead considers the right action of $\gl_2(\Q) \cap \text{M}_2(\Z)$ on $ L(n,A) $.
						For a Dirichlet character $ \varphi $ modulo $N$, let $ L(n,\varphi;A) $ denote the $ \Gamma_{0}(N) $-module $ L(n,A) $ with the action 
						\begin{small}
							\[
							(\gamma  P)(X,Y) = \chi(\gamma)P(aX+bY, cX+dY) = \varphi(d) P(aX+bY, cX+dY),~\forall ~  \gamma = \begin{psmallmatrix} a & b \\ c & d \end{psmallmatrix} \in \Gamma_{0}(N).
							\]
						\end{small}%
						Now this action coincides with the action defined on \cite[Page 177]{hida_93}. To ease the notation, we denote $\widetilde{L(n,A)}$ (resp. $\widetilde{L(n,\varphi;A)}$)  by $ \mathcal{L}(n,A) $ (resp. $ \mathcal{L}(n,\varphi;A) $). 
						We have the   inclusion $H^{i}_{P}(\Gamma_0(N), L(n,\varphi;A)) \hookrightarrow H^{i}_{P}(\Gamma_1(N), L(n,A)) $ given by the restriction map. 

						Let $S_{k}(\Gamma)$ be the space of cusp forms on $ \Gamma $ with coefficients in $ \mathbb{C} $ and 
						$
						\bar{S}_{k}(\Gamma) = \{  \bar{ F} : F \in S_{k}(\Gamma) \}
						$ be the space of anti-holomorphic cusp forms on $ \Gamma$, where $\bar{F}(z) := \overline{F(z)}$. 
						For every $ F \in S_{k}(\Gamma)$ (resp. $\bar{S}_{k}(\Gamma) $), define an $L(n,\mathbb{C})$-valued differential $1$-form by 
							$\omega(F)(z) = 
							F(z)(zX+Y)^{k-2} d z$ 
							(resp.  $ \omega(F)(z) = F(z)(\bar{z}X+Y)^{k-2} d \bar{z}$). 
					With the action as above, it can be checked that $ \gamma^{\ast} \omega(F) = \det(\gamma)^{1-k/2} \gamma \cdot \omega(F\vert \gamma), ~ \forall  \gamma \in \gl_2(\Q) \cap \text{M}_2(\Z)$ (see \cite[\S 6.2]{hida_93}). 
					For every  $ F \in S_{k}(\Gamma)$ (resp. $ \bar{S}_{k}(\Gamma) $),
					set  $ \delta(F)(\gamma) :=
					\int_{ \infty}^{\gamma   \infty} \omega(F), ~\forall \gamma \in \Gamma$. 
					It can be checked that  $ \delta(F) $ is a $ 1 $ cocycle and $\delta: S_{k}(\Gamma) \oplus \bar{S}_{k}(\Gamma) \rightarrow H^{1}(\Gamma, L(k-2,\mathbb{C}) )$ is a homomorphism. 
					A similar construction, defines the map $\delta: S_{k}(\Gamma_0(N), \varphi) \oplus \bar{S}_{k}(\Gamma_0(N),\bar{\varphi}) \rightarrow H^{1}(\Gamma_0(N), L(k-2,\varphi;\mathbb{C}) )$. 
					\begin{theorem} $\mathbf{(}\textbf{Eichler-Shimura} \mathbf{)}  ~ \mathrm{(}$\cite[\S 6.2, Theorem 1]{hida_93}$\mathrm{)}$\label{Eichler Shimura} 
						With the notation as above. The maps $\delta 
						:S_{k}(\Gamma) \oplus \bar{S}_{k}(\Gamma) \rightarrow H^{1}_{P}(\Gamma, L(k-2,\mathbb{C})) \text{ and }
						\delta: S_{k}(\Gamma_0(N), \varphi) \oplus \bar{S}_{k}(\Gamma_0(N),\bar{\varphi})  \rightarrow
						H^{1}_{P}(\Gamma_0(N) , L(k-2,\varphi;\mathbb{C}) )$
						are Hecke equivariant isomorphisms.
					\end{theorem}
					If $ \varepsilon := \begin{psmallmatrix}  -1 & 0 \\ 0 & 1 \end{psmallmatrix} $ normalises $ \Gamma $ (resp. $\Gamma _{0}(N)$),  then $ \varepsilon $ acts on $ H^{i}_{P}( \Gamma , L(n, A) ) $ (resp. $ H^{i}_{P}( \Gamma_0(N) , L(n, \varphi;A) ) $). When $ A = \mathbb{C} $,  we obtain
					$ (\varepsilon \cdot \omega)(z) = \varepsilon\omega(\varepsilon z) = \varepsilon\omega(-\bar{z}) $ at the level of differential forms. Following \cite[\S 9]{SZ}, for a Dirichlet character $\varphi$ of conductor $D$ we define  $tw_{N,\varphi}:  H^{1}_{P}(\Gamma_0(N) , L(k-2,\eta;\mathbb{C}) ) \rightarrow H^{1}_{P}(\Gamma_0(ND^2) , L(k-2,\eta\varphi^2;\mathbb{C}) )$, given by 
					\begin{small}
						\begin{align}\label{twist char cohomology}
							tw_{N,\varphi}(\Psi)(\gamma) = \sum_{a=1}^{D} \varphi(a) \begin{psmallmatrix}
								1 & -a/D \\ 0 & 1
							\end{psmallmatrix} \Psi \Big(  \begin{psmallmatrix}
								1 & a/D \\ 0 & 1
							\end{psmallmatrix} \gamma \begin{psmallmatrix}
								1 & -a/D \\ 0 & 1
							\end{psmallmatrix}  \Big), \quad \forall ~ \gamma \in \Gamma_{0}(ND^2).
						\end{align}
					\end{small}
					Following is a generalisation of arguments in \cite[Lemma 9.6]{SZ} where the case $k=2$ and $\chi$ is a quadratic character.
					\begin{lemma}\label{lem: twist parabolic}
						Let $F \in S_{k}(\Gamma_{0}(N),\eta)$ be a normalised eigenform. Let $\varphi$ be  a  primitive Dirichlet character of conductor $D$. Then  we have 
						\begin{enumerate}
							\item[$\mathrm{(i)}$]  $
							tw_{N,\varphi}\delta(F)= G(\varphi) \delta(F|\bar{\varphi})$.
							\item[$\mathrm{(ii)}$]  $ tw_{N,\varphi} (\varepsilon\cdot\delta(F)) = G(\varphi) \varphi(-1) (\varepsilon\cdot\delta(F|\bar{\varphi}))$.
						\end{enumerate}	
						As a consequence, we have $tw_{N}(\delta(F) \pm \varepsilon\cdot\delta(F)) = G(\varphi) (\delta(F|\bar{\varphi}) \pm \varphi(-1) \varepsilon\cdot\delta(F|\bar{\varphi}))$. 	  
					\end{lemma}
					\begin{proof}
						To ease the notation denote $ \begin{psmallmatrix}
							1 & a/D \\ 0 & 1
						\end{psmallmatrix}   $ by $\gamma'_a$ .  Note that $\gamma'_a \cdot \infty = \infty = \gamma'_{-a} \cdot \infty$ and 
						\begin{align*}
							\delta(f)\Big(  \begin{psmallmatrix}
								1 & a/D \\ 0 & 1
							\end{psmallmatrix} \gamma \begin{psmallmatrix}
								1 & -a/D \\ 0 & 1
							\end{psmallmatrix}  \Big) =  \int_{\gamma'_a\infty}^{\gamma'_a \gamma\infty} \omega(f)(z)    = \int^{\gamma\infty}_{\infty}  ((\gamma'_a)^\ast \omega(f))(z) = \gamma_a'\int^{\gamma\infty}_{\infty}\omega(f|\gamma'_a)(z) = \gamma_a' \delta(f|\gamma'_a)(\gamma). 
						\end{align*}
						Now the assertion (i) follows from \cite[(8.3)]{MTT}.  To prove the assertion (ii) note that 
						\begin{align*}
							(\varepsilon \cdot (\delta(f))\Big(  \begin{psmallmatrix}
								1 & a/D \\ 0 & 1
							\end{psmallmatrix} \gamma \begin{psmallmatrix}
								1 & -a/D \\ 0 & 1
							\end{psmallmatrix}  \Big) =  \int_{\gamma'_a\infty}^{\gamma'_a \gamma\infty} (\varepsilon \cdot \omega(f))(z)    = \int^{\gamma\infty}_{\infty} ((\gamma'_a)^\ast (\varepsilon \cdot\omega(f)))(z).
						\end{align*}	
						Since $((\gamma'_a)^\ast (\varepsilon \cdot\omega(f)))(z) = \varepsilon \omega(f)(\varepsilon \gamma'_a z) = \varepsilon \omega(f)(\gamma'_{-a} \varepsilon z) = \gamma'_a \varepsilon \omega(f|\gamma'_{-a})(\varepsilon z) = \gamma'_a (\varepsilon \cdot \omega(f|\gamma'_{-a}))(z)$. Now the assertion (ii) follows from \cite[(8.3)]{MTT}.
					\end{proof}
					If $ 1/2 \in A $, then
					the map $ \omega \mapsto (\omega + \varepsilon \cdot \omega , \omega - \varepsilon \cdot \omega ) $ induces the decompositions $H^{i}_{P}( \Gamma_1(N), L(n,A) )   = 
					H^{i}_{P}( \Gamma_1(N), L(n,A) )^{+} \oplus H^{i}_{P}(\Gamma_1(N) , L(n,A) )^{-}$ and
					$H^{i}_{P}( \Gamma_0(N), L(n,\varphi;A) )   = 
					H^{i}_{P}( \Gamma_0(N), L(n,\varphi;A) )^{+} \oplus H^{i}_{P}(\Gamma_0(N) , L(n,\varphi;A) )^{-}$. For a homomorphism $ \lambda: h_{k}(\Gamma_1(N);A) \rightarrow  A$, define
					\begin{small} 
						\begin{align*}
							H^{i}_{P}(\Gamma_1(N) , L(n,A) )^{\pm} [\lambda] &= \{ x \in H^{i}_{P}(\Gamma_1(N) , L(n,A) )^{\pm} : x|T = \lambda(T) x, ~ \forall ~T \in  h_{k}(\Gamma_1(N);A) \}, \\
							H^{i}_{P}(\Gamma_0(N) , L(n,\varphi;A) )^{\pm} [\lambda] &= \{ x \in H^{i}_{P}(\Gamma_0(N) , L(n,\varphi;A) )^{\pm} : x|T = \lambda(T) x, ~ \forall ~T \in  h_{k}(\Gamma_0(N),\varphi;A) \}. 
						\end{align*}  
					\end{small} %
					Recall that for a number field $K$ and  a  prime $\mathfrak{p}$ in $K$ dividing $p$,  $\mathcal{O} = \mathcal{O}_{K_{\mathfrak{p}}}$ is the valuation ring  of $K_{\mathfrak{p}}$. Set $\mathcal{W} := \mathcal{O}_{K_{\mathfrak{p}}} \cap K$.
					
					\begin{lemma}$\mathrm{(}$\cite[\S 6.3, (11)]{hida_93}, \cite[Theorem 3.2]{Kitagawa}$\mathrm{)}$ \label{multiplicity one of parabolic cohomology}
						Let the notation be as above. For  a field $  K $ of characteristic zero,  we have 
						$ H^{i}_{P}(\Gamma_0(N) , L(n,\varphi;K) )^{\pm} $ $($resp. $H^{i}_{P}(\Gamma_1(N) , L(n,K) )^{\pm})$ is a free module of rank one  over $h_k(\Gamma_{0}(N),\varphi;K) $ $($resp. $h_{k}(\Gamma_1(N);K))$.
						As a consequence,  $ H^{i}_{P}(\Gamma_0(N) , L(n,\varphi;K) )^{\pm} [\lambda] = H^{i}_{P}(\Gamma_1(N) , L(n,K) ^{\pm})[\lambda]$ is a one dimensional vector space over $K$. Furthermore, $H^{i}_{P}(\Gamma_0(N) , L(n,\varphi;\mathcal{W}) ) \cap H^{i}_{P}(\Gamma_0(N) , L(n,\varphi;K) )^{\pm} [\lambda] $ has rank one over $\mathcal{W} $.
					\end{lemma}

					Suppose $ 1/n! \in A $. Define a pairing $ [ ~ , ~] $
					on $ L(n,\varphi;A) \times L(n,\bar{\varphi};A)  \rightarrow A$ by \cite[\S 6.2, (2a)]{hida_93} 
					\begin{small}
						\begin{align*}
							\Bigg[ \sum_{j=0}^{n} a_{j} X^{j}Y^{n-j}, \sum_{j=0}^{n} b_{j} X^{j}Y^{n-j} \Bigg] =  \sum_{j=0}^{n} (-1)^j \binom{n}{j}^{-1} a_j b_{n-j}.
						\end{align*} 
					\end{small}%
					Note that $ [(zX+Y)^n, (\bar{z}X+Y)^n] = \sum_{j=0}^{n} (-1)^j  \binom{n}{j} z^{n-j} \bar{z}^{j} = (z-\bar{z})^{n}$. The above pairing is perfect as  $n!$ is invertible in $A$. Write $X_N :=X_{\Gamma_{0}(N)}$ to ease the notation. By Poincar\'e duality, the above paring $ [ ~ , ~ ]$ induces a  pairing   (\cite[\S 6.2, (3a)]{hida_93}): 
					\begin{small}	
						\begin{align*}
							H^{1}_{c}(X_{N} , \mathcal{L}(n,\varphi;\mathbb{C}) ) \times 
							H^{1} (X_{N} , \mathcal{L}(n,\bar{\varphi};\mathbb{C})) \xrightarrow[]{\cup} H^{2}_{c}(X_{N} , \mathcal{L}(n,\varphi;\mathbb{C}) \times \mathcal{L}(n,\bar{\varphi};\mathbb{C}))
							\xrightarrow[]{[ ~ , ~ ]} H^{2}_{c}(X_{N} , \mathbb{C} )\xrightarrow[]{\simeq} \mathbb{C},
						\end{align*}
					\end{small}%
					where the first map is the wedge product and the last map is integrating
					the $2$-form on $ X_N$. We continue to denote this paring  by $[ ~ , ~]$.
					Further, for  $x \in  H^{1}(\Gamma_0(N), L(k-2,\varphi;\mathbb{C}) ) $, define $ (x|\tau_N) (\gamma) := \tau_N \cdot x(\tau_N \gamma \tau_N^{-1})  $, for all $ \gamma \in \Gamma_0(N) $. This in turn defines an action of $ \tau $ on $ H^{1} (X_{N} , \mathcal{L}(n,\varphi;\mathbb{C})) \cong H^{1}(\Gamma_0(N), L(n,\varphi;\mathbb{C})) $. Consider the pairing  $\langle x, y \rangle :=  [x, y|\tau_N] = [x|\tau_N, y] $, where $ \tau_N = \begin{psmallmatrix}  0 & -1 \\ N & 0 \end{psmallmatrix}  $.  
					We need the following version of \cite[Theorem 5.16]{Hida3}:
					\begin{theorem}\label{choice of period and petterson innerproduct}
						Let $ F(z) = \sum a(n,F) q^n \in S_{k}(\Gamma_{0}(N),\varphi; K)$ be a normalised Hecke eigenform and $a(n,F) \in \mathcal{O}$. Let $\phi_{F}: h_{k}(\Gamma_{0}(N),\varphi;K) \rightarrow K$ be the homomorphism induced by $T(n) \mapsto a(n,F)$.
						Fix a generator $ \xi^{\pm} $ of $ H^{1}_{P}(\Gamma_{0}(N), L(k-2, \varphi; \mathcal{W})) \cap H^{i}_{P}(\Gamma_0(N) , L(n,\varphi;K) )^{\pm} [\phi_F] $ over $\mathcal{W}$. Then there exist complex periods $ \Omega_{F}^{\pm} $ satisfying \eqref{eq: special values of f}, such that 
						\begin{enumerate}[label=$\mathrm{(\roman*)}$]
							\item $ \delta(F)^{\pm} : = \delta(F) \pm \varepsilon \cdot \delta(F)  = \Omega_{F}^{\pm} \xi^{\pm}$.
							\item Further, we have 
								$2^{k} i^{k+1} N^{k/2-1} \langle F^\rho|\tau_N , F \rangle_N =  \Omega_{F}^{+}\Omega_{F}^{-}  \langle \xi^{+}, \xi^{-}  \rangle$. 
						\end{enumerate}
						
					\end{theorem}
					\begin{proof}
						The first assertion follows from  Lemma~\ref{multiplicity one of parabolic cohomology} and the fact that $  H^{1}_{P}\big(\Gamma_0(N) ,L(k-2,\varphi;\mathbb{C}) \big)^\pm =  H^{1}_{P}\big(\Gamma_0(N), L(k-2,\varphi;\mathcal{W}) \big)^\pm  \otimes \mathbb{C}$.
						Note that $ (\delta(F)| \tau_{N}) 
						=  (-1)^{k-2} N^{k/2-1}  \delta(F| \tau_{N})$
						and $ (\varepsilon \cdot \omega(F)) (z) =  \varepsilon \omega(F)(-\bar{z}) =  - F(-\bar{z}) (\bar{z} X + Y)^{k-2} d \bar{z}  =  - \bar{F}^{\rho}(z) (\bar{z} X + Y)^{k-2} d \bar{z} $. The assertion (ii) was proved in  \cite[Theorem 5.16]{Hida3} when  $ F $ is primitive. The proof extends to this case, we omit the details. 
					\end{proof}
					\begin{remark}
						The periods $\Omega^\pm_{F}$ chosen in Theorem~\ref{choice of period and petterson innerproduct} are some times refereed as canonical periods (cf. \cite[Page 187]{hida_93}). 
					\end{remark}
					We next show that the periods chosen in Theorem~\ref{choice of period and petterson innerproduct}  have the property that  the measure $ \mu_{F}(\cdot) $ in Theorem~\ref{p-adic l-function of modular form} is $\mathcal{O}$-valued.  To do this, it suffices to  show that the modular symbol attached to  $ F $, when multiplied by $ 1/\Omega_{F}^{\pm} $, is integral (see \cite{MTT}, \cite{Kitagawa}). 
					
					Let $ \Delta =  $ Div$(P^{1}(\Q))  $ denote the group of divisors generated on  $ \mathbb{P}^{1}(\Q) $. Let $ \Delta_0 $ denote the subgroup of $ \Delta $ consisting of divisors of degree zero. Recall that $ A $ is a ring with $ 1/2 \in A $ and $ \mathcal{M} $ is an $ A[\Gamma]$-module. For $ \gamma \in \gl_2(\mathbb{Q}) \cap \mathrm{M}_{2}(\mathbb{Z}) $ and $ \{r\} \in \mathbb{P}^{1}(\mathbb{Q}) $, define 
					\begin{small}
						\begin{align*}
							(\gamma \cdot \Psi) (\{r\}) = \gamma \Psi(\{ \gamma^{-1} r\} ), \quad \forall ~ 
							\Psi \in \Hom_{\mathbb{Z}}(\Delta, \mathcal{M}) \text{ or } \Hom_{\mathbb{Z}}(\Delta_0, \mathcal{M}).
						\end{align*}
					\end{small}%

					Let  Symb$_{\Gamma}(\mathcal{M}) := \text{Hom}_{\Gamma}(\Delta_0,\mathcal{M})$ be the group of modular symbols and  BSymb$_{\Gamma}(\mathcal{M}) := \text{Hom}_{\Gamma}(\Delta,\mathcal{M})$ be the group of boundary modular symbols. There is a natural restriction map $\mathrm{res}:$ BSymb$_{\Gamma}(\mathcal{M})  \rightarrow$ Symb$_{\Gamma}(\mathcal{M})$.  
					There is a well-defined action of Hecke algebra on Symb$_{\Gamma}(\mathcal{M})$ (see  \cite[\S 4]{MTT}). 
					By \cite[Theorem 4.3]{gs} (see also \cite[(9)]{Vatsal}), we have the following exact sequence of Hecke modules
					\begin{small}
						\begin{align}\label{seq:parabolic and symbols}
							0 \rightarrow H^{0}(\Gamma, L(n,A)) \rightarrow \text{BSymb}_{\Gamma}(L(n, A))
							\xrightarrow{\text{res}} \text{Symb}_{\Gamma}(L(n, A)) \xrightarrow {\Theta}
							H^{1}_{P}(\Gamma , L(n,A)) \rightarrow 0.
						\end{align}
					\end{small}%
					Following  \cite{gs}, \cite{Kitagawa} and \cite{Vatsal}  we now show in Theorem~\ref{period and integral measure} that the measure in Theorem~\ref{p-adic l-function of modular form} is $ \mathcal{O} $-valued with  the choice of periods $ \Omega_{F}^{\pm} $  as specified in Theorem~\ref{choice of period and petterson innerproduct}. 
					
					\begin{theorem}\label{period and integral measure}
						Let $ F \in S_{k}(\Gamma_{0}(N), \varphi;\mathcal{O})$  be a normalised $ p $-ordinary eigenform with the  residual  Galois representation at $ p $ is irreducible. 
						Let   $ \Omega_{F}^{\pm} $  be the periods as defined in Theorem~\ref{choice of period and petterson innerproduct}. Then with these choices of periods $ \Omega_{F}^{\pm} $  the measure  $\mu_F$ satisfies the interpolation property~\eqref{eq: interpol muF} i.e. 
						for  every finite order character $ \phi  $ of $ (\mathbb{Z}/L\mathbb{Z})^\times \times \mathbb{Z}_{p}^{\times}$ with $ p \nmid L $ and $ 0 \leq j \leq k-1 $,  we have 
						\begin{small}	
							\begin{align*}
								\mu_{F}(x_p^j \phi) = \int_{  (\mathbb{Z}/L\mathbb{Z})^\times \times \mathbb{Z}_{p}^{\times} } x_p^j \phi =\frac{1}{u_{F}^{v}}  \left( 1 - 
								\frac{\phi(p) \varphi(p) p^{k-2-j}}{u_{F}}\right)
								\left(  1 - \frac{\bar{\phi}(p)p^{j}}{u_{F}}\right) 
								\frac{ \mathrm{cond}(\phi)^{j+1}}{(-2\pi i)^{j+1} } \frac{j!}{G(\phi)} 
								\frac{L(j+1,F,\phi)}{
									\Omega_{F}^{\mathrm{sgn}((-1)^j \phi(-1))}  }.
							\end{align*}
						\end{small}%
						Furthermore, $\mu_F$ is an $\mathcal{O}$-valued measure.
					\end{theorem}
					\begin{proof}
						By \cite[1.6]{Vatsal}, there exist $ \Delta_F^\pm  \in \text{Symb}_{\Gamma_{1}(N)}(L(k-2,\mathcal{O}))$ with $\Theta(\Delta^{\pm}(F)) = \delta(F)^{\pm}/\Omega_{F}^\pm$.  
						Let $ \mu $ be the measure as defined in \cite[(4.16),(4.17)]{gs} (See  \cite[\S 4.2]{Kitagawa}), attached to these modular symbols $ \Delta_F^{\pm}$. By \cite[(11)]{Vatsal} and  \cite[Theorem 4.18]{gs} (See also \cite[Theorem 4.8]{Kitagawa}), it follows that $\mu$ satisfies the interpolation formula given in \eqref{eq: interpol muF}.  Hence  the measure $\mu_{F}$ obtained in Theorem~\ref{p-adic l-function of modular form} equals $\mu $.
						That the  measure $ \mu_{F} $ is an $ \mathcal{O} $-valued measure follows from \cite[Lemma 4.3]{Kitagawa}. 
					\end{proof}
					
					
					\begin{definition}\textbf{(Module of congruences)}
						Let $ R $ be a finite, flat  and reduced  algebra over $ A $ and $B =$ Frac$(A)$. 
						Moreover, assume that we are given
						a map $ \lambda: R \rightarrow A$ such that it induces  an $A$-algebra decomposition
						$R \otimes_{A} B \cong B \oplus X$.
						Let  $ 1_{\lambda} $ be the idempotent {corresponding} to the first summand. Put $\mathfrak{a}$ = Ker($R \rightarrow X) = 1_\lambda R \cap R$, $S$ = Im$(R \rightarrow X)$ and $\mathfrak{b}$ = ker($\lambda$). The module of congruences $ C_0(\lambda) $ is defined by 
						\begin{align*}
							C_0(\lambda)  = (R/\mathfrak{a}) \otimes_{R,\lambda} \mathcal{O}_{K} \cong \frac{R}{\mathfrak{a}\oplus \mathfrak{b}} \cong \lambda(R)/\lambda(\mathfrak{a}) \cong
							1_{\lambda} R/\mathfrak{a} \cong S/\mathfrak{b} .
						\end{align*}
					\end{definition}
					If $R$ is Gorenstein i.e.  $\text{Hom}_A(R,A) \cong R$  as $R$-modules,  then 
					$ C_0(\lambda) = A/c(\lambda) $ for some $c(\lambda) \in A$ (cf. \cite[Lemma 6.9]{Hidaadjoint}). In particular, we have $c(\lambda)1_{\lambda} \in R$. Note that $c(\lambda) $ is well defined up to a  unit in $A$.

					For $ F \in S_k(\Gamma_{0}(N),\varphi) $ and $ \Omega_{F}^\pm $ as in Theorem~\ref{period and integral measure},  we now determine $ \langle F^\rho|\tau_{N} , F \rangle_N/\Omega_F^{+}\Omega_F^{-}  $ explicitly up to a $ p $-adic unit. 
					
					Recall that if $ F \in  S_{k}(\Gamma_{0}(N),\varphi;K) $ is either a normalised $ p $-ordinary newform or $p$-stabilization of a normalised $p$-ordinary newform, then the splitting $ h_{k}(\Gamma_{0}(N), \varphi; K_{\mathfrak{p}}) \cong K_{\mathfrak{p}} \oplus X$  holds (See \eqref{splitting of hecke algebra}) and we have the following result due to Hida (cf. \cite[Theorem 5.20]{Hida3}).
					\begin{theorem}\label{c(f) and petterson} Let $ p $ be odd  and $ k \geq 3 $.  Let $ F \in     S_{k}(\Gamma_{0}(N),\varphi;K) $ be either a normalised $ p $-ordinary newform or $p$-stabilization of a normalised $p$-ordinary newform. We assume $ p >3 $, if $ \Gamma_{0}(N)/\{\pm1\} $ has non-trivial torsion elements. Let $ \bar{\rho}_{F} $  be the residual representation  of $ F $ and $  \Omega_{F}^{\pm} $ be as in Theorem~\ref{period and integral measure}.  If $F$ satisfies \ref{irr-f} and \ref{p-dist}, we have 
						\begin{small}
							\begin{align*}
								c(F) \frac{\Omega_{F}^{+}\Omega_{F}^{-}}{2^{k} i^{k+1}N^{k/2-1} \langle F^\rho|\tau_{N} , F \rangle_N} \text{ is a } p\text{-adic unit}.
							\end{align*}
						\end{small}%
					\end{theorem}

					We next compare the periods $\Omega_{\tilde{f}_{0}}^{+} \Omega_{\tilde{f}_{0}}^{-}$ and $\langle (f_0|\chi)^\rho|\tau_{N_{f|\chi}p}, f_0|\chi\rangle_{N_{f|\chi}p}$.
					\begin{lemma}\label{periods comparision Eis}
						Let $f \in S_{k}(\Gamma_{0}(N),\eta)$ be a $p$-ordinary newform and $\tilde{f}_0 = f_0|\iota_{m}$. Assume $f$ satisfies \ref{p-dist} and \ref{irr-f}. Then the following quantity
						\begin{align*}
							\frac{p^{(2-k)/2}c(f_0|\chi) \Omega_{\tilde{f}_{0}}^{+} \Omega_{\tilde{f}_{0}}^{-} }{\langle (f_0|\chi)^\rho|\tau_{N_{f|\chi}p}, f_0|\chi\rangle_{N_{f|\chi}p} } \text{ is a } p\text{-adic unit}.
						\end{align*} 
					\end{lemma}
					\begin{proof}
						Applying Theorem~\ref{c(f) and petterson} for $f_0|\chi $, we have  $\langle (f_0|\chi)^\rho|\tau_{N_{f|\chi}p}, f_0\rangle_{N_{f|\chi}p}  = p^{(2-k)/2} c(f_0|\chi)\Omega^{+}_{f_0|\chi} \Omega_{f_{0}|\chi}^{-}$ up to a $p$-adic unit. So we need to  show that  $\Omega^{+}_{f_0|\chi} \Omega_{f_{0}|\chi}^{-}$ differs from $\Omega^{+}_{\tilde{f}_0} \Omega_{\tilde{f}_{0}}^{-}$ by a $p$-adic unit. To ease the notation, denote the conductor of $\chi$ by $C_{\chi}$. Let 
						$
						tw_{N_{f|\chi}p,\chi}: H^{1}_P(\Gamma_0(N_{f|\chi}p), L(k-2, \eta\chi^2;\mathcal{W}))[\phi_{f_0|\chi}] \rightarrow H^{1}_P(\Gamma_0(N_{f|\chi}C_\chi^2p), L(k-2,\eta \iota_{m};\mathcal{W}))
						$
						be as defined in \eqref{twist char cohomology}. By Theorem~\ref{choice of period and petterson innerproduct}, we have $ \delta(f_0|\chi)^{\pm}/\Omega_{f_0|\chi}^\pm$ is a basis of   $H^{1}_P(\Gamma_0(N_{f|\chi}p), L(k-2,\eta\chi^2;\mathcal{W})) \cap H^{1}_P(\Gamma_0(N_{f|\chi}p), L(k-2,\eta\chi^2;K)) [\phi_{f_0|\chi}]$. Also by Lemma~\ref{lem: twist parabolic}, we have $tw_{N_{f|\chi}p,\chi}(\delta(f_0|\chi)^{\pm}/\Omega_{f_0|\chi}) = G(\chi) \delta(\tilde{f}_0)^{\pm \chi(-1)} /\Omega_{f_0|\chi}^\pm$. Thus the image of $ H^{1}_P(\Gamma_0(N_{f|\chi}p), L(k-2,\eta\chi^2;\mathcal{W})) \cap  H^{1}_P(\Gamma_0(N_{f|\chi}p), L(k-2,\eta\chi^2;K))^\pm [\phi_{f_0|\chi}]$ under $tw_{N_{f|\chi}p,\chi}$ lies inside $H^{1}_P(\Gamma_0(N_{\tilde{f}}p), L(k-2,\eta \iota_{m};\mathcal{W})) \cap H^{1}_P(\Gamma_0(N_{\tilde{f}}p), L(k-2,\eta \iota_{m};K))[\phi_{\tilde{f}_0}]$.

						Next we consider the map $tw_{N_{\tilde{f}}p, \bar{\chi}}: H^{1}_P(\Gamma_0(N_{\tilde{f}}p), L(k-2,\eta\iota_m;\mathcal{W})) \cap  H^{1}_P(\Gamma_0(N_{\tilde{f}}p), L(k-2,\eta\iota_m;K))[\phi_{\tilde{f}_0}]$ $ \rightarrow H^{1}_P(\Gamma_0(N_{\tilde{f}}C_{\chi}^2p), L(k-2,\eta \chi^2;\mathcal{W})) $. From Lemma~\ref{lem: twist parabolic},  it follows that $tw_{N_{\tilde{f}}p, \bar{\chi}}(\delta(\tilde{f}_0)^{\pm \chi(-1)} /\Omega_{f_0|\chi}) = tw_{N_{\tilde{f}}p, \bar{\chi}}(\delta(\tilde{f}_0)^{\pm \chi(-1)} )/\Omega_{f_0|\chi}^\pm = G(\bar{\chi})\delta(f_0|\chi)^{\pm} /\Omega_{f_0|\chi}^\pm$ and  image of $H^{1}_P(\Gamma_0(N_{\tilde{f}}p), L(k-2,\eta\iota_m;K))[\phi_{\tilde{f}_0}] \cap H^{1}_P(\Gamma_0(N_{\tilde{f}}p), L(k-2,\eta\iota_m;\mathcal{W}))$ under $tw_{N_{\tilde{f}}p, \bar{\chi}}$ lies in $H^{1}_P(\Gamma_0(N_{f|\chi}p), L(k-2,\eta\chi^2;\mathcal{W}))[\phi_{f_0|\chi}]$. Thus $tw_{N_{\tilde{f}}p, \bar{\chi}} \circ tw_{N_{f|\chi}p, \chi} (\delta(f_0|\chi)^{\pm} /\Omega_{f_0|\chi}^\pm) = G(\chi) G(\bar{\chi}) \delta(f_0|\chi)^{\pm} /\Omega_{f_0|\chi}^\pm $. Since $G(\chi) G(\bar{\chi}) =  \chi(-1) C_\chi$ is prime to $p$ and $\delta(f_0|\chi)^{\pm} /\Omega_{f_0|\chi}^\pm $ generates $H^{1}_P(\Gamma_0(N_{f|\chi}p), L(k-2,\eta\chi^2;\mathcal{W})) \cap H^{1}_P(\Gamma_0(N_{f|\chi}p), L(k-2,\eta\chi^2;K))[\phi_{f|\chi}]^\pm$, we get 
						$tw_{N_{\tilde{f}}p, \bar{\chi}}$ is  surjective. As the $\mathcal{W}$-modules, $H^{1}_P(\Gamma_0(N_{f|\chi}p), L(k-2,\eta\chi^2;\mathcal{W})) \cap H^{1}_P(\Gamma_0(N_{f|\chi}p), L(k-2,\eta\chi^2;K))[\phi_{f|\chi}]$ and $H^{1}_P(\Gamma_0(N_{\tilde{f}}p), L(k-2,\eta\iota_m;\mathcal{W}))\cap H^{1}_P(\Gamma_0(N_{\tilde{f}}p), L(k-2,\eta\iota_m;K))[\phi_{\tilde{f}_0}]$ have rank $2$, we get $tw_{N_{\tilde{f}}p, \bar{\chi}}$ is an isomorphism. Since $(\delta(\tilde{f}_0)^{\pm}/\Omega_{\tilde{f}_0}^{\pm})$ is a basis of $H^{1}_P(\Gamma_0(N_{\tilde{f}}p), L(k-2,\eta\iota_m;\mathcal{W}))\cap H^{1}_P(\Gamma_0(N_{\tilde{f}}p), L(k-2,\eta\iota_m;K))[\phi_{\tilde{f}_0}]$, we get $tw_{N_{\tilde{f}}p, \bar{\chi}}(\delta(\tilde{f}_0)^{\pm} /\Omega_{\tilde{f}_0}^\pm) = G(\bar{\chi})\delta(f_0|\chi)^{\pm\chi(-1)} /\Omega_{\tilde{f}_0}^\pm$ is a  $\mathcal{W}$-basis of $H^{1}_P(\Gamma_0(N_{f|\chi}p), L(k-2,\eta\chi^2;\mathcal{W})) \cap H^{1}_P(\Gamma_0(N_{f|\chi}p), L(k-2,\eta\chi^2;K))[\phi_{f|\chi}]$. Since $\delta(f_0|\chi)^\pm/\Omega_{f_0|\chi}^{\pm}$ is also a $\mathcal{W}$-basis of $H^{1}_P(\Gamma_0(N_{f|\chi}p), L(k-2,\eta\chi^2;\mathcal{W})) \cap H^{1}_P(\Gamma_0(N_{f|\chi}p), L(k-2,\eta\chi^2;K))[\phi_{f|\chi}]$, we get $\Omega_{\tilde{f}_0}^+\Omega_{\tilde{f}_0}^- $ is equal to $\Omega_{f_0|\chi}^+\Omega_{f_0|\chi}^-$ up to a $p$-adic unit. This proves the lemma. 
					\end{proof}
					\begin{lemma}\label{period comparision cusp}
						Let $f \in S_{k}(\Gamma_{0}(N),\eta)$ be a $p$-ordinary newform and $\tilde{f}_0 = f_0|\iota_{m}$. Assume $f$ satisfies \ref{p-dist} and \ref{irr-f}. Then the following quantity
						\begin{align*}
							\frac{c(f_0|\chi)}{c(f_0)} \frac{\langle f_0^\rho | \tau_{Np}, f_0\rangle_{Np}}{\langle (f_0|\chi)^\rho | \tau_{N_{f|\chi}p}, f_0|\chi \rangle_{N_{f|\chi}p}} \text{ is a } p\text{-adic unit}.
						\end{align*} 
						
					\end{lemma}
					\begin{proof}
						Applying Lemma~\ref{c(f) and petterson}, the stated ratio equals $\Omega_{f_0}^{+} \Omega_{f_0}^{-}/ \Omega_{f_{0}|\chi}^{+}\Omega_{f_{0}|\chi}^{-}$ up to a $p$-adic unit. Consider the following maps
						\begin{align*}
							&tw_{Np,\bar{\chi}}: H^{1}_P(\Gamma_0(Np), L(k-2,\eta;\mathcal{W})) \rightarrow H^{1}_P(\Gamma_0(NC_{\chi}^2p), L(k-2,\eta\chi^2;\mathcal{W})) \\
							&tw_{N_{f|\chi}p,\chi}: H^{1}_P(\Gamma_0(N_{f|\chi}p), L(k-2,\eta\chi^2;\mathcal{W})) \rightarrow H^{1}_P(\Gamma_0(N_{f|\chi}C_{\chi}^2p), L(k-2,\eta\iota_m;\mathcal{W})).
						\end{align*}		
						As in Lemma~\ref{periods comparision Eis}, we can show 
						\begin{alignat*}{2}
							&tw_{Np,\bar{\chi}}: H^{1}_P(\Gamma_0(Np),L(k&&-2,\eta;\mathcal{W})) \cap H^{1}_P(\Gamma_0(Np), L(k-2,\eta;K))[\phi_{f_0}] \rightarrow \\
							&  &&H^{1}_P(\Gamma_0(N_{f|\chi}p), L(k-2,\eta\chi^2;\mathcal{W})) \cap H^{1}_P(\Gamma_0(N_{f|\chi}p), L(k-2,\eta\chi^2;K))[\phi_{f_0|\chi}], \\
							&tw_{N_{f|\chi}p,\chi}: H^{1}_P(\Gamma_0(N_{f|\chi}p)&&, L(k-2,\eta\chi^2;\mathcal{W}))\cap H^{1}_P(\Gamma_0(N_{f|\chi}p), L(k-2,\eta\chi^2;K))[\phi_{f_0|\chi}]  \rightarrow \\ 
							&  &&H^{1}_P(\Gamma_0(N_{\tilde{f}_0}p), L(k-2,\eta\iota_m;\mathcal{W}))\cap H^{1}_P(\Gamma_0(N_{\tilde{f}_0}p), L(k-2,\eta\iota_m;K)[\phi_{\tilde{f}_0}].
						\end{alignat*} 
						are injective and $tw_{N_{f|\chi}p,\chi}$ is an isomorphism.  Further by Lemma~\ref{lem: twist parabolic}, we have  $tw_{Np,\bar{\chi}}(\delta(f_0)^\pm) = G(\bar{\chi}) (\delta({f}_0|\chi))^{\pm\chi(-1)}$ and  $tw_{N_{f|\chi}p,\chi}((\delta({f}_0|\chi))^\pm) = G(\chi) \delta({\tilde{f}}_0)^{\pm\chi(-1)}$ (after extending scalars to $\C$).  Thus we have $tw_{N_{f|\chi}p,\chi} \circ tw_{Np,\bar{\chi}} (\delta(f_0)^\pm) = G(\chi)G(\bar{\chi})\delta({\tilde{f}}_0)^{\pm}$. Note that for every prime $r$, we have
						\begin{align*}
							f_0|\iota_{r} = \begin{cases}
								f_0 - (f_0|T(r))|[r] + r^{k-1} \eta(r)f_0|[r^2] & \text{ if } r \nmid Np, \\
								f_0 - (f_0|T(r))|[r] & \text{ if } r \mid Np.
							\end{cases}
						\end{align*}
						For $p\nmid r$, let $[r]:\Gamma_{0}(Nr) \rightarrow \Gamma_{0}(N)$ be map given by $\gamma \mapsto \begin{psmallmatrix}
							r & 0 \\0 & 1
						\end{psmallmatrix} \gamma \begin{psmallmatrix}
							r & 0 \\0 & 1
						\end{psmallmatrix}^{-1} $ and $[r]^\ast: H^{1}(\Gamma_{0}(N), L(k-2,\eta;\mathcal{W})) \rightarrow  H^{1}(\Gamma_{0}(Nr), L(k-2,\eta;\mathcal{W}))$ be the corresponding map on cohomology. Since  $tw_{N_{f|\chi}p,\chi} \circ tw_{Np,\bar{\chi}}(\delta(f_0)^\pm) =  G(\chi)G(\bar{\chi}) \delta({\tilde{f}}_0)^{\pm}$ and $\{ \delta(f_0)^+/\Omega_{f_0}^+, \delta(f_0)^-/\Omega_{f_0}^-\}$ is a $\mathcal{W}$-basis of  $H^{1}_P(\Gamma_0(Np), L(k-2,\eta;\mathcal{W})) \cap H^{1}_P(\Gamma_0(Np), L(k-2,\eta;K))[\phi_{f_0}]$, we get $tw_{N_{f|\chi}p,\chi} \circ tw_{Np,\bar{\chi}} = G(\chi)G(\bar{\chi}) \prod_{r \mid m, r \mid Np} (1- T(r)\circ[r]^\ast) \prod_{r \mid m, r \nmid Np} (1- T(r)\circ[r]^\ast + r^{k-1} \eta(r) [r^2]^\ast )$.  In the case $k=2$, the maps $(1- T(r)\circ[r]^\ast)$ for $r\mid Np$ and $(1- T(r)\circ[r]^\ast + r^{k-1} \eta(r) [r^2]^\ast )$ for $r\nmid Np$ was shown to be isomorphism by Wiles (cf. \cite[\S 4.4]{DDT}).  The arguments of \cite[\S 4.4]{DDT} generalise for the case  $k\geq 3$ and we get, $tw_{N_{f|\chi}p,\chi} \circ tw_{Np,\bar{\chi}}$ is an isomorphism for $k\geq 3$ (cf. \cite[Proposition 1.4]{DFG}). Hence $tw_{N_{f|\chi}p,\chi} \circ tw_{Np,\bar{\chi}}$ is an isomorphism. Thus $\Omega_{\tilde{f}_0}^{+} \Omega_{\tilde{f}_0}^{-}$ is equal to $\Omega_{f_0}^{+} \Omega_{f_0}^{-} $ up to a $p$-adic unit.
					\end{proof}
					We note the following consequence which was obtained in the course of the proof of above two lemmas.
					\begin{corollary}\label{periods comparision}
						Let $f \in S_{k}(\Gamma_{0}(N),\eta)$ be a $p$-ordinary newform and $\tilde{f}_0 = f_0|\iota_{m}$. Assume $f$ satisfies \ref{p-dist} and \ref{irr-f}. Then 
							$  \Omega_{\tilde{f}_0}^{+} \Omega_{\tilde{f}_0}^{-} / \Omega_{f_0|\chi}^{+} \Omega_{f_0|\chi}^{-} $ and $\Omega_{f_0|\chi}^{+} \Omega_{f_0|\chi}^{-} / \Omega_{f_0}^{+} \Omega_{f_0}^{-}$
						are $p$-adic units. 
					\end{corollary}
					\subsection{Congruences of the \texorpdfstring{$p$}{}-adic \texorpdfstring{$L$}{}-functions}\label{sec: Congruences of p-adic L functions}
					Throughout  \S\ref{sec: Congruences of p-adic L functions}, we choose the periods $ \Omega_{\tilde{f}_0}^{\pm} $ as  in Theorem~\ref{choice of period and petterson innerproduct}  and  the measure $ \mu_{\tilde{f}_0}( \cdot) $  as defined in Theorem~\ref{period and integral measure}. In Theorem~\ref{analytic final}, we prove our main result on the congruences of the  $ p $-adic $ L $-functions. Using Lemma~\ref{twist and untwist Eis} and the $ p $-adic Weierstrass preparation theorem, we first show that the  $p$-adic Rankin-Selberg $L$-function $\mu_{p,f|\chi\times g|\bar{\chi}}(\cdot)$ is congruent to  the product of $p$-adic $L$-functions $\mu_{\tilde{f}_0|\xi_1}(\cdot) \mu_{\tilde{f}_0|\xi_2}(\cdot) $ modulo $\pi$ in Theorem~\ref{analytic final1}. Then using Lemma~\ref{congruence of measures},  we obtain the desired  congruence between  the $ p $-adic $ L $-functions.
					
					We first prove following lemma which compares the Gauss sums. 
					\begin{lemma}\label{comparision Gauss sum}
						Let $\varphi$ be a primitive Dirichlet character. Then we have $G(\varphi\chi)/G(\varphi) $ is a $p$-adic unit.
					\end{lemma}
					\begin{proof}
						Let $\varphi = \varphi_p \varphi'$, where $\varphi_p$ is a Dirichlet character of conductor of  $p$-power and  $\varphi'$ is a Dirichlet character of conductor prime to $p$. By \cite[Lemma 3.1.2]{Miyake}, we get $G(\varphi) = \varphi_p(\mathrm{cond}(\varphi')) \varphi'(\mathrm{cond}(\varphi_p)) G(\varphi_p) G(\varphi')$. Since $\chi$ has conductor co-prime to $p$, by a similar argument as above using \cite[Lemma 3.1.2]{Miyake} gives $G(\varphi\chi) = \varphi_p(\mathrm{cond}(\varphi'\chi)) \varphi'\chi(\mathrm{cond}(\varphi_p)) G(\varphi_p) G(\varphi'\chi)$. Combining the above two expressions, we get  $G(\varphi\chi)/G(\varphi)= \varphi_p(\mathrm{cond}(\varphi'\chi)) \bar{\varphi}_p(\mathrm{cond}(\varphi')) \chi(\mathrm{cond}(\varphi_p)) G(\varphi'\chi)/G(\varphi')$. As $\varphi'$ and $\varphi'\chi$ has conductor prime to $p$, it follows from \cite[Lemma 3.1.1]{Miyake} that $G(\varphi'\chi),G(\varphi')$ are $p$-adic units.	   		    
					\end{proof}
					As the root number of Eisenstein series involves various Gauss sums, the above lemma will help us to make comparison of $W(g^\rho|\bar{\phi})$ and $W(g^\rho|\chi\bar{\phi})$ for every finite order character $\phi$ of $\mathbb{Z}_p^\times$.

					\begin{theorem}\label{analytic final1}
						Let $ f \in S_{k}(\Gamma_{0}(N), \eta)$ be a $ p $-ordinary and newform with $ p \nmid N $. 
						Let  $ g = E_{l}(\xi_{2} \omega_p^{1-l}, \xi_{1}) $ and $t(f|\chi, g|\bar{\chi})$ be as in Lemma~\ref{twist and untwist Eis}. Set $ \tilde{f}_0 = f_0 \vert \iota_{m}$. 
						Suppose that $f$ satisfies the assumptions \ref{irr-f} and \ref{p-dist}. 
						Then for every finite order character $ \phi$ of $ \mathbb{Z}_{p}^{\times} $ and  $ 0 \leq j \leq k-l-1 $, we have 
						\begin{align*}
							\mu_{f|\chi \times g|\bar{\chi}}( x_p^{j} \phi)  
							\equiv  (\ast)  \mu_{\tilde{f}_0 }(  x_p^{j} \xi_1 \phi) \mu_{\tilde{f}_0}(x_p^{j}  \xi_2 \phi)  \mod \pi.
						\end{align*}
						Here $ (\ast) $ is the $p$-adic unit given by 
						$
						\frac{t(f|\chi, g|\bar{\chi}) u_{f} W(g^\rho|\chi\phi)}{W(g^\rho|\phi)\chi(p)^{\beta-1} \Gamma(l+j)\Gamma(j)} \frac{p^{(2-k)/2}c(f_0|\chi) \Omega_{\tilde{f}_{0}}^{+} \Omega_{\tilde{f}_{0}}^{-} }{2^{k-1}i^{k-1}\langle (f_0|\chi)^\rho|\tau_{N_{f|\chi}p}, f_0|\chi\rangle_{N_{f|\chi}p} }
						$ with $\beta = v_p(\mathrm{cond}(\phi))+ v_p(\mathrm{cond}(\xi_1 \omega_p^{1-l}\phi))$.
					\end{theorem}
					\begin{proof}
						By the $p$-adic Weierstrass preparation theorem, it is sufficient to show that the congruence holds for  all finite order characters of $\mathbb{Z}_p^\times$ with $\bar{\phi} \neq (\xi_{1}\omega_p^{1-l})_p$.  
						Note that $P_p(g^{\rho},p^j u_f^{-1}) =
						(1 - \bar{\xi}_{2}(p) p^j u_f^{-1}) (1-\bar{\xi}_{1}\bar{\omega}_p^{1-l}(p)p^{l+j-1}u_f^{-1})
						= e_{p}(u_{\tilde{f}_0}, x_p^j \xi_{2}) e_{p}(u_{\tilde{f}_0}, x_p^{l+j-1} \xi_{1}\omega_{p}^{1-l})$, where $ e_{p}(u_{\tilde{f}_0}, \cdot) $ is as defined in Theorem~\ref{p-adic l-function of modular form} and $ u_f = a(p,f_0) = a(p,\tilde{f}_0) $. Thus by  Lemma~\ref{twist and untwist Eis}(i),  we get
						\begin{small}
							\begin{align}\label{eq: rankin p-adic as product  trivial char}
								\begin{split}	
									\mu_{f|\chi \times g|\bar{\chi}} (x_p^{j}\iota_{p}) ~ & {=\joinrel=} ~ c(f_{0}|\chi)  t(f|\chi,g|\bar{\chi})  p^{s l/2} p^{sj} p^{(2-k)/2}  W(g^\rho|\chi) u_{f}^{1-s}  \chi(p)^{1-\beta}
									e_{p}(u_{\tilde{f}_0}, x_p^j \xi_{2}) e_{p}(u_{\tilde{f}_0}, x_p^{l+j-1} \xi_{1}\omega_{p}^{1-l}) \\ 
									& \qquad  \qquad \qquad   \times
									\frac{ L(j+1,\tilde{f}_0,\xi_{2}  \phi) L(l+j, \tilde{f}_0, \xi_{1}\omega_p^{1-l} \phi)} {(2i)^{k+l+2j} \pi^{l+2j+1} \langle (f_0|\chi)^{\rho}\vert\tau_{N_{f|\chi}p}, f_0|\chi \rangle_{N_{f|\chi}p}} \\
									&\stackrel{\eqref{root number of g definition}}{=\joinrel=} ~ \frac{t(f|\chi,g|\bar{\chi}) W(g^\rho|\chi)\xi_{2}(-1)}{\chi(p)^{\beta-1}W(g^\rho)M_0^{l/2+j} }  \times u_{\tilde{f}_0}^{1-s}  \frac{\mathrm{cond}(\xi_{1}\omega_{p}^{1-l})^{l+j} e_{p}(u_{\tilde{f}_0},x_p^{l+j-1} \xi_{1}\omega_{p}^{1-l})	L(l+j,\tilde{f}_0,\xi_{1}\omega_{p}^{1-l}\phi)}{(2 \pi i)^{l+j}G(\xi_{1}\omega_{p}^{1-l})} \\ 
									& \qquad   \times  \frac{p^{(2-k)/2} c(f_{0}|\chi) }{2^{k-1}i^{k-1} \langle (f_{0}|\chi)^\rho\vert\tau_{N_{f|\chi}p}, f_{0}|\chi \rangle_{N_{f|\chi}p}}\times  \frac{\mathrm{cond(\xi_2)}^{j+1} e_{p}(u_{\tilde{f}_0},x_p^j \xi_{2})	L(j+1,\tilde{f}_0,\xi_{2}\phi )}{(2 \pi i)^{j+1}G(\xi_{2})}.
								\end{split}	     
							\end{align}
						\end{small}%
						If $ \phi \neq  \iota_p$ and  $(\phi\xi_1\omega_p^{1-l})_0(p)=0$, then $  \xi_1 \omega_p^{1-l}\phi(p)= 0=  \phi\xi_{2}(p) $. Thus  $ e_{p}(u_{\tilde{f}_0},x_p^{l+j-1} \xi_{1}\omega_{p}^{1-l}\phi) =1 = e_{p}(u_{\tilde{f}_0},x_p^j \xi_{2}\phi) $. Let $ \beta =v_p( \mathrm{cond}(\xi_1\omega_p^{1-l}\phi))+ v_p(\mathrm{cond}(\xi_2\phi)) $. Then it follows from Lemma~\ref{twist and untwist Eis}(ii)   that 
						\begin{small}
							\begin{align}\label{eq: rankin p-adic as product non trivial char}
								\begin{split}	
									\mu_{f|\chi \times g|\bar{\chi}}(x_p^j \phi) & = c(f_{0}|\chi) t(f|\chi, g|\bar{\chi}) p^{\beta l/2} p^{\beta j}  p^{(2-k)/2} W(g^{\rho}|\chi\phi) u_{f}^{1-\beta}\chi(p)^{1-\beta} \frac{L(l+j, \tilde{f}_0, \xi_{1} \omega^{1-l} \phi) L(j+1, \tilde{f}_0, \xi_{2}\phi)}{(2i)^{k+l}\pi^{l+2j+1} \langle (f_{0}|\chi)^{\rho}\vert\tau_{N_{f|\chi}p}, f_0|\chi \rangle_{N_{f|\chi}p}} \\
									&=    \frac{t(f|\chi,g|\bar{\chi}) W(g^\rho|\chi\phi)\xi_{2}(-1)}{W(g^\rho|\chi)M_0^{l/2+j} \chi(p)^{1-\beta}}  \times u_{\tilde{f}_0}^{1-\beta} \frac{\mathrm{cond(\xi_{1}\omega_{p}^{1-l}\phi)}^{l+j} e_{p}(u_{\tilde{f}_0},x_p^{l+j-1} \xi_{1}\omega_{p}^{1-l}\phi)	L(l+j,\tilde{f}_0,\xi_{1}\omega_{p}^{1-l}\phi)}{(2 \pi i)^{l+j}G(\xi_{1}\omega_{p}^{1-l}\phi)} \\ 
									& \qquad   \times \frac{p^{(2-k)/2} c(f_{0}|\chi) }{2^{k-1}i^{k-1} \langle (f_{0}|\chi)^{\rho}\vert\tau_{N_{f|\chi}p}, f_{0}|\chi \rangle_{N_{f|\chi}p}} \times \frac{\mathrm{cond(\xi_2\phi)}^{j+1} e_{p}(u_{\tilde{f}_0},x_p^j \xi_{2}\phi)	L(j+1,\tilde{f}_0,\xi_{2}\phi )}{(2 \pi i)^{j+1}G(\xi_{2}\phi)} 
								\end{split}		
							\end{align} 
						\end{small}%
						Applying Theorem~\ref{period and integral measure} and choosing $ L = M_0 =$ the prime to $p$-part of $\mathrm{cond}(\xi_1\omega^{1-l})\mathrm{cond}(\xi_2)$, we obtain
						\begin{small}
							\begin{align}\label{eq: product of p-adic L functions of f}
								\begin{split}
									\mu_{\tilde{f}_0}(x_p^{l+j-1}\xi_{1}\omega_p^{1-l}\phi) 
									\mu_{\tilde{f}_0}(\tilde{f},x_p^{j} \xi_{2} \phi)
									&=  \frac{u_{f}^{-(\star)}}{\Omega_{\tilde{f}_0 }^{+} \Omega_{\tilde{f}_0}^{-}}  
									(l+j-1)!j!
									\mathrm{cond}(\xi_{2}  \phi)^{j+1}  e_{p}(u_{\tilde{f}_0},x_p^{j} \xi_{2}\phi) \frac{L(j+1, \tilde{f}_0, \xi_{2} \phi)}{G({\xi}_{2}{\phi})(-2 \pi i)^{j+1}} \\
									& 	 
									e_{p}(u_{\tilde{f}_0},x_p^{l+j-1} \xi_{1}\omega_{p}^{1-l}\phi) \mathrm{cond}(\xi_{1} \omega_p^{1-l} \phi)^{l+j}  \frac{L(l+j-1,\tilde{f}_0,\xi_{1} \omega_p^{1-l} \phi)} {G(\xi_{1}\omega_p^{1-l}\phi)(-2 \pi i)^{l+j}} ,
								\end{split}	
							\end{align}
						\end{small}%
						where $ (\star) $ is the power of $ p $ dividing and $ \mathrm{cond}(\xi_{1} \omega_p^{1-l} \phi)  \mathrm{cond}(\xi_{2} \phi) $. Note that $ (\star) =s $ if $ \phi = \iota_p $ and $ (\star) = \beta  $ if $ \phi \neq (\xi_{1} \omega^{1-l})_p^{-1} $.
						Substituting \eqref{eq: product of p-adic L functions of f} in \eqref{eq: rankin p-adic as product trivial char} and \eqref{eq: rankin p-adic as product non trivial char}, we get 
						\begin{small}
							\begin{align}\label{eq: last congruence}
								\mu_{f|\chi \times g|\bar{\chi}}(x_p^j \phi)  =
								\frac{t(f|\chi, g|\bar{\chi}) u_{f} W(g^\rho|\chi\phi)\chi(p)^{\beta-1}}{M_0^{l/2+j}W(g^\rho|\phi) \Gamma(l+j)\Gamma(j)}  \frac{p^{(2-k)/2} c(f_{0}|\chi) }{(2i)^{k-1} \langle (f_{0}|\chi)^\rho\vert\tau_{N_{f|\chi}p}, f_{0}|\chi \rangle_{N_{f|\chi}p}}   \mu_{\tilde{f}_0}( x_p^{l+j-1} \xi_{1}  \omega_p^{1-l} \phi) \mu_{\tilde{f}_0}( x_p^j \xi_{2} \phi).
							\end{align}
						\end{small}%
						Since $ x_p^{l-1} \omega_p^{1-l} \lvert_{\mathbb{Z}_p^{\times}} \equiv  1 \mod p $ and $ \mu_{\tilde{f}_0}(\cdot) $ is $ \mathcal{O} $-valued, we obtain $\mu_{\tilde{f}_0}(x_p^{l+j-1} \xi_{1}  \omega_p^{1-l} \phi) \equiv  \mu_{\tilde{f}_0}(x_p^{j} \xi_{1}  \phi) \mod \pi$. It remains to show 
						\begin{small}
							\begin{align*}
								\frac{t(f|\chi, g|\bar{\chi}) u_{f} W(g^\rho|\chi\phi)}{M_0^{l/2+j}W(g^\rho|\phi)\chi(p)^{1-\beta} \Gamma(l+j)\Gamma(j)}  \frac{p^{(2-k)/2} c(f_{0}|\chi) }{2^{k-1}i^{k-1} \langle (f_{0}|\chi)^\rho \vert\tau_{N_{f|\chi}p}, f_{0}|\chi \rangle_{N_{f|\chi}p}}
							\end{align*}
						\end{small}
						is a $p$-adic unit. The second term is a $p$-adic unit by Lemma~\ref{periods comparision Eis}. For the  first quantity, using the formula \eqref{root number of g definition}, we get 
						\begin{align*}
							\frac{W(g^\rho|\bar{\phi}\chi)}{W(g^\rho|\bar{\phi})} = \chi(-1) \frac{\mathrm{cond}(\xi_{1}\omega_p^{1-l}\bar{\chi} \phi)^{l/2}}{ \mathrm{cond}(\xi_{1}\omega_p^{1-l}\phi)^{-l/2}}   \frac{\mathrm{cond}(\xi_{2}\bar{\chi} \phi)^{-l/2}}{\mathrm{cond}(\xi_{2}\phi)^{l/2}}  \frac{G(\bar{\xi}_{1}\bar{\omega}_p^{1-l}\chi\bar{\phi})}{G(\bar{\xi}_{1}\bar{\omega}_p^{1-l}\bar{\phi})} \frac{G(\xi_{2}\phi\bar{\chi})}{ G(\xi_{2}\phi)}.
						\end{align*} 
						Since $\chi$ has conductor prime to $p$, we have $\mathrm{cond}(\xi_{1}\omega_p^{1-l}\bar{\chi}\phi)/\mathrm{cond}(\xi_{1}\omega_p^{1-l}\phi)$ and $\mathrm{cond}(\xi_{2}\bar{\chi}\phi)/\mathrm{cond}(\xi_{2}\phi)$ are co-prime to $p$ for all $\phi$. By Lemma~\ref{comparision Gauss sum}, it follows $G(\xi_{2}\bar{\chi}\phi)/G(\xi_{2}\phi)$ and $G(\bar{\xi}_{1}\bar{\omega}_p^{1-l}\chi\bar{\phi})/G(\bar{\xi}_{1}\bar{\omega}_p^{1-l}\bar{\phi})$ are $p$-adic units.  Observe that  $t(f|\chi, g|\bar{\chi})/\Gamma(l+j)\Gamma(j) = [N_{f|\chi},N_{g|\bar{\chi}}] N_{f|\chi}^{k/2} (N_{g|\bar{\chi}})_0^{(l+2j)/2} $ is a $p$-adic unit. Now the theorem follows.
					\end{proof} 
					
					\begin{remark}
						Note that for every Dirichlet character $ \varphi $, we have  $L(j,\tilde{f_{0}},\varphi) =  \left( 1 - \frac{\eta(p) \varphi(p) p^{k-1-j}}{u_f} \right) L(j,\tilde{f},\varphi)$. Taking $ \Omega_{\tilde{f}}^{\pm} :=  \Omega_{\tilde{f}_0}^{\pm}$ in Theorem~\ref{p-adic l-function of modular form}, we get a bounded measure $ \mu_{\tilde{f}} $. It follows that  $ \mu_{\tilde{f}}  = \mu_{\tilde{f}_0} $, with $ \mu_{\tilde{f}_0}$ as defined in Theorem~\ref{period and integral measure}. So we may replace the $\mathcal{O}$-valued measure $ \mu_{\tilde{f}_0} $ in Theorem~\ref{analytic final1} by  $\mathcal{O}$-valued measure $\mu_{\tilde{f}} $.  
					\end{remark}
					
					We next show that the root number of a primitive $p$-ordinary form $h$ twisted by $\chi$ is equal to the  root number of $h$ up to a $p$-adic unit.  
					\begin{lemma}\label{comparision root number cusp}
						Let $h \in S_{l}(I, \psi)$ be  a $p$-ordinary newform. Then  for every finite order character $\phi$ of $\mathbb{Z}_p^\times$ we have, $W(h^\rho|\phi\chi)/W(h^\rho|\phi)$ is a $p$-adic unit. 	   		 
					\end{lemma}
					\begin{proof}
						As in \cite[\S 5]{Hidarankin2} we write $W(h^\rho|\phi) = W_p(h^\rho|\phi) W'(h^\rho|\phi)$, where $W_p(h^\rho|\phi)$ is the local $\varepsilon$-factor of $h^\rho|\phi$ at $p$. Similarly, let $W(h^\rho|\chi\phi) = W_p(h^\rho|\chi\phi) W'(h^\rho|\chi\phi)$. From \cite[(5.4 a)-(5.4.c)]{Hidarankin2}, it follows that $W'(h^\rho|\phi),W'(h^\rho|\chi\phi)$ are $p$-adic units. Hence it is enough to show that $W_p(h^\rho|\chi\phi)/W_p(h^\rho|\phi)$ is a $p$-adic unit.
						
						Let $I = I_0 p^\alpha$ with $p\nmid I_0$. Note that $h$ is $p$-minimal and  the local automorphic representation of $h$ at $p$, say $\pi_p(h)$, is not supercuspidal. Also we have $\pi_p(h)$ 
						is principal series (respectively special) implies that  $\pi_p(h|\chi)$ is principal series (respectively special).  Thus by \cite[(5.5 b)]{Hidarankin2}, we have 
						\begin{align*}
							W_{p}(h|\chi) = \begin{cases}
								\chi(p)^\alpha W_p(h) & \mathrm{if} ~ \pi_p(h) ~\mathrm{is ~ principal}, \\
								\chi(p)  W_p(h) & \mathrm{if} ~ \pi_p(h) ~\mathrm{is ~ special}.	   		     	     	   
							\end{cases}
						\end{align*} 
						If $\phi$ is non-trivial character of $\mathbb{Z}_p^\times$, then it follows from by \cite[(5.5 c)]{Hidarankin2} that
						\begin{align*}
							W_{p}(h^\rho|\chi\phi) = \begin{cases}
								\chi(p)^{v_p(\mathrm{cond}(\phi))+v_p(\mathrm{cond}(\phi\psi))} W_p(h|\phi) & \mathrm{if} ~ \pi_p(h) ~\mathrm{is ~ principal}, \\
								\chi(p)^{v_p(\mathrm{cond}(\phi))} W_p(h|\phi) & \mathrm{if} ~ \pi_p(h) ~\mathrm{is ~ special}.
							\end{cases}
						\end{align*}
						Thus from above we have $W_{p}(h^\rho|\chi\phi) $ differs from $W_{p}(h^\rho|\phi) $ by a $p$-adic unit. 
					\end{proof}
					Let $\Sigma_{0} $ be the set of all rational primes dividing $ m $ and  $  \mu_{f \times h}$ be the $\mathcal{O}$-valued measure as in Theorem~\ref{padic rankin}. For every finite order character $\phi$ of $\mathbb{Z}_p^\times$ and $0 \leq j \leq k-l-1$, set
					\begin{align}\label{p-adic Rankin Sigma and p-adic Rankin}
						\mu^{\Sigma_0}_{f \times h} (x_p^j\phi) := \mu_{f\times h} (x_p^j\phi) \times \prod_{r \in \Sigma_0}L_{r}(l+j, f_0, g| \phi).
					\end{align}
					We next show that the $p$-adic $L$-function $\mu_{f|\chi \times h|\bar{\chi}}$ generates the same ideal as the imprimitive $p$-adic $L$-function $\mu^{\Sigma_0}_{f \times h}$.
					\begin{lemma}\label{analytic final cuspform}
						Let the notation be as in Lemma~\ref{twist and untwist cusp}.  Then for every finite order character $ \phi$ of of $ \mathbb{Z}_{p}^{\times} $ and  $ 0 \leq j \leq k-l-1 $, the values $\mu_{f|\chi \times h|\bar{\chi}}(x_p^j\phi)$ and $  \mu^{\Sigma_0}_{f \times h} (x_p^j\phi)$ differ by a $p$-adic unit.	     
					\end{lemma}
					\begin{proof}
						By Lemma~\ref{c(f) and petterson} and Lemma~\ref{period comparision cusp}, we have \begin{align*}
							\frac{c(f_0|\chi)}{c(f_0)} \frac{\langle f_0^\rho | \tau_{Np}, f_0\rangle_{Np}}{\langle f_0^\rho|\chi | \tau_{N_{f|\chi}p}, f_0|\chi \rangle_{N_{f|\chi}p}} 
						\end{align*}
						is a $p$-adic unit. By Lemma~\ref{comparision root number cusp}, we have  $W(h^\rho|\phi\chi)/W(h^\rho|\phi)$ is $p$-adic unit. Now the lemma follows from Lemma~\ref{twist and untwist cusp}.
					\end{proof}
					Let $\mu_{\tilde{f}_0}(\cdot)$ be the $p$-adic $L$-function as defined in Theorem~\ref{period and integral measure}. For $ 0 \leq j \leq k-2 $ and Dirichlet character $ \varphi $,  consider  the measure  $\mu^{\Sigma_{0}}_{p,f,\varphi, j}$ on $ (1+p\mathbb{Z}_{p}) $  satisfying the interpolation property
					\begin{equation}\label{def: p-adic L-function f}
						\mu^{\Sigma_{0}}_{p,f,\varphi, j}(\phi ) = \mu_{\tilde{f}_0}(x_p^j \varphi \phi) =  \frac{e_{p}(u_{f},\varphi \phi x_p^j)}{u_f^{\mathrm{cond}(\phi)}} \frac{\mathrm{cond}(\varphi \phi)^{j+1} j!}{G(\chi \phi)} \frac{L(j+1, \tilde{f},\varphi \phi)}{(2 \pi i)^{j+1}\Omega_{\tilde{f}_0}^{\mathrm{sgn}((-1)^j\varphi\phi(-1))}} 
					\end{equation}
					for every finite order character  $ \phi: \mathbb{Z}_p^{\times} \rightarrow \mathbb{C}_p^{\times} $ of $ p $-power order. 

					For $ l-1 \leq j \leq k-2 $, consider  the measure  $\mu^{\Sigma_0}_{p,f h, j}$ on $(1+p\mathbb{Z}_{p})$ satisfying the interpolation property:
					\begin{align}\label{def: p-adic L-function f x h}
						\mu^{\Sigma_0}_{p,f \times  h,j}(\phi) = \mu^{\Sigma_0}_{f \times  h}( x_p^{j-l+1} \phi ) =  \mu_{f\times h} (x_p^{j-l+1} \phi) \times \prod_{r \in \Sigma_0}L_{r}(j+1, f_0, g| \phi)
					\end{align}
					for every finite order character  $ \phi: \mathbb{Z}_p^{\times} \rightarrow \mathbb{C}_p^{\times} $ of $ p $-power order. Recall that $ \Q_{\cyc} $ is the cyclotomic $\Z_p $-extension of $ \Q $ and $ \Gamma = \mathrm{Gal}(\Q_{\cyc}/\Q) \cong \Z_p$. Then it is known that $\mu^{\Sigma_{0}}_{p,f,\varphi, j}$ (resp. $\mu^{\Sigma_0}_{p,f \times h, j}$) corresponds to a power series in the Iwasawa algebra $\mathcal{O}[[\Gamma]]$ which we continue to denote by  $\mu^{\Sigma_{0}}_{p,f,\varphi, j}$ (resp. $\mu^{\Sigma_0}_{p,f \times h, j}$).
					\begin{theorem}\label{analytic final}
						Let $ f \in S_{k}(\Gamma_{0}(N), \eta)$ be a $ p $-ordinary newform with $ p \nmid N $.  Let $ h \in S_{l}(\Gamma_{0}(I),\psi) $ be a $ p $-ordinary eigenform  such that $ 2 \leq l < k $ and $ (T_{h}/\pi)^{ss} \cong \bar{\xi}_{1} \oplus \bar{\xi}_{2} $. Suppose that $f$ satisfies the assumptions \ref{irr-f} and  \ref{p-dist}. Then for $l-1 \leq j \leq k-1$, we have the following congruence  of ideals in the Iwasawa algebra $ \mathcal{O}[[\Gamma]] $
						\begin{align*}
							(\mu^{\Sigma_0}_{p,f\times h, j}) \equiv (\mu^{\Sigma_{0}}_{p,f,\xi_1, j}) (\mu^{\Sigma_{0}}_{p,f,\xi_2, j}) \mod \pi.
						\end{align*}
					\end{theorem}
					\begin{proof}
						This follows from Lemma~\ref{congruence of measures}, Theorem~\ref{analytic final1} and Lemma~\ref{analytic final cuspform}. 
					\end{proof}
					\section{Selmer group of Modular form and Rankin-Selberg product}\label{sec: Selmer groups}
					In this section, we discuss the $ p^\infty $-Greenberg Selmer group attached to {a} modular form and {the} Rankin-Selberg product. 
					We prove a control theorem for the $ p^\infty $-Selmer group attached to the Rankin-Selberg product and also discuss  explicit description of 
					various $ p^\infty $-Selmer groups that appear.
					\subsection{Background on   Selmer groups}\label{sec2}
					We recall the following notation from  \S\ref{section: simplyfying rankin}:  For a normalised eigenform  $ F $ with  nebentypus $\varphi$,    
					$K_F$ denotes the corresponding number field. 
					Let $ L $ be a number field containing $ K_F $, $\mathfrak{p}$ is a prime ideal in $L$ dividing $p$ (compatible with $\iota_p$), $L_{\mathfrak{p}}$ be the completion of $L$ at $\mathfrak{p}$ and  $\pi_L$ denote a uniformizer of the ring of integers  $ \mathcal{O}_{L_\mathfrak{p}}  $ of   $L_\mathfrak{p}$. Also, recall $\chi_{\cyc}: G_\Q \lra \Z_p^\times$ {is} the $p$-adic cyclotomic character. 
					For a discrete  $ \mathbb{Z}_{p}[[\Gamma]] $-module $\mathcal{M}$, let  $M^\vee:=\mathrm{Hom}_{\text{cont}}(\mathcal{M},\frac{\Q_p}{\Z_p})$ be the Pontryagin dual of $ \mathcal{M} $.
					\begin{theorem}\label{rhof}$($Eichler,  Shimura, Deligne, Mazur-Wiles, Wiles$)$
						Let $F(z) = \sum_{n \geq 1} a(n,F) q^n\in S_k(\Gamma_0(A),\varphi)$  be a $p$-ordinary newform with $k\geq 2$.  Then there exists a  Galois representation,
							$\rho_{F}: G_\Q \lra \mathrm{GL}_2(L_{\mathfrak p})$  such that
							
							\begin{enumerate}[label=$\mathrm{(\roman*)}$]
								\item \label{unr}
								{For} all primes $r \nmid Ap$, $\rho_F$ is unramified  and for the $($arithmetic$)$ Frobenius $\mathrm{Frob}_r$ at $r$, we have 
								$$\mathrm{trace}(\rho_{F}(\mathrm{Frob}_r))=a(r, F),\quad \det(\rho_{F}(\mathrm{Frob}_r)) = \varphi(r) \chi_{\cyc}(\mathrm{Frob}_r)^{k-1} 
								= \varphi(r) r^{k-1}.$$
								It follows (by the Chebotarev density theorem) that $\det(\rho_{F})=\varphi \chi_{\cyc}^{k-1}$.
								\item 
								As $ f $ is $p$-ordinary, let  $u_F$  be the  unique $p$-adic unit  root of  $X^2 -a(p,F)X +\varphi(p) p^{k-1}.$  Let   $\lambda_{F}$ be the unramified character with $\lambda_{F}(\mathrm{Frob}_{p})=u_F$. Then \begin{small}
									$$ \rho_{F}|_{G_{p}} \sim \begin{pmatrix} \lambda_{F}^{-1} \varphi \chi_{\cyc}^{k-1}  & * \\   0 & \lambda_{F} \end{pmatrix}.$$
								\end{small}%
							\end{enumerate} 
						\end{theorem}
						Let $V_{F} \cong L_{\mathfrak p}^{\oplus 2}$ denote the representation space of $\rho_{F}$. By compactness of $G_\Q$, there exists a $ G_{\Q} $ invariant   $\mathcal{O}_{L_\mathfrak{p}}$-lattice $T_{F}$ of  $V_{F}$. 
						Let 
							$	\bar{\rho}_{F}: G_{\Q} \lra \mathrm{GL_2}(\mathcal{O}_{L_\mathfrak {p}}/{\pi_L})$
						be the residual representation of $\rho_{F}$.

						Recall $f \in S_k(\Gamma_0(N), \eta)$  and $ h\in S_l(\Gamma_0(I_0p^\alpha), \psi) $ are $ p $-ordinary newforms with $ p \nmid N I_0 $. Let $K$ be a number field containing $K_f$ and $ K_h$. To ease the notation, we denote  $\mathcal{O}_{K_\mathfrak{p}}$  by $\mathcal{O}$, $\pi_{K}$ by $ \pi $ and $\mathcal{O}/\pi$ by $\mathbb{F}$. Let $\Sigma$ be a finite set of primes of $\Q$  such that $\Sigma \supset \{ \ell : \ell \mid pNI \infty \}$ and  $\Q_\Sigma$ be the  {maximal} algebraic extension of $\Q$ unramified outside $\Sigma$. Set $G_\Sigma(\Q) := \text{Gal}(\Q_{\Sigma}/\Q)$.     Let us take $\mathfrak{g} \in \{f,h\}$ 
						and set $A_{\mathfrak{g}}:= V_{\mathfrak{g}}/T_{\mathfrak{g}}$. For a character $ \varphi $, let $ V_{\mathfrak{g}} (\varphi) := V_{\mathfrak{g}} \otimes_{\mathbb{Q}_p} \varphi $ and $ T_{\mathfrak{g}} (\varphi) := T_{\mathfrak{g}} \otimes_{\mathbb{Z}_p} \varphi $. Further, put  $ V_{\mathfrak{g}}(j) := V_{\mathfrak{g}} (\chi_{\cyc}^{-j})  $, $ T_{\mathfrak{g}}(j) := T_{\mathfrak{g}} (\chi_{\cyc}^{-j})$ and  $ A_{\mathfrak{g}}(j) = V_{\mathfrak{g}} (j)/ T_{\mathfrak{g}}(j) \cong A_{\mathfrak{g}} \otimes \omega_p^{-j}$. We have the canonical maps
							$0 \lra T_{\mathfrak{g}} \lra V_{\mathfrak{g}} \lra A_{\mathfrak{g}} \lra 0$.
						As  $\mathfrak{g}$ is $p$-ordinary, 
						we have  the following  filtrations as a $G_{p}$-module
						\begin{small}
							\begin{equation}
								0 \lra V_{\mathfrak{g}}^+ \lra V_{\mathfrak{g}} \lra V_{\mathfrak{g}}^{-} \lra 0 \quad \text{and} \quad 0 \lra T_{\mathfrak{g}}^+ \lra T_{\mathfrak{g}} \lra T_{\mathfrak{g}}^{-} \lra 0,
							\end{equation}
						\end{small}%
						where  both $T_{\mathfrak{g}}^+$ (resp. $V_{\mathfrak{g}}^+$) and  $T_{\mathfrak{g}}^{-}$ (resp. $V_{\mathfrak{g}}^-$) are free  $\mathcal{O}$-modules (resp. $K_{\mathfrak{p}}$-vector spaces) of rank $1$ and the action of $G_\Q$ on $T_{\mathfrak{g}}^{-}$ and $V_{\mathfrak{g}}^-$ is unramified at $p$.  Following  \cite{gr1},  for a modular form $\mathfrak{g}$ of weight $k \geq 2$, define a filtration on $ V_{\mathfrak{g}}(j)$ with $ 0  \leq j \leq k-2$, by
						\begin{small}
							\begin{align*}
								F^s(V_{\mathfrak{g}}(j)) = 
								\begin{cases}
									V_{\mathfrak{g}}(j)  & \text{ if } s \leq -j, \\
									V_{\mathfrak{g}}^{+}(j) & \text{ if } -j+1 \leq s \leq k-1-j,\\
									0 & \text{ if } s \geq k-j. 
								\end{cases}
							\end{align*}
						\end{small}%
						We have a corresponding  filtration on $ T_{\mathfrak{g}}(j) $. Also, define $A^+_{\mathfrak{g}}(j) = (V_{\mathfrak{g}}^+(j)/T_{\mathfrak{g}}^+(j))$ and 
						$A_{\mathfrak{g}}^{-}(j) = A_{\mathfrak{g}}(j)/A^+_{\mathfrak{g}}(j)$.
						
						Set $V := V_f \otimes_{K_{\mathfrak{p}}} V_h$ and $T:=T_f\otimes_{\mathcal{O}}T_h$. Then we have an induced filtration on $V$ and $T$ respectively:		
						\begin{equation}\label{Filtration on T}
							\begin{aligned} 
								&0\subset V_f^+\otimes_{K_{\mathfrak{p}}}V_h^+ \subset V_f^+\otimes_{K_{\mathfrak{p}}} V_h \subset  V_f^+\otimes_{K_{\mathfrak{p}}} V_h + V_f\otimes_{K_{\mathfrak{p}}} V_h^+ \subset V_f\otimes_{K_{\mathfrak{p}}} V_h,	\\	
								&0\subset T_f^+\otimes_{\mathcal{O}}T_h^+ \subset T_f^+\otimes_{\mathcal{O}}T_h \subset  T_f^+\otimes_{\mathcal{O}}T_h + T_f\otimes_{\mathcal{O}} T_h^+ \subset T_f\otimes_{\mathcal{O}} T_h.
							\end{aligned}
						\end{equation}
						For every  $l-1 \leq j \leq k-2$, we define $V_j := V \otimes_{\mathbb{Q}_p} \chi_{\cyc}^{-j}  $, $T_j := T \otimes_{\mathbb{Z}_p}  \chi_{\cyc}^{-j}$ and  $A_j := \frac{V_j}{T_j} $. Note that $A_j \cong T_f(\chi_{\cyc}^{-j}) \otimes_{\mathcal{O}} A_h \cong T_h(\chi_{\cyc}^{-j}) \otimes_{\mathcal{O}} A_f.$
						The action of $I_p$  on the successive quotients in our filtration \eqref{Filtration on T} is given by $\chi_{\cyc}^{k+l-2}, \chi_{\cyc}^{k-1}, \chi_{\cyc}^{l-1}, \chi_{\cyc}^{0}$. In particular, for $l-1 \leq j \leq k-2$, we have the following filtration of $ V_j = V_f\otimes V_h(j)$
						\begin{small}
							\begin{align*}
								F^s(V_f\otimes V_h(j)) =
								\begin{cases}
									V_f\otimes V_h(j) &\text{ if } s \leq -j, \\
									V_f^{+} \otimes V_h(j) + V_f \otimes V_h^+(j) & \text{ if } -j+1 \leq s \leq l-1-j, \\
									V_f^{+} \otimes V_h(j) & \text{ if } l-j \leq s \leq k-1-j,\\
									V_f^{+}\otimes V_h^{+}(j)  & \text{ if } k-j \leq s \leq k+l-2-j, \\
									0 &  \text{ if } s \geq k+l-1-j.
								\end{cases}
							\end{align*}
						\end{small}%
						Similarly, define a filtration on $  T_f \otimes T_h(j)$.  Following \cite{gr1}, we have
						$F^+(V_j) = V^+_f\otimes V_h(j)$
						and  $F^+(T_j) = T^+_f\otimes T_h(j)$. Finally, we put 
						\begin{align}\label{def A minus j}
							F^+(A_j) :=F^+(V_j)/ {F^+(T_j)} \quad \text{and } \quad  A^{-}_j = A_j/{F^+(A_j)}.
						\end{align}
						Note that $A_j^-=T_h \otimes A_f^-(j)$, where $ A_f^-(j) =  A_f(j)/ A_f^{+}(j)$.

						\par  Let $ L $ be a number field and $ \Sigma_L $ be the set of all primes in $ L $ lying above $ \Sigma $.  Put  $ G_{\Sigma}(L) = \mathrm{Gal}(L_{\Sigma_{L}}/L) $. 
						For every prime $ v \in \Sigma_L$ and $ B_j \in \{  A_f(\xi_{1}\omega_p^{-j}), A_j, A_f(\xi_{2}\omega_p^{-j})\} $, let choose a submodule $H^1_{\dagger} (L_v,B_j) \subset H^1 (L_v,B_j)  := H^1 (G_{L_v},B_j)$. 
						For this choice,  we  define $\dagger$-Selmer group $S_\dagger(B_j/L)$  as 		
							\begin{equation*}
								S_{\mathrm{\dagger}}(B_j/L): =\mathrm{Ker}\bigg(H^1(G_\Sigma(L), B_j) \lra \underset{v \in \Sigma_L}\prod \frac{H^1(L_{v}, B_j)}{H_{\dagger}^1(L_{v}, B_j)}\bigg).
							\end{equation*}
						For any $r \in \N$, there is a natural inclusion  
						\begin{equation}\label{ifr}
							0 \lra B_{j}[\pi^r] \stackrel{i_{r}}{\lra} B_{j}.
						\end{equation}		
						Next we define $\pi^r$     $\dagger$-Selmer group $S_{\dagger}(B_{j}[\pi^r]/L)$  as 
						\begin{equation*}
							S_{\dagger}(B_{j}[\pi^r]/L): =\mathrm{Ker}\Big(H^1(G_\Sigma(L), B_{j}[\pi^r]) \lra \underset{v \in \Sigma_L}\prod \frac{H^1(L_v, B_{j}[\pi^r])} {H_{\dagger}^1(L_v, B_{j}[\pi^r])}\Big),
						\end{equation*}
						where $H_\dagger^1(L_v, B_{j}[\pi^r]):=  { i_{r}^{*^{-1}}} (H^1_\dagger (L_v,B_{j}))$ for every $v \in \Sigma_L$. Here $i_{r}^{*}: H^1(L_v,  B_{j}[\pi^r]) \lra H^1(L_v, B_{j}) $ is  induced from $i_{r}$ in \eqref{ifr}. For a prime $v$ in $L$, let $G_{L_{v}}$ and $I_{v}$ denote the decomposition group and  inertia group at $v$ respectively. For a $G_{L}$-module $\mathcal{M}$ we sometimes denote  the $i^{\mathrm{th}}$-cohomology group $H^i(G_{L},\mathcal{M})$ by $H^i(L,\mathcal{M})$. Similarly, for fields $ L' \subset L \subset \bar{\Q}$, let $H^{i}(L/L',\mathcal{M})  := H^i(\mathrm{Gal}(L/L'),\mathcal{M})$.
						
						\begin{definition}\label{def: dagger cohomology}
							Let $ B_j \in \{  A_f(\xi_{1}\omega_p^{-j}), A_j, A_f(\xi_{2}\omega_p^{-j})\} $, $L$ be a number field and $v \in \Sigma_L$. 
							For $\dagger \in \{\mathrm{Gr}, \mathrm{str} \}$, define 
							\begin{align*}
								H^1_{\dagger}(L_v, B_{j}) : =
								\begin{cases} 
									\text{Ker}\big(H^1(G_{L_v}, B_{j}) \lra H^1(I_v, B_{j})\big) & \text{if } v \nmid p \text{ and } \dagger \in \{\mathrm{Gr}, \mathrm{str} \}, \\
									\text{Ker}\big(H^1(G_{L_v}, B_{j}) \lra H^1(I_v, B^-_{j})\big) & \text{if } v \mid p \text{ and } \dagger = \mathrm{Gr}, \\
									\text{Ker}\big(H^1(G_{L_v}, B_{j}) \lra H^1(G_v, B^-_{j})\big) & \text{if } v \mid p \text{ and } \dagger = \mathrm{str}.
								\end{cases}
							\end{align*}
						\end{definition}

						Recall that	$ \Q_{\cyc} $ is the cyclotomic $ \Z_p $-extension of $ \Q $ and $\Gamma := \mathrm{Gal}(\Q_{\cyc}/\Q) \cong \Z_p$. Set $ \Q_n := \Q_\cyc^{\Gamma^{p^n}} $ for every $ n \geq 1 $. Define 
						\begin{small}
							$$S_\dagger (B_j/\Q_\cyc): = \underset{n}{\varinjlim} ~S_\dagger (B_j/\Q_n) \quad \text{and} \quad  S_\dagger(B_j[\pi^r]/\Q_\cyc): = \underset{n}{\varinjlim} ~S_\dagger(B_j[\pi^r]/\Q_n).
							$$
						\end{small}%
						Let $\Q_{\cyc,w}$ denote the completion at $w$ for a prime $w$ in $\Q_{\cyc}$ and set $\Sigma^{\infty}$ to be the primes in $\Q_{\cyc}$ lying above $\Sigma$.  It can be checked that  for $r \geq 0$, we have
						\begin{small}
							\begin{align*}
								S_\dagger(B_j[\pi^r]/\Q_\cyc) =  \mathrm{Ker}\Big(H^1(G_\Sigma(\Q_\cyc), B_{j}[\pi^r]) \lra \underset{w \in \Sigma^\infty}\prod \frac{H^1(\Q_{cyc,w}, B_{j}[\pi^r])} {H_{\dagger}^1(\Q_{cyc,w}, B_{j}[\pi^r])}\Big),
							\end{align*}
						\end{small}%
						where $H_{\dagger}^1(\Q_{cyc,w}, B_{j}[\pi^r])$ is as defined in Definition~\ref{def: dagger cohomology}.
						\begin{theorem}
							The kernel and the cokernel of the map $ S_\dagger(A_j/\Q_n) \lra S_\dagger(A_j/\Q_\cyc)^{\Gamma^{p^n}} $ are finite and uniformly bounded  independent of $ n $, for $\dagger \in \{ \mathrm{str}, \mathrm{Gr}\}$. 
						\end{theorem}
						
						\begin{proof}
							Consider the commutative diagram
							\begin{tiny}		
								\[
								\begin{tikzcd}
									0 \arrow[r] 
									&  S_{\dagger}(A_j/\Q_n) \arrow[d, "r_{\dagger,n}"] \arrow[r] &  H^1(G_\Sigma(\Q_n), A_j)  \arrow[d] \arrow[r] & \underset{v \in \Sigma_{\Q_n}}\prod  \frac{H^1(\Q_{n,v}, A_j)}{H_{\dagger}^1(\Q_{n,v}, A_j)}   \arrow[d]\\
									0 \arrow[r]
									& S_{\dagger}(A_j/\Q_{\cyc})^{\Gamma^{p^n}} \arrow[r] &   H^1(G_\Sigma(\Q_{\cyc}), A_j)^{\Gamma^{p^n}} \arrow[r] &  \underset{v \in \Sigma_{\Q_n}}\prod \Big(\underset{{w|v}} \prod  \frac{H^1(\Q_{\cyc,w}, A_j)}{H_{\dagger}^1(\Q_{\cyc,w}, A_j)}\Big)^{\Gamma^{p^n}}.   
								\end{tikzcd}
								\]
							\end{tiny}%
							To show $\mathrm{ker}(r_{\dagger,n})$ is finite and uniformly bounded for every  $n$, by  \cite[Theorem 3.5(1)]{o1},  it is enough to show that $H^0(\Q_n, V_j)=0$, for all $n$. In fact, it suffices to show that $H^0(\Q_{n,v}, V_j)=0$, for  $ v \mid p$. This will follow if we can establish for  $ v \mid p $ 
							\begin{equation}\label{hkjadlk:DK?ALF}
								H^0(\Q_{n,v}, V_f^+\otimes V_h(-j))=0 = H^0(\Q_{n,v}, V_f^-\otimes V_h(-j)).
							\end{equation}
							
							Further, to establish the second equality in \eqref{hkjadlk:DK?ALF}, it suffices to show, for $ v \mid p $			
							\begin{equation}\label{hkjadlk:DK?ALF2}
								H^0(\Q_{n,v}, V_f^-\otimes V^{+}_h(-j))=0 = H^0(\Q_{n,v}, V_f^-\otimes V^{-}_h(-j)).
							\end{equation}
							Let $K' = \Q_{n,v}(\mu_{p^\infty}) = \Q_v(\mu_{p^\infty})$. Then $G_{K'}$ acts on $V_f^-\otimes V^+_h(-j)$  and  $V_f^-\otimes V^-_h(-j)$) by  $\lambda_f\lambda_h^{-1}\theta$ and  $\lambda_f\lambda_h$ respectively where $\theta$ is a finite order character. Note that as $f,h$ are ordinary at $p$, by Ramanujan–Petersson conjecture,  $|\lambda_f(\mathrm{Frob}_{v})|_\mathbb{C} = p_{v}^{\frac{k-1}{2}}$ and $|\lambda_h(\mathrm{Frob}_{v})|_\mathbb{C} =p_{v}^{\frac{l-1}{2}}$, where $p_v$ is the cardinality of the residue field at $v$. Thus looking at the action of Frobenius at the prime $v$ dividing $p$, we deduce 
							$H^0(G_{K'}, V_f^-\otimes V^+_h)=H^0(G_{K'}, V_f^-\otimes V^-_h)=0$ and  \eqref{hkjadlk:DK?ALF2} follows. We can deduce the first vanishing result in \eqref{hkjadlk:DK?ALF} similarly. Putting these together, we get $\text{ker}(r_{\dagger,n})$ is finite and uniformly bounded.
							
							Next we show that $\mathrm{coker}(r_{\mathrm{str},n})$ is finite and uniformly bounded independent of $n$. Let $\tilde{Fr}_{v,n}$ be a fixed lift of $\mathrm{Frob}_v$ in $G_{K_{n,v}}$. Again by  \cite[Theorem 3.5(2)]{o1}, it is enough to show for all places  $v \mid p$ in $K$, the action of $\tilde{Fr}_{v,n}$ is non-trivial on the successive quotients of the filtration $\frac{F^iV_j}{F^{i+1}V_j}\otimes \chi_{\cyc}^{-i}$. In our case,  the action of $\tilde{Fr}_{v,n}$ on $\frac{F^iV_j}{F^{i+1}V_j}\otimes \chi_{\cyc}^{-i}$ is given by $\lambda_f^{\pm 1} \lambda_h^{\pm 1} \theta$ where $\theta$ is a finite order   character. As before, 
							we can show that the action of $\tilde{Fr}_{v,n}$ is non-trivial. Hence we deduce $\mathrm{coker}(r_{\mathrm{str},n})$ is finite and uniformly bounded independent of $n$.
							
							Next we will show $\mathrm{coker}(r_{\mathrm{Gr},n})$ is finite and uniformly bounded independent of $n$. 
							Note that the local condition defining the  strict Selmer differs with the Greenberg Selmer only at primes dividing $p$. Let $\phi_{\mathrm{Gr,str}}$ be the natural map  $S_{\mathrm{str}}(A_j/\Q_\cyc) \stackrel{\phi_{\mathrm{Gr,str}}}{\lra} S_{\mathrm{Gr}}(A_j/\Q_\cyc)$. From the definitions of Greenberg and strict Selmer group, for every $n$, there is a natural map from $\mathrm{coker}(r_{\mathrm{str},n}) \lra \mathrm{coker}(r_{\mathrm{Gr},n})$ such that the order of the  cokernel is bounded by  the order of  $\mathrm{coker}(\phi_{\mathrm{Gr,str}})$.  Thus it suffices to show  $\mathrm{coker}(\phi_{\mathrm{Gr,str}})$ is finite. Now the order of the $\mathrm{coker}(\phi_{\mathrm{Gr,str}})$ is bounded by the order of $\underset{w\mid p}{\oplus}H^1(\Q_{\cyc,w}^{\mathrm{unr}}/{\Q_{\cyc,w}}, A_j^{I_w})$ where $w$ is a prime in $\Q_\cyc$ dividing $p$ and $I_w$ is the inertia subgroup of $\bar{\mathbb{Q}}_p/\Q_{\cyc,w}$ at $w$. As $\text{Gal}(\Q_{\cyc,w}^{\mathrm{unr}}/{\Q_{\cyc,w}})$ is topologically cyclic, it suffices to show  $H^0(\Q_{\cyc,w}^{\mathrm{unr}}/{\Q_{\cyc,w}}, A_j^{I_w})=A_j^{G_{\Q_{\cyc, w}}}$ is finite for each $w \mid p$.  Thus it further reduces  to show $H^0(\Q_{\cyc, w}, V_j) =0$. This follows from the proofs of \eqref{hkjadlk:DK?ALF}, \eqref{hkjadlk:DK?ALF2} written above. 
						\end{proof}		
						For any subset $\Sigma' \subset \Sigma$, such that $p \not \in \Sigma'$ and $ B_j \in \{  A_f(\xi_{1}\omega_{p}^{-j}), A_j, A_f(\xi_{2}\omega_{p}^{-j})\} $ we define
						\begin{small}%
							\begin{align*}
								S^{\Sigma'}_{\mathrm{Gr}}(B_j[\pi]/L):=\mathrm{Ker}\bigg(H^1(\Q_\Sigma/L, B_j[\pi]) \lra \underset{v \in \Sigma_L \setminus \Sigma_{L}'}\prod \frac{H^1(L_v, B_j[\pi])}{H_\dagger^1(L_v, B_j[\pi])}\bigg),
							\end{align*}
						\end{small}%
						where $\Sigma_{L}'$ are the set of all primes in $L$ lying above $\Sigma'$.

						Let $B_j \in \{  A_f(\xi_{1}\omega_{p}^{-j}), A_j, A_f(\xi_{2}\omega_{p}^{-j})\}$. Note that if $H^0(G_L,B_j[\pi])=0$, then the natural map $H^1(\Q_\Sigma/L, B_j[\pi])  \lra H^1(\Q_\Sigma/L, B_j)[\pi]$ is an isomorphism. The following lemma is immediate from the definition of $H_{\mathrm{Gr}}^1(L_v, B_j[\pi])$. 

						\begin{lemma}\label{inside-out-lem}
							Let $B_j \in \{  A_f(\xi_{1}\omega_{p}^{-j}), A_j, A_f(\xi_{2}\omega_{p}^{-j})\}$. Assume $H^0(G_L, B_j)=0$. 
							Then the natural map $S^{\Sigma'}_\mathrm{Gr}(B_j[\pi]/L) \lra S^{\Sigma'}_\mathrm{Gr}(B_j/L)[\pi]$ is an isomorphism.
						\end{lemma}
						
						\subsection{Explicit description of  Selmer Groups}\label{subsec: explicit selmer}
						Recall that  $ f \in S_{k}(\Gamma_{0}(N), \eta)$, $ h \in S_{l}(\Gamma_{0}(I), \psi)$ and let $ m $  be as in \eqref{definition m}.  From now on, we choose and fix $ \Sigma := \{ \ell : \ell \mid pNI \infty\}  $. Also recall $ \Sigma_0 := \{ \ell : \ell  \mid m\} $.  Let  $ \Sigma^{\infty} $ (resp. $ \Sigma_{0}^\infty $) be the set of primes in $ \Q_{\cyc} $ lying above $ \Sigma $ (resp. $ \Sigma_{0}$).
						In this subsection, we give a  more explicit description of  $ H^{1}_{\mathrm{Gr}}(\Q_{\cyc,w},A_j[\pi]) $ and $ H^{1}_{\mathrm{Gr}}(\Q_{\cyc,w},A_f(\xi_i\omega_p^{-j})[\pi])  $ for $  w \in \Sigma^{\infty} \setminus \Sigma^{\infty}_{0} $.
						\begin{proposition}\label{prop:Greenberg at p}
							Let $ w \in \Sigma^\infty$ and  $w \mid p $. Assume the order of $\psi |I_{\cyc,w}$ is co-prime to $p$.
							Then we have 
							\begin{align*}
								H^{1}_{\mathrm{Gr}} (\Q_{\cyc,w},A_j[\pi]) =
								\mathrm{Ker}\big(H^1(G_{\Q_{\cyc,w}}, A_{j}[\pi]) \lra H^1(I_{\cyc,w}, A^{-}_{j}[\pi])\big).
							\end{align*}
						\end{proposition}
						
						\begin{proof} 
							
							First note that by  definition of $ H^1_{\text{Gr}}(\Q_{\cyc,w},A_j[\pi]) $, we have the following exact sequence
							\begin{equation*}
								0\lra H^1_{\text{Gr}}(\Q_{\cyc,w},A_j[\pi]) \lra H^1(G_{\Q_{\cyc,w}},A_j[\pi]) \stackrel{\varphi}{\lra} H^1(I_{\cyc,w},A^-_j). 
							\end{equation*}			
							Consider the commutative diagram:
							\begin{small}
								\begin{equation}\label{dqGYaidhkAdoaU}
									\begin{tikzcd}
										&   & H^1(G_{\Q_{\cyc,w}},A_j[\pi]) \ar[r,"\nu"] \ar[d,"\cong"]  &  H^1(I_{\cyc,w},A^-_j[\pi])  \ar[d,"\epsilon"] \\
										0 \ar[r] & H^1_{\text{Gr}}(G_{\Q_{\cyc,w}},A_j[\pi]) \ar[r] & H^1(G_{\Q_{\cyc,w}},A_j[\pi]) \ar[r,"\kappa"] & H^1(I_{\cyc,w},A^-_j) .
									\end{tikzcd}
								\end{equation}
							\end{small}%
							From the diagram \eqref{dqGYaidhkAdoaU}, we see that $H^1_{\text{Gr}}(\Q_{\cyc,w},A_j[\pi]) = \text{ker}(\kappa) = \nu^{-1}(\text{ker}(\epsilon)).$ Thus it suffices to show $\text{ker}(\epsilon) =0 $. 
							As  $T_f^-$ is unramified at $ v $, we have  $H^1(I_{\cyc,w}, A_j^-[\pi]) \cong H^1(I_{\cyc,w}, A_h(j)[\pi]) \otimes T_f^-$ and $H^1(I_{\cyc,w}, A_j^-) \cong H^1(I_{\cyc,w}, A_h(j)) \otimes T_f^-$. Let $\epsilon'$ be the natural map  $H^1(I_{\cyc,w}, A_h(j)[\pi]) \lra H^1(I_{\cyc,w}, A_h(j))$. Further, it suffices to show ker$ (\epsilon') =0 $.
							
							As $ I_{\cyc, w} $ has $ p $-th cohomological dimension $ 1 $, $ H^{2}(I_{\cyc,w}, A_h^{-}(j)[\pi]) =0 = H^{2}(I_{\cyc,w}, A_h^{-}(j)) $ and  we have 
							\begin{small}
								\[
								\begin{tikzcd}
										H^0(I_{\cyc,w}, A_h^-(j)[\pi]) \ar[r] \ar[d,"\cong"] & H^1(I_{\cyc,w}, A_h^+(j)[\pi]) \ar[r]\ar[d,"\epsilon_1"] & H^1(I_{\cyc,w} ,  A_h(j)[\pi]) \ar[r]\ar[d,"\epsilon' "] &  H^1(I_{\cyc,w}, A_h^-(j)[\pi]) \ar[d, "\epsilon_2"]  \ar[r]& 0 \\
										H^0(I_{\cyc,w}, A_h^-(j))[\pi]  \ar[r]& H^1(I_{\cyc,w},  A_h^+(j))[\pi] \ar[r] & H^1(I_{\cyc,w},A_h(j))[\pi] \ar[r] & H^1(I_{\cyc,w}, A_h^-(j))[\pi] \ar[r]& 0 . 
									\end{tikzcd}
									\]
								\end{small}%
								Note as $A_h^{-}$ is unramified, $I_{\cyc,w}$ acts on the  $p$-primary group $A_h^-(j)[\pi]$ by the character $\omega_p^j$ which has order prime to $p$. Thus either $(A_h^{-}(j)[\pi])^{I_{\cyc,w}}  =0$ or we have $(A_h^{-}(j))^{I_{\cyc,w}}$ is divisible. In either case, it follows that $\frac{(A_h^{-}(j))^{I_{\cyc,w}}}{\pi} =0$ and hence $\epsilon_2$ is an isomorphism. As $\psi_h \omega_{p}^{l-1}|I_{\cyc,w}$ has order prime to $p$, a similar argument shows that $\epsilon_1$ is also an isomorphism. By the Five lemma, it follows that 
								ker$(\epsilon')$=0.  
							\end{proof}

				\begin{rem}\label{189}
					Assume $p >k$. Then note that $A_j^-[\pi] = (A_f^-\otimes T_h)(j) [\pi] \cong (A_f^-[\pi]\otimes T_h)(j) \cong (A_f^-[\pi]\otimes_{\mathcal{O}/\pi} \frac{T_h}{\pi})(j) \cong ((A_f[\pi])_{I_p}\otimes_{\mathcal{O}/\pi} A_h[\pi] )(j)$. The last equality is true as $p>k$ by \cite[(3.5), (3.6)]{p-selmer}.
					This  shows that if $p>k$ then  $A_j^-[\pi]$ is determined by $A_f[\pi] $ and $A_h[\pi]$.

				\end{rem}
				
				\begin{proposition}\label{prop:Greenberg away from N}
					Let  $ w \in \Sigma^\infty$, $ w \nmid N pm $  and $ w \mid I $.  Assume  $ \psi \lvert I_{\cyc,w} $ has order co-prime to $ p $. Then we have 
					\begin{align*}
						H^{1}_{\mathrm{Gr}} (\Q_{\cyc,w},A_j[\pi]) =
						\mathrm{Ker}\big(H^1(G_{\Q_{\cyc,w}}, A_{j}[\pi]) \lra H^1(I_{\cyc,w}, A_{j}[\pi])\big).	 	
					\end{align*}
				\end{proposition}
				\begin{proof}
					From Lemma~\ref{lem: exactly divides} and  \cite[Theorem 3.26(iii)(a)]{Hida3}, we have
					$\rho_{h}|_{I_{\cyc,w}} \sim \begin{psmallmatrix} \psi_h  & 0 \\   0 &1 \end{psmallmatrix}$.
					By the definition of $ H^1_{\text{Gr}}(G_{\Q_{\cyc,w}},A_j[\pi]) $, we have the following exact sequence
					\begin{equation*}
						0\lra H^1_{\text{Gr}}(\Q_{\cyc,w},A_j[\pi]) \lra H^1(G_{\Q_{\cyc,w}},A_j[\pi]) \longrightarrow H^1(I_{\cyc,w},A_j). 
					\end{equation*}
					The rest of the argument is similar to the proof of Proposition~\ref{prop:Greenberg at p}.
				\end{proof}

				\begin{proposition}\label{prop:Greenberg dividing N}
					Recall that $ M_0 $ is the prime to $ p $-part of  $ \mathrm{cond}(\bar{\xi}_1\bar{\omega}_p^{1-l}) \mathrm{cond}(\bar{\xi_2})$. Let $ w \in \Sigma^\infty $, $ w \mid N  $ and $ w \nmid pm $. Assume $ (N,M_0) =1 $. Then we have 
					\begin{small}
						\begin{align*}
							H^{1}_{\mathrm{Gr}} (\Q_{\cyc,w},A_j[\pi]) =
							\mathrm{Ker}\Big(H^1(G_{\Q_{\cyc,w}}, A_{j}[\pi]) \lra \Big(H^1(I_{\cyc,w}, A_{f}(j))[\pi] \otimes \frac{T_h}{\pi}\Big)^{\frac{G_{\Q_{\cyc,w}}}{I_{\cyc,w}}}\Big).
						\end{align*}
					\end{small}%
				\end{proposition}
				\begin{proof}
					Let $ \Delta_v = G_{\Q_{\cyc,w}}/I_{\cyc,w}  = \mathrm{Gal}(\Q_{\cyc,w}^{\text{unr}}/\Q_{\cyc,w}) $. Using inflation-restriction sequence, we obtain the image of $  H^1(G_{\Q_{\cyc,w}},A_j) \rightarrow H^1(I_{\cyc,w},A_j) $ lies in $H^1(I_{\cyc,w},A_j) ^{\Delta_v}$. From the  definition of $ H^1_{\text{Gr}}(\Q_{\cyc,w},A_j[\pi]) $, we have the following exact sequence
					\begin{small}
						\begin{equation*}
							0\lra H^1_{\text{Gr}}(G_{\Q_{\cyc,w}},A_j[\pi]) \lra H^1(G_{\Q_{\cyc,w}},A_j[\pi]) \xrightarrow{\varphi}  H^1(I_{\cyc,w},A_j)^{\Delta_v}. 
						\end{equation*}
					\end{small}%
					Since the image of $ \varphi $ is  $ \pi $-torsion, we get 
					\begin{small}
						\begin{equation}\label{eq:Greenberg away from N}
							0\lra H^1_{\text{Gr}}(G_{\Q_{\cyc,w}},A_j[\pi]) \lra H^1(G_{\Q_{\cyc,w}},A_j[\pi]) \stackrel{\varphi}{\longrightarrow} H^1(I_{\cyc,w},A_j)^{\Delta_w}[\pi] = \big(H^1(I_{\cyc,w},A_j)[\pi] \big)^{\Delta_w} . 
						\end{equation}
					\end{small}%
					As $ w \nmid M_0  pm $, we have $ w \nmid I  $. 
					Since $ A_j = A_{f} \otimes T_h (j)$ and $ T_h $ are unramified at $ v $, we obtain
					$
					H^1(I_{\cyc,w},A_j)[\pi] = \big(H^1(I_{\cyc,w},A_f(j))  \otimes T_{h} \big) [\pi]=  H^1(I_{\cyc,w},A_f(j))[\pi] \otimes T_h = H^1(I_{\cyc,w},A_f(j))[\pi] \otimes \frac{T_h}{\pi} $.
					Substituting this in \eqref{eq:Greenberg away from N}, we get the required description of $H^{1}_{\mathrm{Gr}} (\Q_{\cyc,w},A_j[\pi])$.          
				\end{proof}
				Recall that $ \xi_1 $ and $ \xi_{2} $ are Dirichlet characters whose reductions are $ \bar{\xi}_1 $ and $ \bar{\xi}_2 $ respectively. We now prove an analogues of Propositions~\ref{prop:Greenberg at p}-\ref{prop:Greenberg dividing N} for  $ f \otimes \xi_{1} $ and $ f \otimes \xi_{2} $.
				\begin{proposition}\label{prop:Greenberg for modular form}
					Let $ w \in \Sigma^\infty \setminus  \Sigma^\infty_{0}$. Then  
					\begin{enumerate}[label=$\mathrm{(\roman*)}$]
						\item  
						If $ w \mid p $, then for $ i=1,2 $,
						\begin{align*}
							H^{1}_{\mathrm{Gr}} (\Q_{\cyc,w},A_f(\xi_{i} \omega_p^{-j})[\pi]) =
							\mathrm{Ker}\big(H^1(G_{\Q_{\cyc,w}}, A_f(\xi_{i} \omega_p^{-j})[\pi]) \lra H^1(I_{\cyc,w}, A_{f}^{-}(\xi_{i} \omega_p^{-j})[\pi])\big). 
						\end{align*}
						\item If $ w \nmid Nm $ and $ w \mid I_0 $, then for $ i=1,2 $, 
						\begin{align*}
							H^{1}_{\mathrm{Gr}} (\Q_{\cyc,w},A_f(\xi_{i} \omega_p^{-j})[\pi]) = \mathrm{Ker}\big(H^1(G_{\Q_{\cyc,w}}, A_f(\xi_{i} \omega_p^{-j})[\pi]) \lra H^1(I_{\cyc,w}, A_f(\xi_{i} \omega_p^{-j})[\pi])\big).
						\end{align*}
						\item	Assume $ (N,M_0) =1 $. If $ w \mid N $ and $ w \nmid m $, then for $ i=1,2 $,
						\begin{small}
							\begin{align*}
									H^{1}_{\mathrm{Gr}} (\Q_{\cyc,w},A_f(\xi_{i} \omega_p^{-j})[\pi]) =
									\mathrm{Ker}\Big(H^1(G_{\Q_{\cyc,w}}, A_f(\xi_{i} \omega_p^{-j})[\pi])  \lra (H^1(I_{\cyc,w}, A_{f}(j))[\pi] \otimes \bar{\xi}_{i})^{\frac{G_{\Q_{\cyc,w}}}{I_{\cyc,w}}}\Big).%
							\end{align*}
						\end{small}%
					\end{enumerate}
				\end{proposition}
				\begin{proof}
					For $ i = 1,2$ and  $w \in \Sigma^{\infty} \setminus \Sigma^\infty_{0} $, we have the following commutative diagram  
					\begin{small}
						\[
						\begin{tikzcd}
							H^{1}(G_{\Q_{\cyc,w}}, A_f(\xi_i \omega_{p}^{-j})[\pi]) \arrow[r, "\nu"] \arrow[d, "i^{\ast}_1"]
							& H^{1}(I_{\cyc,w}, A_{f,w}(\xi_i \omega_{p}^{-j})[\pi]) \arrow[d, "\epsilon"] \\
							H^{1}(G_{\Q_{\cyc,w}}, A_{f}(\xi_i \omega_{p}^{-j})) \arrow[r, "\kappa"]
							& H^{1}(I_{\cyc,w}, A_{f,w}(\xi_i \omega_{p}^{-j})),
						\end{tikzcd}
						\]
					\end{small}%
					where $ A_{f,w}  = A_f^{-}$ if $ w \mid p $ and $ A_{f,w} = A_f $ if $ w \nmid p $. By definition, we have  $ H^{1}_{\mathrm{Gr}}(G_{\Q_{\cyc,w}},A_j[\pi]) = \mathrm{ker}(\kappa \circ i_1^{\ast}) = \mathrm{ker}(\epsilon \circ \nu) $. It suffices to show $\mathrm{ker}(\epsilon)  = A_{f,w}(\xi_i \omega_{p}^{-j})^{I_{\cyc,w}}/\pi A_{f,w}(\xi_i \omega_{p}^{-j})^{I_{\cyc,w}} $ is zero. 
					\begin{enumerate}[label=$\mathrm{(\roman*)}$]
						\item   Since  $ \xi_{2} $ is unramified at $ w $ and $ \omega_p $ is ramified at $ w $, we get $ \xi_{2} \omega_p^{-j}|I_p =1 \Longleftrightarrow (p-1) \mid j $. Thus $ A_{f,w}(\xi_2 \omega_{p}^{-j})^{I_{\cyc,w}} =  A_{f,w}(\xi_2 \omega_{p}^{-j}) $ if $ (p-1) \mid j $ and $ A_{f,w}(\xi_2 \omega_{p}^{-j})^{I_{\cyc,w}} = 0$ if $ (p-1) \nmid j $. So in either case, we have $A_{f,w}(\xi_2 \omega_{p}^{-j})^{I_{\cyc,w}}$ is $ \pi $-divisible and hence $ \mathrm{ker}(\epsilon) =0 $. Thus the assertion (i) for $ \xi_2 $ follows from the above commutative diagram. Since  $ \bar{\xi}_{1}$ has order prime to $ p $,  by construction (see the proof of Lemma~\ref{lemma lifting characters}) we have $\xi_{1}$ has order prime to $ p $. Thus $ \bar{\xi}_{1} \bar{\omega}_p^{1-l}|_{I_{\cyc,w}} = 1$ if and only if   $ \xi_{1} \omega_p^{1-l}|_{I_{\cyc,w}} = 1 \mod \pi$. A similar argument as above shows that $A_{f,w}(\xi_1 \omega_{p}^{-j})^{I_{\cyc,w}}$ is $ \pi $-divisible and $ \mathrm{ker}(\epsilon)=0 $. 
						\item Let $ r $ be prime in $ \Z $ lying below $ w $. Then $ r \nmid N p$. Since $ r \nmid m$ and $ r \mid I_0 $, we must have $ r \nmid I_0/M_0 $ and $ r \mid M_0 $. Since $ \mathrm{cond}_{r}(\omega_{p}^{-j}\xi_i) = \mathrm{cond}_{r}(\omega_{p}^{-j}\bar{\xi}_i) $ for $ i =1,2 $, we get $ \mathrm{cond}_{r}(\rho_{f \otimes \xi_{i}\omega_{p}^{-j}}) =  \mathrm{cond}_{r}(\xi_{i}\omega_{p}^{-j}) = \mathrm{cond}_{r}(\bar{\xi}_i\omega_{p}^{-j}) = \mathrm{cond}_{r}(\bar{\rho}_{f \otimes \xi_{i}\omega_{p}^{-j}}) $. Thus $ A_{f}(\xi_{i}\omega_p^{-j})^{I_w} $ is $ \pi$-divisible by \cite[Lemma 4.1.2]{epw}. So $ \mathrm{ker}(\epsilon) $ in the above commutative diagram is zero. Now part (ii) follows.
						\item Note that $ \xi_{i} $ is unramified at $ w $. The proof is similar to Proposition~\ref{prop:Greenberg dividing N}. \qedhere
					\end{enumerate} 
				\end{proof}
				\section{Congruence of the characteristic ideals}\label{sec: Congruences of char ideals}
				In this section, we obtain the congruence between the characteristic ideal associated to the Selmer group of $ f \otimes h $ and the product of characteristic ideals associated to the Selmer group of $ f \otimes \xi_1 $, $ f \otimes \xi_2$. 

				From the exact sequence \eqref{eq: splitting rho_h}, it follows that we have the following exact sequence of $G_\Q$-modules
				\begin{equation}\label{fund-resb}
					0 \lra \frac{\mathcal{O}}{\pi}(\bar{\xi}_1) \lra A_h[\pi] \lra \frac{\mathcal{O}}{\pi}(\bar{\xi}_2) \lra 0.
				\end{equation} 
				Tensoring with $-\otimes_{\mathcal{O}} T_f{(-j)} $ and noting that $ T_f{(-j)}  $ is a free $ \mathcal{O} $-module,  we get 
				\begin{equation}\label{fund-reds2}
					0 \lra A_{f}(\xi_1\omega_p^{-j})[\pi] \lra A_j[\pi] \lra A_{f}(\xi_2\omega_p^{-j})[\pi] \lra 0.
				\end{equation}
				In this section, we consider the following assumption on $f \otimes \xi_1$
				\begin{align}\label{H2van}
					H^2(\Q_\Sigma/\Q_{\cyc}, A_f(\xi_1\omega_p^{-j})[\pi]) =0.\tag{$H^{2}_{\xi_1}$-van}
				\end{align}
				\begin{lemma}\label{lem: cohomology exact sequence of Q_cyc}
					Assume  $f$ satisfies \ref{irr-f}, $h$ satisfies  \eqref{intro rho_h} and $f \otimes \xi_1$ satisfies \eqref{H2van}. Then the  following sequence is exact
					\begin{small}
						\begin{equation*}
							0 \rightarrow H^1(\Q_\Sigma/\Q_\cyc, A_f(\xi_1\omega_p^{-j})[\pi]) \rightarrow H^1(\Q_\Sigma/\Q_\cyc, A_j[\pi]) \rightarrow H^1(\Q_\Sigma/\Q_\cyc, A_f(\xi_2\omega_p^{-j})[\pi]) \rightarrow 0.
						\end{equation*}
					\end{small}%
				\end{lemma}
				\begin{proof}
					The exact sequence \eqref{fund-reds2} induces the following exact sequence of cohomology  groups
					\begin{small}
						\[
						H^1(\Q_\Sigma/\Q_\cyc, A_f(\xi_1\omega_p^{-j})[\pi]) \rightarrow H^1(\Q_\Sigma/\Q_\cyc, A_j[\pi]) \rightarrow H^1(\Q_\Sigma/\Q_\cyc, A_f(\xi_2\omega_p^{-j})[\pi]).
						\]  
					\end{small}%
					By the assumption $H^2(\Q_\Sigma/\Q_{\cyc}, A_f(\xi_1\omega_p^{-j})[\pi]) =0$, we obtain the right most map is surjective. Note that by our assumption  $A_f[\pi]$ is an irreducible $G_\Q$-module, it follows that $A_f(\xi_1\omega_p^{-j})[\pi]$ is  also an irreducible $G_\Q$-module, hence  $H^0(\Q,A_f(\xi_1\omega_p^{-j})[\pi])=0$. Thus by Nakayama lemma, $H^0(\Q_\cyc,A_f(\xi_1\omega_p^{-j})[\pi])$ also vanishes. Hence the left most map in the above exact sequence is injective.  
				\end{proof}
				We now state a technical assumption, which we will need later. 
				\begin{equation}\label{assumption}
					\text{ Let }  w  \text{ be the prime in } \Q_{\cyc} \text{ dividing } p. \text{ Then } \bar{\rho}_h |I_{\cyc,w} \text{ is semi-simple i.e.  } \bar{\rho}_h|I_{\cyc,w} \cong \bar{\xi}_1 \oplus \bar{\xi}_2  
					\tag{ss-red$_p$}. 
				\end{equation}
				
				Recall that $ \Sigma = \{ \ell: 
				\ell \mid pNI\infty \}$, $ \Sigma_{0} = \{ \ell : \ell \mid  m \}$ and $ \Sigma^\infty , \Sigma_0^\infty$ are the corresponding  primes in  $ \Q_{\cyc} $.
				
				\begin{proposition}\label{prop:exact sequence of inertia}
					Let $ (N,M_0) =1$  and $ j $ be an integer.  Assume $f$ satisfies \ref{irr-f}, $h$ satisfies  \eqref{intro rho_h} and $ \psi \vert_{I_{\cyc,w}}  $ has order prime to $ p $ for $  w \mid pI_0$. If $ (p-1) \mid j $ and $H^0(G_{\Q_{\cyc,w}},  A^-_f(\xi_2)[\pi]) \neq 0$ at the prime $ w \mid p $, then  we further  assume that  \eqref{assumption} holds.
					Then  for every  $ w \in \Sigma^\infty \setminus \Sigma^\infty_{0} $, we have the following exact sequence
					\begin{equation}\label{split-local-c2}
						0 \rightarrow \frac{H^1(\Q_{\cyc,w}, A_f(\xi_1\omega_p^{-j})[\pi])}{H_{\mathrm{Gr}}^1(\Q_{\cyc,w},A_f(\xi_1\omega_p^{-j})[\pi])} \rightarrow \frac{H^1(\Q_{\cyc,w}, A_j[\pi])}{H^1_{\mathrm{Gr}}(\Q_{\cyc,w}, A_j[\pi])} \rightarrow \frac{H^1(\Q_{\cyc,w}, A_f(\xi_2\omega_p^{-j})[\pi])}{H^1_{\mathrm{Gr}}(\Q_{\cyc,w},A_f(\xi_2\omega_p^{-j})[\pi])} .
					\end{equation}
				\end{proposition} 
				\begin{proof}
					As the $ p $-th cohomological dimension of $ G_{\Q_{\cyc,w}} $ is $ 1 $, from \eqref{fund-reds2} we get the exact sequence
					\begin{align}\label{195}
						H^1(G_{\Q_{\cyc,w}}, A_f(\xi_1\omega_p^{-j})[\pi]) \rightarrow H^1(G_{\Q_{\cyc,w}}, A_j[\pi])  \rightarrow H^1(G_{\Q_{\cyc,w}}, A_f(\xi_2\omega_p^{-j})[\pi]) \rightarrow 0.
					\end{align}
					We denote $ G_{\Q_{\cyc,w}}/I_{\cyc,w} $ by $ \Delta_w $. 
					
					\underline{Case $w \mid p$}: As the $ p $-th cohomological dimension of $ G_{\Q_{\cyc,w}} $ is $ 1 $, the surjectivity $ H^1(G_{\Q_{\cyc,v}}, A_j[\pi]) \twoheadrightarrow H^1(G_{\Q_{\cyc,v}}, A^{-}_j[\pi])$ follows from the definition  $A_j^-$  in \eqref{def A minus j}. 
					By the inflation-restriction sequence, we have the $ H^1(G_{\Q_{\cyc,w}}, A^{-}_j[\pi]) \rightarrow H^1(I_{\cyc,w}, A^{-}_j[\pi])^{\Delta_w}$ is surjective.  Hence the image of the following composition 
					\begin{align*}
						H^1(G_{\Q_{\cyc,w}}, A_j[\pi]) \rightarrow H^1(G_{\Q_{\cyc,w}}, A^{-}_j[\pi])) \rightarrow H^1(I_{\cyc,w}, A^{-}_j[\pi]))
					\end{align*} 
					is equal to  $H^1(I_{\cyc,w}, A^{-}_j[\pi])^{\Delta_w}$. Similarly 
					image $\big( H^1(G_{\Q_{\cyc,w}}, A_f(\xi_{i}\omega_{p}^{-j})[\pi] \rightarrow H^1(I_{\cyc,w}, A^{-}_f(\xi_{i}\omega_{p}^{-j})[\pi]\big) = (H^1(I_{\cyc,w}, A^{-}_f(\xi_{i}\omega_{p}^{-j})[\pi])^{\Delta_w}$ for $ i=1,2 $.
					Thus it follows from  \eqref{fund-resb} and  Propositions  \ref{prop:Greenberg at p}, \ref{prop:Greenberg for modular form} that the  following  diagram
					\begin{small}
						\begin{equation}\label{commutative diagram}
							\begin{tikzcd} 
								0  \ar[d] & 0 \ar[d]  & 0 \ar[d] \\
								H^1_{\mathrm{Gr}}(G_{\Q_{\cyc,w}}, A_f(\xi_1\omega_p^{-j})[\pi]) \ar[r]  \ar[d] & H_{\mathrm{Gr}}^1(G_{\Q_{\cyc,w}}, A_j[\pi]) \ar[r]  \arrow{d} 
								& H_{\mathrm{Gr}}^1(G_{\Q_{\cyc,w}}, A_f(\xi_2\omega_p^{-j})[\pi])\ar[d]   \ar[r] & 0\\ 
								H^1(G_{\Q_{\cyc,w}}, A_f(\xi_1\omega_p^{-j})[\pi]) \ar[r]  \ar[d] & H^1(G_{\Q_{\cyc,w}}, A_j[\pi]) \ar[r]  \ar[d] & H^1(G_{\Q_{\cyc,w}}, A_f(\xi_2\omega_p^{-j})[\pi])\ar[d]  \ar[r]& 0  \\
								H^{1}(I_{\cyc,w}, A^{-}_{f}(\xi_1\omega_p^{-j})[\pi])^{\Delta_w} \ar[r,"\nu"] \ar[d] &  H^{1}(I_{\cyc,w}, A^{-}_{j}[\pi])^{\Delta_v}  \ar[r] \ar[d] &  H^{1}(I_{\cyc,w}, A^{-}_{f}(\xi_2\omega_p^{-j})[\pi])^{\Delta_w}. \ar[d] & \\
								0 & 0 & 0
							\end{tikzcd}
						\end{equation}
					\end{small}%
					is commutative. So to prove the proposition, it suffices  to show  $ \mathrm{ker}(\nu) = 0$.
					
					If $(p-1) \nmid j$, then $H^{0}(I_{\cyc,w},A_{f}^{-}(\xi_{2}\omega_{p}^{-j})[\pi]) = 0  $.  Tensoring \eqref{fund-resb} with $ A_{f}^{-} $ and then taking the induced long exact sequence in cohomology, we obtain 
					\begin{align*}
						0 \rightarrow  H^1(I_{\cyc,w}, A^{-}_f(\xi_{1}\omega_{p}^{-j})[\pi])   \rightarrow 
						H^1(I_{\cyc,w}, A^{-}_j[\pi])    \rightarrow H^1(I_{\cyc,w}, A^{-}_f(\xi_{2}\omega_{p}^{-j})[\pi]) \rightarrow 0. 
					\end{align*} 
					Taking $ \Delta_w $ invariants, we obtain $ \mathrm{ker}(\nu) =0 $.
					
					Next consider the case $ (p-1) \mid j$ and  $H^0(G_{\Q_{\cyc,w}}, A^{-}_f(\xi_{2})[\pi]) \neq 0 $.  Then by the assumption 
					\eqref{assumption},  we have  $ A_{h}[\pi] \cong (T_h/\pi)|I_{\cyc,w} \cong \bar{\xi}_{1} \oplus \bar{\xi}_{2} $. Thus we have $H^1(I_{\cyc,w}, A^{-}_j[\pi]) = H^1(I_{\cyc,w}, A^{-}_f(\xi_{1}\omega_{p}^{-j})[\pi]) \oplus  H^1(I_{\cyc,w}, A^{-}_f(\xi_{2}\omega_{p}^{-j})[\pi])$. Again taking $ \Delta_w $ invariants, we get $ \mathrm{ker}(\nu) =0 $.
					
					Finally consider  the case $ (p-1) \mid j $ and $ H^0(G_{\Q_{\cyc,w}}, A^{-}_f(\xi_{2})[\pi]) = 0 $. Then $ H^{0}(\Delta_w,A_{f}^{-}(\xi_{2})[\pi]^{I_{\cyc,w}}) =0 $. As $ \Delta_w = G_{\Q_{\cyc,w}}/I_{\cyc,w}  $ is  procyclic, we get  $ H^{1}(\Delta_w,A_{f}^{-}(\xi_{2})[\pi]^{I_{\cyc,w}}) =0 $.  
					Thus by the inflation-restriction sequence and \eqref{fund-resb}, in this case we have the following commutative diagram 
					\begin{small}
						\[
						\begin{tikzcd}
							0 \ar[r] & H^1(\Delta_w, A^{-}_f(\xi_1 \omega_p^{-j})[\pi]^{I_{\cyc,w}}) \ar[r]  \ar[d] & H^1(\Delta_w, A^{-}_j[\pi]^{I_{\cyc,w}}) \ar[r]  \ar[d] & H^1(\Delta_w, A^{-}_f(\xi_2\omega_p^{-j})[\pi]^{I_{\cyc,w}}) =0 \ar[d] & \\
							& H^1(G_{\Q_{\cyc,w}}, A^{-}_f(\xi_1\omega_p^{-j})[\pi]) \ar[r] & H^1(G_{\Q_{\cyc,w}}, A^{-}_j[\pi]) \ar[r]   & H^1(G_{\Q_{\cyc,w}}, A^{-}_f(\xi_2\omega_p^{-j})[\pi]) \ar[r] & 0 .
						\end{tikzcd}
						\]
					\end{small}%
					where all the vertical maps are injective.  Now the proposition  follows from the snake lemma and the inflation-restriction sequence.

					\underline{Case $w \mid I$, $ w \nmid Npm $}: 
					Let $ w \mid \ell $  in $ \Z $. Since $ \ell \nmid m$ and $ \ell \mid I_0 $, it forces that  $ \ell \nmid I_0/M_0 $ and $ \ell \mid M = \mathrm{cond}(\xi_1)\mathrm{cond}(\xi_2)$. By Lemma~\ref{lem: exactly divides} and \cite[Theorem 3.26]{Hida3}, we have $ \bar{\rho}_h | G_{\Q_{\cyc,w}} \cong \bar{\xi}_1 \oplus \bar{\xi}_2$. Tensoring with $ A_f(-j) $ induces following  exact sequence of $ \Delta_w $-modules 
					\[
					0 \rightarrow H^1(I_{\cyc,w}, A_f(\xi_1 \omega_p^{-j})[\pi]) \lra H^1(I_{\cyc,w}, A_j[\pi]) \lra H^1(I_{\cyc,w}, A_f(\xi_2\omega_p^{-j})[\pi]) \rightarrow 0.
					\]  
					Note that we have an analogue of commutative diagram \eqref{commutative diagram} even in this case. So the assertion follows from Proposition~\ref{prop:Greenberg away from N} and Proposition  \ref{prop:Greenberg for modular form}(ii) by a similar argument as in the case $ w \mid p $.
					
					\underline{Case $w \mid N$ and $ w \nmid Im $}: We claim that $ (T_h/\pi)  \cong \bar{\xi}_1 \oplus \bar{\xi}_2$ as $ G_{\Q_{\cyc,w}}  $-modules.  Since $ (N, M_0) =1$ and $ w \mid N $, we get $ w \nmid  M_0$. Again using  $ v \nmid m $, we obtain $ w \nmid I_0 $. Thus $ T_h/\pi $ is unramified at $ w $ and the action of $ G_{\Q_{\cyc, v}} $ factors through $ \Delta_w$. As the order of $ \bar{\xi}_1 , \bar{\xi}_2 $ are co-prime to $ p $, the action of $ G_{\Q_{\cyc, w}} $  on $ T_{h}/\pi $ further factors through  $ G_{\Q_{\cyc, w}}/B$  such that $B$ is a finite index normal subgroup of $ G_{\Q_{\cyc, w}} $ containing
					$ I_{\cyc,w} $. Since $ G_{\Q_{\cyc, w} } /B $ has order co-prime to $ p $, we get $ T_{h}/\pi $ is a semi-simple $ G_{\Q_{\cyc, w} }  $-module. Hence $  T_{h}/\pi \cong \bar{\xi}_1 \oplus \bar{\xi}_2 $ as $G_{\mathbb{Q}_{\mathrm{cyc}, w}}$-module.  Now the proposition in this case follows from Propositions~\ref{prop:Greenberg dividing N}, \ref{prop:Greenberg for modular form} using similar arguments as in previous two cases.
				\end{proof}
				\begin{remark}
					We will apply Proposition~\ref{prop:exact sequence of inertia} in Theorem~\ref{thm:congruence of ideals} and Theorem~\ref{congruence main conjecture }  for  $ l-1 \leq j \leq k-2 $. In that case,  the assumption \eqref{assumption} in  Proposition~\ref{prop:exact sequence of inertia} is required whenever $  \omega_{p}^{j} =1$ 
					and $H^0(G_{\Q_{\cyc,w}},  A^-_f(\xi_2)[\pi]) \neq 0$. In particular, either  $ p > k-1 $ or $H^0(G_{\Q_{\cyc,w}},  A^-_f(\xi_2)[\pi]) = 0$ holds, then the assumption \eqref{assumption} is not required.   
				\end{remark}
				\begin{proposition}\label{prop: selmer exact sequence}
					Let $ (N,M_0) =1$ and $ \psi \vert_{I_{\cyc,w}}  $ has order co-prime to $ p $ for all primes $ w \mid pI_0$ in $\Q_{\cyc}$.  Assume  that $f$ satisfies \ref{irr-f}, $h$ satisfies  \eqref{intro rho_h} and $f \otimes \xi_1$ satisfies \eqref{H2van}. 
					If $ (p-1) \mid j $ and $H^0(G_{\Q_{\cyc,w}},  A^-_f(\xi_2)[\pi]) \neq 0$ at the prime $ w \mid p $, then  we further  assume that  \eqref{assumption} holds for $w$.   Then we have the following  exact sequence
					\begin{equation}\label{split-selmer-01}
						0 \rightarrow S^{\Sigma_0}_{\mathrm{Gr}}( A_f(\xi_1\omega_p^{-j})[\pi]/\Q_\cyc) \rightarrow S^{\Sigma_0}_{\mathrm{Gr}}( A_j[\pi]/\Q_\cyc) \rightarrow S^{\Sigma_0}_{\mathrm{Gr}}( A_f(\xi_2\omega_p^{-j})[\pi]/\Q_\cyc) \lra 0.
					\end{equation} 
				\end{proposition}
				\begin{proof}
					By Lemma~\ref{lem: cohomology exact sequence of Q_cyc}  and Proposition~\ref{prop:exact sequence of inertia}, we have the following commutative diagram 
					\begin{tiny}{
							\[
							\begin{tikzcd}
								0 \arrow{r}  &   H^1(\Q_\Sigma/\Q_\cyc, A_f(\xi_1\omega_p^{-j})[\pi]) \arrow{r}  \arrow{d}{\epsilon_1} & H^1(\Q_\Sigma/\Q_\cyc, A_j[\pi]) \arrow{r}  \arrow{d}{\epsilon} & H^1(\Q_\Sigma/\Q_\cyc, A_f(\xi_2\omega_p^{-j})[\pi])\arrow{r}  \arrow{d}{\epsilon_2} & 0\\
								0 \arrow{r}  &   \prod\limits_{w \in \Sigma^\infty \smallsetminus \Sigma^\infty_{0}} \frac{H^{1}(\mathbb{Q}_{\cyc,w}, A_{f}(\xi_1\omega_p^{-j})[\pi])}{ H^{1}_{\mathrm{Gr}}(\mathbb{Q}_{\cyc,w}, A_{f}(\xi_1\omega_p^{-j})[\pi])} \arrow{r}  & \prod\limits_{w \in \Sigma^\infty \smallsetminus \Sigma^\infty_{0}} \frac{H^{1}(\mathbb{Q}_{\cyc,w}, A_{j}[\pi])}{H_{\mathrm{Gr}}^{1}(\mathbb{Q}_{\cyc,w}, A_{j}[\pi])}  \arrow{r} & \prod\limits_{w \in \Sigma^\infty \smallsetminus \Sigma^\infty_{0}} \frac{H^{1}(\mathbb{Q}_{\cyc,w}, A_{f}(\xi_2\omega_p^{-j})[\pi])}{H_{\mathrm{Gr}}^{1}(\mathbb{Q}_{\cyc,w}, A_{f}(\xi_2\omega_p^{-j})[\pi])}.   
							\end{tikzcd}
							\]
					}\end{tiny}
					Applying the snake lemma to the above commutative diagram, we get the following exact sequence
					\begin{small}
						\begin{equation*}
							0 \rightarrow S^{\Sigma_0}_{\mathrm{Gr}}( A_f(\xi_1\omega_p^{-j})[\pi])/\Q_\cyc) \rightarrow S^{\Sigma_0}_{\mathrm{Gr}}( A_j[\pi]/\Q_\cyc) \rightarrow S^{\Sigma_0}_{\mathrm{Gr}}( A_f(\xi_2\omega_p^{-j})[\pi]/\Q_\cyc) \rightarrow \mathrm{coker}(\epsilon_1).
						\end{equation*}
					\end{small}%
					To prove the lemma, we need to show  coker($\epsilon_1$) is zero. 
					In this setting,  the following global to local map  
					\begin{small}{
							\begin{equation*}\label{surj-f-sel3}
								H^1(\Q_\Sigma/\Q_\cyc, A_f(\xi_1\omega_p^{-j})[\pi]) \lra \underset{w \mid p}{\bigoplus}H^1(\mathbb{Q}_{\cyc,w}, A^-_f(\xi_1\omega_p^{-j})[\pi]) \underset{w \in \Sigma\setminus \Sigma^\infty_{0} , ~w \nmid p}{\bigoplus} H^1(\mathbb{Q}_{\cyc,w},A_f(\xi_1\omega_p^{-j})[\pi]) 
							\end{equation*} 
					}\end{small}%
					is surjective. Indeed, the surjectivity of such global to local maps has been extensively studied by Greenberg (cf. \cite[\S5.3]{gr3}). In our setting above, it can also be conveniently found in  \cite[Theorem 5.2]{ms}. Using the inflation-restriction sequence and $ \Delta_w $ is topologically cyclic, we obtain the following  surjective map  
					\begin{small}{
							\begin{equation*}
								H^1(\Q_\Sigma/\Q_\cyc, A_f(\xi_1\omega_p^{-j})[\pi]) \lra \underset{v \mid p}{\bigoplus}H^1(I_{\cyc,w}, A^-_f(\xi_1\omega_p^{-j})[\pi])^{\Delta_w} \bigoplus_ {w \in \Sigma^\infty\setminus \Sigma^\infty_{0}, w \nmid p} H^1(\mathbb{Q}_{\cyc,w},A_f(\xi_1\omega_p^{-j})[\pi])  
							\end{equation*} 
					}\end{small}
					Since $\frac{H^1(\mathbb{Q}_{\cyc,w}, A_f(\xi_1\omega_p^{-j})[\pi])}{H^1_{\mathrm{Gr}}(\mathbb{Q}_{\cyc,w}, A_f(\xi_1\omega_p^{-j})[\pi])} = H^1(I_{\cyc,w}, A^-_f(\xi_1\omega_p^{-j})[\pi])^{\Delta_w} $ for $ w \mid p $ (see Diagram \eqref{commutative diagram}), it follows that 
					\begin{small}{
							\begin{equation*}\label{surj-f-sel2}
								H^1(\Q_\Sigma/\Q_\cyc, A_f(\xi_1\omega_p^{-j})[\pi]) \lra \underset{w \mid p}{\bigoplus}\frac{H^1(\mathbb{Q}_{\cyc,w}, A^-_f(\xi_1\omega_p^{-j})[\pi])}{H^1_{\mathrm{Gr}}(\mathbb{Q}_{\cyc,w}, A_f(\xi_1\omega_p^{-j})[\pi])} \underset{\substack{w \in \Sigma^\infty \setminus \Sigma^\infty_{0} \\ w \nmid p}}{\bigoplus} \frac{H^1(\mathbb{Q}_{\cyc,w},A_f(\xi_1\omega_p^{-j})[\pi])}{H^1_{\mathrm{Gr}}(\mathbb{Q}_{\cyc,w},A_f(\xi_1\omega_p^{-j})[\pi])}. 
							\end{equation*} 
					}\end{small}%
					is also surjective.
					This shows $ \epsilon_1 $ is surjective. Hence the right exactness follows in \eqref{split-selmer-01}.
				\end{proof}
				
				
				
				\begin{lemma}\label{lem: dual selmer group sequence}
					We keep the setting and hypotheses as in Proposition~\ref{prop: selmer exact sequence}. 
					Then we have the following exact sequence
					\begin{small}%
						\begin{equation}\label{split-selmer-03}
							0 \rightarrow \frac{S^{\Sigma_0}_{\mathrm{Gr}}( A_f(\xi_2\omega_p^{-j})/\Q_\cyc)^\vee}{\pi} \rightarrow \frac{S^{\Sigma_0}_{\mathrm{Gr}}( A_j/\Q_\cyc)^\vee}{\pi} \rightarrow \frac{S^{\Sigma_0}_{\mathrm{Gr}}(A_f(\xi_1\omega_p^{-j})/\Q_\cyc)^\vee}{\pi}  \rightarrow 0.
						\end{equation}
					\end{small}%
				\end{lemma}
				\begin{proof}
					We claim that  $H^0(G_{\Q_\cyc}, A_f(\xi_i \omega_p^{-j}))=0$, for $i=1,2$ and $H^0(G_{\Q_\cyc}, A_j)=0$. From the assumption  $A_f[\pi]$ (and hence $A_f(\xi_i\omega_p^{-j})[\pi]$) is an irreducible  $G_\Q$-module, we have $H^0(G_{\Q}, A_f(\xi_i\omega_p^{-j})[\pi])=0$, for $i=1,2$. From Nakayama's Lemma, it follows that $H^0(G_{\Q_\cyc}, A_f(\xi_i \omega_p^{-j})[\pi])=0$, for $i=1,2$. 
					From the exact sequence \eqref{fund-reds2}, it follows that $H^0(G_{\Q_\cyc}, A_j[\pi])=0$. 
					For $B_j \in \{  A_f(\xi_{1}\omega_{p}^{-j}), A_j, A_f(\xi_{2}\omega_{p}^{-j})\}$, applying Lemma \ref{inside-out-lem} we get $S^{\Sigma_0}_{\mathrm{Gr}}( B_j[\pi]/\Q_\cyc) = S^{\Sigma_0}_{\mathrm{Gr}}( B_j/\Q_\cyc)[\pi]$. Then it follows from  \eqref{split-selmer-01} that 
					\begin{equation}\label{split-selmer-02}
						0 \rightarrow S^{\Sigma_0}_{\mathrm{Gr}}( A_f(\xi_1\omega_p^{-j})/\Q_\cyc)[\pi] \rightarrow S^{\Sigma_0}_{\mathrm{Gr}}( A_j/\Q_\cyc)[\pi] \rightarrow S^{\Sigma_0}_{\mathrm{Gr}}( A_f(\xi_2\omega_p^{-j})/\Q_\cyc)[\pi]  \rightarrow 0
					\end{equation}
					is exact. Taking the Pontryagin dual, we obtain the lemma.
				\end{proof}
				
				As $f$ is ordinary at $p$, it follows from a deep result of Kato \cite[Theorem 17.4]{kato}  that 
				\begin{align*}\label{kato torsion}
					\text{The dual Selmer groups } S^{\Sigma_0}_{\mathrm{Gr}}( A_f(\xi_i\omega_p^{-j})/\Q_\cyc)^\vee \text{ are torsion } \mathcal{O}[[\Gamma]]\text{-modules for } i=1,2.\tag{co-tors}
				\end{align*}
				By the assumption \ref{Stors}, we have $S^{\Sigma_0}_{\mathrm{Gr}}( A_j/\Q_\cyc)^\vee$ is also a torsion $\mathcal{O}[[\Gamma]]$-module. Thus $S_{\mathrm{Gr}}(B_j/\Q_\cyc)$ is  co-torsion $\mathcal{O}[[\Gamma]]$-module for $B_j \in \{ A_{f}(\xi_1\omega_p^{-j}), A_j , A_{f}(\xi_2\omega_p^{-j})\}$.  We next prove that $S_{\mathrm{Gr}}(B_j/\Q_\cyc)^\vee$ has no  non-zero pseudo-null $\mathcal{O}[[\Gamma]]$ submodules.
				
				\begin{proposition}\label{no-finite-part-result1}
					Let  $ B_{j} \in \{ A_{f}(\xi_1\omega_p^{-j}), A_j, A_{f}(\xi_2\omega_p^{-j})\}$. Assume  that  $f$ satisfies \ref{irr-f} and  \ref{Stors} holds for $A_j$.
					Then $S_{\mathrm{Gr}}(B_j/\Q_\cyc)^\vee$ has no non-zero  pseudo-null $\mathcal{O}[[\Gamma]]$-submodule. 
				\end{proposition}
				\begin{proof}
					
					We   apply \cite[Proposition 1.8]{we} to deduce $S_{\mathrm{Gr}}(B_j/\Q_\cyc)^\vee$ has no non-zero  pseudo-null $\mathcal{O}[[\Gamma]]$ submodule. Following \cite{we}, define $\mathrm{Sel}_{\mathrm{cr}}(\Q_{\cyc},B_j) := \mathrm{Ker} \big(H^{1} (G_{\Sigma}(\Q_{\cyc}), B_j) \rightarrow \oplus_{w \in \Sigma^\infty} Loc_w \big)$,
					where $ Loc_w = H^{1}({\mathbb{Q}_{\cyc,w}}, B_j)$ if $ w \nmid p $ and $ Loc_w = (\mathrm{image}( H^{1}({\mathbb{Q}_{\cyc,w}}, B_j) \rightarrow  H^{1}(I_{\cyc,w}, B_j^{-}) ))$ if $ w \mid p  $. For $w \nmid p$, $G_{\Q_{\cyc,w}}/I_{\cyc,w}$ has profinite order prime to $ p $,  thus $H^1(G_{\Q_{\cyc,w}}/I_{\cyc,w}, B_{j}^{I_{\cyc,w}}) =0$.  Thus by the inflation-restriction sequence we get $H^{1}({\mathbb{Q}_{\cyc,w}}, B_j) \rightarrow  H^{1}(I_{\cyc,w}, B_j) $ is injective. With this, it is easy to see that  $\mathrm{Sel}_{\mathrm{cr}}(\Q_{\cyc},B_j) = S_{\mathrm{Gr}}(B_j/\Q_\cyc)$. 
					
					We now claim that $\text{Hom}(B_j(-1)[\pi],\mathbb{F})^{G_\Q}=  0$.   Since $A_f[\pi]$ is an irreducible  $G_\Q$-module, we have $\text{Hom}(A_f(\xi_i \omega_p^{-j-1})[\pi],\mathbb{F})^{G_\Q}=  0$.  Thus from the exact sequence \eqref{fund-reds2} it follows that  $\text{Hom}(A_j(-1)[\pi],\mathbb{F})^{G_\Q}=  0$. 
					Since $\big(\mathrm{Hom}(B_j, \frac{K_\mathfrak p}{\mathcal{O}}(1))\otimes \frac{K_\mathfrak p}{\mathcal{O}}\big) [\pi] \cong   \text{Hom}(B_j(-1)[\pi], \mathbb{F})$, we get  $ (\mathrm{Hom}(B_j, \frac{K_\mathfrak p}{\mathcal{O}}(1))\otimes \frac{K_\mathfrak p}{\mathcal{O}} )^{G_\Q} [\pi] = 
					0$. Note that  $\big(\mathrm{Hom}(B_j, \frac{K_\mathfrak p}{\mathcal{O}}(1))\otimes \frac{K_\mathfrak p}{\mathcal{O}}\big)^{G_\Q}=0$ implies that $H^0(G_{\Q_\cyc}, \mathrm{Hom}(B_j, \frac{K_\mathfrak p}{\mathcal{O}}(1))\otimes \frac{K_\mathfrak p}{\mathcal{O}})=0$ by Nakayama's lemma.  
					Now the proposition follows from \cite[Proposition 1.8]{we}.
				\end{proof}
				Since $S_{\mathrm{Gr}}(B_j/\Q_\cyc)$ is a co-torsion $\mathcal{O}[[\Gamma]]$-module, we have  $S^{\Sigma_0}_{\mathrm{Gr}}( B_j/\Q_\cyc)^\vee$ is a finitely generated torsion $\mathcal{O}[[\Gamma]]$-module for $B_j \in \{ A_f(\xi_1\omega_p^{-j}), A_j, A_f(\xi_2\omega_p^{-j})\}$.
				We next study relation between $S_{\mathrm{Gr}}( B_j/\Q_\cyc)$ and $S^{\Sigma_0}_{\mathrm{Gr}}( B_j/\Q_\cyc)$.
				\begin{lemma}
					Let the hypotheses be as in Lemma~\ref{lem: dual selmer group sequence} and also assume \ref{Stors} holds for $A_j$. Then for $B_j \in \{ A_f(\xi_1\omega_p^{-j}), A_j, A_f(\xi_2\omega_p^{-j})\}$, we have the following exact sequence
					\begin{equation}\label{last-alg1}
						0 \lra S_{\mathrm{Gr}}( B_j/\Q_\cyc) \lra S^{\Sigma_0}_{\mathrm{Gr}}( B_j/\Q_\cyc) \lra \underset{w \in \Sigma^\infty_0}{\prod}H^1(\Q_{\cyc,w}, B_j) \lra 0.
					\end{equation}  
				\end{lemma}
				\begin{proof}
					For every $ w \in \Sigma^\infty_0$,  $G_{\Q_{\cyc,w}}/I_{\cyc,w}$ is a pro-$\ell$ group with $\ell\neq p$. Thus $H^1(G_{\Q_{\cyc,w}}/I_{\cyc,w}, B_j^{I_{\cyc,w}}) =0$ and $H^{1}_{\mathrm{Gr}}(G_{\Q_{\cyc,w}}, B_j) =0$ for $w \in \Sigma^\infty_0$. We first treat the case $B_j = A_f(\xi_1\omega_p^{-j}), A_f(\xi_2\omega_p^{-j})$. Let $\mu_{p^{\infty}} := \{ \zeta \in \C: \zeta^{p^r} =1 \text{ for some } r \in \Z \}$ and $A_{f}(\xi_i \omega_p^{-j})^\ast := \Hom(T_{f}(\xi_i\omega_p^{-j}), \mu_p^\infty)$. From \ref{irr-f} and $\dim_{\mathcal{O}}(T_f) =2$, we have  $\Hom(T_{f}(\xi_i\omega_p^{-j})/\pi, \mu_p^\infty)$ is an irreducible $G_{\Q}$-module. By topological Nakayama lemma, $H^{0}(G_{\Q_{\cyc}}, \Hom(T_{f}(\xi_i\omega_p^{-j})/\pi, \mu_p^\infty)) = 0$. As $\Hom(T_{f}(\xi_i\omega_p^{-j})/\pi, \mu_p^\infty) \cong \Hom(T_{f}(\xi_i\omega_p^{-j}), \mu_p^\infty)[\pi] $, it follows that $ H^{0}(G_{\Q_{\cyc}}, A_{f}(\xi_i\omega_p^{-j})^\ast[\pi])=0$. Thus we obtain $ H^{0}(G_{\Q_{\cyc}}, A_{f}(\xi_i\omega_p^{-j})^\ast)^\vee/\pi =0$. Since $(A_{f}(\xi_i\omega_p^{-j})^\ast)^\vee$ is a finitely generated $\Z_p$-module, 
					by Nakayama lemma, we get $H^{0}(G_{\Q_{\cyc}}, A_{f}(\xi_i\omega_p^{-j})^\ast)^\vee =0$. Now the result in this case follows from  \cite[Corollary 2.3]{gv} (See also \cite[Proposition 3.14]{SU}).
					
					We next treat the case $B_j = A_j$. Tensoring \eqref{eq: splitting rho_h} with $T_f(j)$, we obtain that
					\begin{small}
						\begin{align*}
							0 \rightarrow T_{f}(\xi_1 \chi_{\cyc}^{-j})/\pi \rightarrow T_{j}/\pi \rightarrow T_{f}(\xi_2 \chi_{\cyc}^{-j})/\pi \rightarrow 0.
						\end{align*} 
					\end{small}%
					Applying $\Hom_{\Z_{p}}(-, \mu_{p^\infty})$ and taking $G_{\Q_{\cyc}}$-invariants, we obtain  
					\begin{small}
						\begin{align*}
							0 \rightarrow \Hom_{\Z_{p}}(T_{f}(\xi_2 \chi_{\cyc}^{-j})/\pi, \mu_{p^\infty})^{G_{\Q_{\cyc}}} \rightarrow \Hom_{\Z_{p}}(T_j/\pi, \mu_{p^\infty})^{G_{\Q_{\cyc}}} \rightarrow \Hom_{\Z_{p}}(T_{f}(\xi_1 \chi_{\cyc}^{-j})/\pi, \mu_{p^\infty})^{G_{\Q_{\cyc}}}.
						\end{align*}
					\end{small}%
					Since $\Hom_{\Z_{p}}(T_{f}(\xi_i \chi_{\cyc}^{-j})/\pi, \mu_{p^\infty})^{G_{\Q_{\cyc}}} = H^{0}(G_{\Q_{\cyc}},\Hom_{\Z_{p}}(T_{f}(\xi_i \chi_{\cyc}^{-j})/\pi, \mu_{p^\infty})) =0$ for $i=1,2$, it follows that $H^{0}(G_{\Q_{\cyc}},\Hom_{\Z_{p}}(T_j/\pi, \mu_{p^\infty})) =0$. By a similar argument as above,  we get $H^0(G_{\Q_{\cyc}}, \Hom(T_j, \mu_{p}^{\infty}))$ vanishes.  Now the lemma, under \ref{Stors} can be deduced adopting the arguments in \cite[Corollary 2.3]{gv}. 
				\end{proof}
				We now show that the $\Sigma_0$-imprimitive Selmer groups $S^{\Sigma_0}_{\mathrm{Gr}}( B_j/\Q_\cyc)^\vee$ have no non-zero pseudo-null $ \mathcal{O}[[\Gamma]] $ submodules.
				\begin{corollary}\label{lem:psuedonullity}
					Let the hypotheses be as in Lemma~\ref{lem: dual selmer group sequence} and also assume \ref{Stors} holds for $A_j$. 
					Then the  modules $ S^{\Sigma_0}_{\mathrm{Gr}}( A_f(\xi_i\omega_p^{-j})/\Q_\cyc)^\vee$ for $ i=1,2 $ and $S^{\Sigma_0}_{\mathrm{Gr}}( A_j/\Q_\cyc)^\vee$ have no non-zero  pseudo-null $ \mathcal{O}[[\Gamma]] $ submodules.
				\end{corollary}
				\begin{proof}
					Let $B_j \in \{ A_f(\xi_1\omega_p^{-j}), A_j, A_f(\xi_2\omega_p^{-j})\}$.  We note that, $\prod_{w \in \Sigma^\infty_0} H^1(\Q_{\cyc,w}, B_j)^\vee$ has no non-zero pseudo-null submodule as $B_j$ is divisible and $G_{\Q_{\cyc,w}}$ has $p$-th cohomological dimension 1. 
					Now the  lemma follows from Proposition~\ref{no-finite-part-result1} and \eqref{last-alg1}.	   
				\end{proof}
				
				For a finitely generated torsion  $\mathcal{O}[[\Gamma]]$ (resp. $\mathbb{F}[[\Gamma]]$) module $\mathcal{M}$, we denote the characteristic ideal of $ \mathcal{M} $ over  $\mathcal{O}[[\Gamma]]$ (resp. $\mathbb{F}[[\Gamma]]$) by $C_{\mathcal{O}[[\Gamma]]}(\mathcal{M})$ (resp. $C_{ \mathbb{F}[[\Gamma]]}(\mathcal{M})$). 
				Note that by the assumption \ref{Stors} we have $S_{\mathrm{Gr}}( A_j/\Q_\cyc)^\vee $ is a torsion $\mathcal{O}[[\Gamma]]$-module. Also by \eqref{kato torsion}, we have $S^{\Sigma_0}_{\mathrm{Gr}}( A_f(\xi_i\omega_p^{-j})/\Q_\cyc)^\vee$ are finitely generated torsion $\mathcal{O}[[\Gamma]]$-modules for $i=1,2$.  Thus we can consider their characteristic ideal. 
			From Lemma~\ref{lem:psuedonullity} and \cite[Corollary 3.8(i), Corollary 3.21(iii)]{SU} we have 
			\begin{equation} 
				C_{\mathcal{O}[[\Gamma]]}(S^{\Sigma_0}_{\mathrm{Gr}}( B_j/\Q_\cyc)^\vee) ~(\text{mod }\pi) = C_{ \mathbb{F}[[\Gamma]]}(S^{\Sigma_0}_{\mathrm{Gr}}(B_j/\Q_\cyc)^\vee/{\pi})
			\end{equation}
			for $B_j \in \{ A_f(\xi_1\omega_p^{-j}), A_j,  A_f(\xi_2\omega_p^{-j}) \} $. Using this in \eqref{split-selmer-03}, we deduce the following theorem:
			
			\begin{theorem}\label{thm:congruence of ideals}
				Let $ (N,M_0) =1$ and  $\psi  \lvert I_{\cyc,w}$ has order prime to $p$ for $ w \mid pI_0$, where $\psi$ is the nebentypus of $h$. 
				Assume  $f$ satisfies \ref{irr-f}, $h$ satisfies  \eqref{intro rho_h}, $f \otimes \xi_1$ satisfies \eqref{H2van} and \ref{Stors} holds for $A_j$. If $ (p-1) \mid j $ and $H^0(G_{\Q_{\cyc,w}},  A^-_f[\pi](\bar{\xi}_2)) \neq 0$ at the prime $ w \mid p $, then  we further assume that  \eqref{assumption} holds.  Then for $ l-1 \leq j \leq k-2 $, we have
				\begin{small}{
						\begin{equation}\label{split-selmer-04}
							C_{\mathcal{O}[[\Gamma]]}\Big(S^{\Sigma_0}_{\mathrm{Gr}}( A_j/\Q_\cyc)^\vee\Big) = C_{\mathcal{O}[[\Gamma]]}\Big(S^{\Sigma_0}_{\mathrm{Gr}}( A_f(\xi_1 \omega_p^{-j})/\Q_\cyc)^\vee\Big)  C_{\mathcal{O}[[\Gamma]]}\Big(S^{\Sigma_0}_{\mathrm{Gr}}( A_f(\xi_2 \omega_p^{-j})/\Q_\cyc)^\vee\Big)~  \mod   \pi.
						\end{equation}
				}\end{small}%
			\end{theorem}
			\begin{remark}
				As $ f $ is $ p $-ordinary and $ \bar{\rho}_f $ is irreducible, it follows the $C_{\mathcal{O}[[\Gamma]]}\Big(S^{\Sigma_0}_{\mathrm{Gr}}( A_f(\xi_i\omega_p^{-j})/\Q_\cyc)^\vee\Big)$ is independent of the choice of the lattice $ T_f $ (See \cite[Page 34]{SU}). Thus by Theorem~\ref{thm:congruence of ideals}, it follows that $C_{\mathcal{O}[[\Gamma]]}\Big(S^{\Sigma_0}_{\mathrm{Gr}}( A_j/\Q_\cyc)^\vee\Big) $ is also independent of the choice of the lattice $ T = T_f \otimes T_g $.
			\end{remark}

			\section{Iwasawa main conjecture modulo \texorpdfstring{$\pi$}{}}\label{sec: IMC}
			Recall  that $\mu^{\Sigma_{0}}_{p,f, \varphi,j}$ and  $\mu^{\Sigma_{0}}_{p,f\times h,j}$ are the $ p $-adic $ L $-functions attached to $ f \otimes \varphi $ and $f\otimes h,$ (see \eqref{def: p-adic L-function f}, \eqref{def: p-adic L-function f x h}). We now recall the Iwasawa Main Conjecture for modular forms (See Greenberg \cite{gr1}):
			\begin{conjecture}\label{IWC for modular form}$($\textbf{Greenberg Iwasawa Main Conjecture for modular forms}$)$
				Let $ F \in S_{k}(\Gamma_{1}(N)) $ be a $p$-ordinary eigenform and let $\varphi$ be a Dirichlet character. For every critical value $ 0 \leq j \leq k-1 $, we have 
				\begin{align*}\label{IMC-modular}
					(\mu_{p,F,\varphi, j}) = C_{\mathcal{O}[[\Gamma]]}\Big(S_{\mathrm{Gr}}( A_F(\varphi\omega_p^{-j})/\Q_\cyc)^\vee\Big).\tag{IMC}
				\end{align*}
			\end{conjecture}
			
			\begin{remark}
				Conjecture~\ref{IWC for modular form} is known for a large class of modular forms combining the results of Kato \cite{kato} and Skinner-Urban \cite{SU}. 
			\end{remark} 
			We next show that the Iwasawa main conjecture i.e. \eqref{IMC-modular}  for $f \otimes \xi_i$ implies the $\Sigma_0$-imprimitive Iwasawa main conjecture for $f \otimes \xi_i$.
			\begin{lemma}\label{IMC imprimitive for modular form}
				Let $ f \in S_{k}(\Gamma_{0}(N),\eta) $ be a $p$-ordinary eigenform with $p \nmid N$ and $f$ satisfies \ref{irr-f}. For every critical value $ 0 \leq j \leq k-1 $ and $i=1,2$, we have 
				\begin{align*}
					(\mu_{p,f,\xi_i, j}) = C_{\mathcal{O}[[\Gamma]]}\Big(S_{\mathrm{Gr}}( A_f(\xi_i\omega_p^{-j})/\Q_\cyc)^\vee\Big) \Rightarrow (\mu^{\Sigma_0}_{p,f,\xi_i, j}) = C_{\mathcal{O}[[\Gamma]]}\Big(S^{\Sigma_0}_{\mathrm{Gr}}( A_f(\xi_i\omega_p^{-j})/\Q_\cyc)^\vee\Big).
				\end{align*}
			\end{lemma}
			\begin{proof}
				By \cite[Proposition 2.4]{gv} (see also \cite[Lemma 3.13]{SU}), for every prime $\ell \neq p$, the characteristic polynomial of $ \prod_{w \mid \ell} H^{1}(\mathbb{Q}_{\cyc,w}, A_{f}(\xi_i \omega_p^{-j}))$ equals $ P_{i,\ell}(\ell^{-j-1}\gamma) $, where $ P_{i,\ell}(X) = \det(I -X \mathrm{Frob}_{\ell}\big\vert{V_f(\xi_i)^{I_{\ell}}} ) $ and  $ \gamma $ is a topological generator of $ \Gamma = \mathrm{Gal}(\Q_{\cyc}/\Q) $. 
				
				Note that from the exact sequence \eqref{last-alg1}, $S_{\mathrm{Gr}}( A_f(\xi_i\omega_p^{-j})/\Q_\cyc)^\vee$ is torsion over $\mathcal{O}[[\Gamma]]$ implies the same for $S_{\mathrm{Gr}}^{\Sigma_0}( A_f(\xi_i\omega_p^{-j})/\Q_\cyc)^\vee$. Also it follows that
				\begin{small}
					\[ C_{\mathcal{O}[[\Gamma]]}\Big(S^{\Sigma_0}_{\mathrm{Gr}}( A_f(\xi_i\omega_p^{-j})/\Q_\cyc)^\vee\Big) = \prod_{\ell \in \Sigma_0}P_{i,\ell}(\ell^{-j-1}\gamma)  C_{\mathcal{O}[[\Gamma]]}\Big(S_{\mathrm{Gr}}( A_f(\xi_i \omega_p^{-j})/\Q_\cyc)^\vee\Big).\]
				\end{small}%
				Observe that $P_{i,\ell}(X) = \det(I -X \mathrm{Frob}_{\ell}\big\vert{V_{f}(\xi_i)^{I_{\ell}}} )$ gives the Euler factor of the $L$-function of $f \otimes \xi_i$ at $\ell$. By \eqref{def: p-adic L-function f}, we have $ \mu^{\Sigma_0}_{p,f, \xi,j} = \mu_{p,f, \xi,j}\prod_{\ell \in \Sigma_{0}} P_{i,\ell}(\ell^{-j-1}\gamma) $. Now multiplying both sides of  $(\mu_{p,f,\xi_i, j}) = C_{\mathcal{O}[[\Gamma]]}\Big(S_{\mathrm{Gr}}( A_f(\xi_i\omega_p^{-j})/\Q_\cyc)^\vee\Big)$ by $\prod_{\ell \in \Sigma_0} P_{i,\ell}(\ell^{-j-1}\gamma)$, the  lemma follows. 
			\end{proof}
			
			We now state our main result on the Iwasawa Main Conjecture for Rankin-Selberg $ L $-function  mod $ \pi $. By the assumption \ref{Stors}, we have  $S_{\mathrm{Gr}}( A_j/\Q_\cyc)^\vee$ is a torsion $\mathcal{O}[[\Gamma]]$-module (also see Remark~\ref{remark mu invariant and H2}) 
			\begin{theorem}\label{congruence main conjecture } 
				We assume that \eqref{IMC-modular} holds for  $ f \otimes \xi_1$ and $ f \otimes \xi_{2}$ with $ l-1 \leq  j \leq k-2$. Let $ f \in S_{k}(\Gamma_{0}(N),\eta), h \in S_{l}(\Gamma_{0}(I),\psi) $ be $ p $-ordinary newforms. We assume all the hypotheses of Theorem~\ref{analytic final} as well as Theorem~\ref{thm:congruence of ideals} hold. Then for every $ l-1 \leq j \leq k-2 $, we have 
				\begin{align}\label{IMC congruence Rankin}
					(\mu_{p,f\times h,j}) \equiv C_{\mathcal{O}[[\Gamma]]}\Big(S_{\mathrm{Gr}}( A_j/\Q_\cyc)^\vee\Big) \mod \pi.
				\end{align}
			\end{theorem}
			
			\begin{proof}
				By Theorem~\ref{analytic final}, we have 
				$ (\mu^{\Sigma_{0}}_{p,f \times  h,j})  \equiv  (\mu^{\Sigma_{0}}_{p,f,\xi_1,j}) 
				(\mu^{\Sigma_{0}}_{p,f, \xi_2,j}) \mod \pi$, for all $ l-1 \leq j \leq k-2 $.  By \eqref{IMC-modular} and Lemma~\ref{IMC imprimitive for modular form}, we have $(\mu^{\Sigma_{0}}_{p,f, \xi_i, j} ) = C_{\mathcal{O}[[\Gamma]]}(S^{\Sigma_0}_{\mathrm{Gr}}( A_f(\xi_i \omega_p^{-j})/\Q_\cyc)^\vee)$, for $i=1,2$.  Thus we have $(\mu^{\Sigma_0}_{p,f \times h,j}) \equiv C_{\mathcal{O}[[\Gamma]]}(S^{\Sigma_0}_{\mathrm{Gr}}( A_f(\xi_1 \omega_p^{-j})/\Q_\cyc)^\vee) C_{\mathcal{O}[[\Gamma]]}(S^{\Sigma_0}_{\mathrm{Gr}}( A_f(\xi_2 \omega_p^{-j})/\Q_\cyc)^\vee) \mod \pi$. Now it follows from Theorem~\ref{thm:congruence of ideals} that
				\begin{equation}\label{Sigma congruence}
					(\mu^{\Sigma_0}_{p,f \times h,j}) 
					\equiv C_{\mathcal{O}[[\Gamma]]}\Big(S^{\Sigma_0}_{\mathrm{Gr}}( A_j/\Q_\cyc)^\vee\Big) \mod \pi.
				\end{equation}
				By a similar argument as in Lemma~\ref{IMC imprimitive for modular form}, using \cite[Proposition 2.4]{gv} and \eqref{last-alg1}, we obtain 
				\begin{small}	
					\begin{align*}
						C_{\mathcal{O}[[\Gamma]]}\Big(S^{\Sigma_0}_{\mathrm{Gr}}( A_j/\Q_\cyc)^\vee\Big) 
						=C_{\mathcal{O}[[\Gamma]]}\Big(S_{\mathrm{Gr}}( A_j/\Q_\cyc)^\vee\Big)  \prod_{\ell \in \Sigma_{0}} P_{\ell}(\ell^{-j-1}\gamma) , \quad 
						\forall ~~ l-1 \leq j \leq k-2.
					\end{align*}
				\end{small}%
				where $ P_{\ell}(X) = \det(I -X \mathrm{Frob}_{\ell}\big\vert{V^{I_{\ell}}} ) $ and  $ \gamma $ is a topological generator of $ \Gamma = \mathrm{Gal}(\Q_{\cyc}/\Q) $.  Note that the Euler factor  for the Rankin-Selberg $ L $-function of $f\otimes h$ at $ \ell $ is given by $ P_{\ell}(X)$. Also from \eqref{p-adic Rankin Sigma and p-adic Rankin}, we have $ \mu^{\Sigma_0}_{p,f \times h,j} = \mu_{p,f \times h,j} \prod_{\ell \in \Sigma_{0}} P_{\ell}(\ell^{-j-1}\gamma) $.  Following the proof of  \cite[Proposition 2.4]{gv}, we deduce the polynomial $P_{\ell}(X)$ has $ \mu$-invariant  zero, that is, $ P_{\ell}(X) \not \equiv 0 \mod \pi$. Cancelling  $P_{\ell}(\ell^{-j-1}\gamma)$  on either side of the congruence \eqref{Sigma congruence}, the desired congruence modulo $ \pi $ follows.
			\end{proof}
			
			In a special case when either $ \mu_{p,f\times h,j}  $ or $ C_{\mathcal{O}[[\Gamma]]}(S_{\mathrm{Gr}}( A_j/\Q_\cyc)^\vee) $ is a unit, the  congruence in \eqref{IMC congruence Rankin} leads to the following: 
			\begin{corollary}\label{cor: unit implies unit}
				Let the setting be as in Theorem~\ref{congruence main conjecture }. Then $ \mu_{p,f\times h,j}  $  is a unit in the Iwasawa algebra $ \mathcal{O}[[\Gamma]] $ if and only if $ C_{\mathcal{O}[[\Gamma]]}\Big(S_{\mathrm{Gr}}( A_j/\Q_\cyc)^\vee\Big) $ is a unit in the Iwasawa algebra $ \mathcal{O}[[\Gamma]] $.
			\end{corollary}
			
			\section{Examples}\label{section: examples}
			
			In this section, we illustrate Theorem~\ref{congruence main conjecture } and Corollary~\ref{cor: unit implies unit}. 
			\begin{example}\label{example 1}
				We take $p=11$. Let  $  \Delta = \sum_{n=1}^{\infty} \tau(n) q^n  $ be the Ramanujan Delta function. We have $ \Delta \in S_{12}(\mathrm{SL}_{2}(\mathbb{Z})) $ and 
				$ \tau(11) $ is $ 11 $-adic unit. Consider the quadratic character $ \chi_{K} = (\frac{-23}{\cdot})$ associated to the field $ K = \mathbb{Q}(\sqrt{-23}) $. Let $ f := \Delta \otimes \chi_{K} $. Then $ f \in S_{12}(\Gamma_{0}(23^2), \chi_{K}^2)$ is a primitive $p$-ordinary form. 
				Moreover, the residual representation $ \bar{\rho}_{\Delta}$ is an irreducible $ G_{\mathbb{Q}} $-module. 
				
				Having chosen $f$, we  next choose the candidate for the normalised cusp form $h$. 
				
				Consider the Eisenstein series $g'(z) = E_2(z) - 23E_2(23z)$, where $ E_{2}(z) = \frac{1}{24} + \sum_{n=1}^{\infty} \sigma(n) q^n $. Then by a result of Mazur  \cite[Proposition 5.12(ii)]{Mazur}, there exists a normalised Hecke  eigenform $h \in S_2(\Gamma_{0}(23))$ with LMFDB label 23.2.a such that
				$
				a(n,h) \equiv a(n,g') \mod \mathfrak{p},
				$
				where $ \mathfrak{p} = (11, 4 - \sqrt{5}) \subset \mathbb{Z}[(1 + \sqrt{5})/2]$. By [DS, Theorem 9.6.6], we have $ \bar{\rho}_{h} \simeq \bar{\rho}_{g'} \simeq \bar{\omega}_{p} \oplus 1 $. Thus $\bar{\xi}_{1}=\bar{\omega}_{p}$ and $\bar{\xi}_{2}=1$. By our construction, $\xi_{1}= \omega_{p} $ and $\xi_{2} = 1$. Thus by \eqref{definition of Eis g}, we have $g(z)=E (1_p ,1)(z)= E_2(z) - pE_2(pz)$, where $1_p$ is the trivial character modulo $p = 11$. As $ \mathrm{cond}(\rho_{h}) =23 $ and $ \mathrm{cond}(\bar{\rho}_{h}) = 11  $, we obtain $ m = 23 $ by \eqref{definition m}. Thus $\Sigma=\{ 11, 23, \infty\} $ and $\Sigma_0 = \{ 23 \}$.
				
				As $\bar{\rho}_{\Delta}$ is irreducible and $p$-distinguished it follows that $ \bar{\rho}_{f_0}  (= \bar{\rho}_{\Delta} \otimes \bar{\chi} )$ is also irreducible and $ p $-distinguished. 
				Hence the hypotheses of Theorem~\ref{analytic final} are satisfied.

				Next we check  the hypotheses of Theorem~\ref{thm:congruence of ideals}. We have $ N = N_f = 23^2 $, $ M_0 =1 $, $ (N,M_0) =1 $ and the nebentypus of $ h $ i.e. $\psi$ is trivial. Also $ A_{f}[\pi] = A_{\Delta}[\pi] \otimes \chi_K $ is an irreducible $ G_\mathbb{Q} $-module. Note that $ \mathrm{Frob}_{p} $ acts on $ A_{f}^{-}(\xi_{2}) $ by the scalar $ a(p,f) \xi_{2}(p) = \chi_{K}(p) a(p,\Delta) = - \tau(p) = -\tau(11) \neq 1 \mod \pi $. Thus $H^0(I_{\cyc,v}, A_f^{-}(\xi_2)) = 0$, so we need not check \eqref{assumption}. 
				
				It remains to check that \ref{Stors} and \eqref{H2van} hold to conclude all the conditions in Theorem~\ref{thm:congruence of ideals}  are met. For this, by Remark~\ref{remark mu invariant and H2}, it suffices to  show that $\mu(S^{\Sigma_0}_{\mathrm{Gr}}(A_{f}(\xi_{i}\omega_{p}^{-j})/\mathbb{Q}_{\cyc})^\vee) = 0 $ for $i=1,2$.  Moreover, as $\xi_1 = \omega_p$ and $\xi_2 =1$, it further reduces to show $ \mu(S^{\Sigma_0}_{\mathrm{Gr}}(A_{f}(\omega_{p}^{-j})/\mathbb{Q}_{\cyc})^\vee) = 0$ for $1 \leq j \leq 10$. 
				
				Consider the elliptic curve $E := X_0(11)$ of conductor $11$ over $\mathbb{Q}$ (LMFDB label 11.2.a.a) defined by
				$$
				y^2 + y = x^3 - x^2 -10x - 20.
				$$
				Let $ f_E $ be the modular form corresponding to $ E $ under modularity theorem. For $p=11$, there is a $ \Lambda $-adic form  whose specialisation  at $ k=2 $ is  $ f_E $  and at $ k=12 $ is $ \Delta $ \cite[Page 409]{gs}.  Hence the residual Galois representation of $ E $ and $ \Delta $ at $ p=11 $ are isomorphic. Since  $ f $ and $ f_{E} \otimes \chi_K $ lie in the same branch of Hida family, 
				it is enough to know that $ \mu (S_{\mathrm{Gr}}(f_{E} \otimes \chi_K \omega_{p}^{-j}/\mathbb{Q}_{\cyc})^\vee) =0$ (cf. \cite[Theorem 4.3.3]{epw}). 
				
				By a deep result of  Kato \cite[Theorem 17.4]{kato} and \cite[Theorem 4.3.3]{epw} it follows
				\begin{align}\label{comparison mu invariants}
					\mu (S_{\mathrm{Gr}}(f_{E} \otimes \chi_K \omega_{p}^{-j}/\mathbb{Q}_{\cyc})^\vee) = \mu(S_{\mathrm{Gr}}(A_{f}(\omega_{p}^{-j})/\mathbb{Q}_{\cyc})^\vee) \leq \mu^{\mathrm{an}}(f \otimes \omega_{p}^{-j})
					= \mu^{\mathrm{an}}(f_{E} \otimes \chi_K \omega_{p}^{-j}).
				\end{align}
				Thus it suffices to show that  the analytic $ \mu $-invariant vanishes i.e. the $ p $-adic $ L $-function $ \mu_{p, f_{E},  \chi_K, j}(T)$ is not divisible in the Iwasawa algebra $\mathcal{O}[[T]]$ by  $ p $ for all $0 \leq j \leq 9$. We now compute  $ \mu_{p,f_{E}, \chi_K,j} ( 0) $ for $0 \leq j \leq 9$. 
				
				Let the periods $\Omega^{\pm}_{f_E \otimes \chi_{K}}$ be as in Theorem~\ref{choice of period and petterson innerproduct}. From \cite{MTT} (cf. \cite[Section 5.1]{C}), we get
				\begin{small}
					\begin{align}\label{eq: modular symbol and L-value}
						\mu_{p, f_{E}, \chi_{K}, j}(0) = \begin{cases} \frac{1}{2\alpha}
							\sum_{b=1}^{p-1} \bar{\omega}_{p}^{j}(b) \text{x}^{sign(\omega_{p}^j)}(b/p) &\text{if} ~
							\omega_{p}^{j} \neq 1, \\
							(1-\alpha^{-1})^2 \frac{L(1,f_{E}, \chi_{K})}{\Omega^{+}_{f_E \otimes \chi_{K}}}
							&\text{otherwise},
						\end{cases}
					\end{align}
				\end{small}%
				where $ \text{x}^{\pm} $ are the $ \pm $ modular symbols associated to $ f_{E} \otimes \chi_K  $, as defined by Manin and $ \alpha $ is a root of $ p^{\mathrm{th}} $-Hecke polynomial of $ f_{E}\otimes \chi_{K} $ with $ |\alpha|_p =1$. Using SAGE, we compute $ \text{x}^{\pm} $. The values are 
				\begin{small}
					\begin{center}
						\begin{tabular}{| c | c |  c |  c  | c | c | c | c | c | c | c |}
							\hline
							$\frac{b}{11}$  & $1/11$ & $ 2/11 $ & $ 3/11 $ & $ 4/11 $ & $ 5/11 $ & $ 6/11 $ & $ 7/11 $ & $ 8/11 $ & $ 9/11 $ & $ 10/11 $ \\ 
							\hline
							$ \text{x}^{+} $ &  2 &  0 & 5 & 5 & 0 & 0 & 5 & 5 & 0 & 2 \\  
							\hline
							$ \text{x}^{-} $  & 0 &  0 & -5 & 5 & 0 & 0 & -5 & 5 & 0 & 0 \\ 
							\hline  
						\end{tabular}
					\end{center}
				\end{small}%
				Let  $\zeta_{10}$ be a primitive root of unity and $ \omega_{p}(2) = \zeta_{10} $ and $ \omega_{p}(-1) = -1 $. Then it follows $ \omega_{p}(3)  = \omega_{p}(2)^3 \omega_{p}(-1) = -  \zeta_{10}^3 $,  $\omega_{p}(4) =   \zeta_{10}^2 $ and $\omega_{p}(5) = \omega_{p}(2)^4 = \zeta_{10}^4 =  \zeta_{10}^3 - \zeta_{10}^2 +\zeta_{10} -1 $. Thus from \eqref{eq: modular symbol and L-value}, 
				\begin{small}
					\begin{align*}
						\mu_{p,f_{E}, \chi_{K}, j} (0) \sim \begin{cases}
							(2 +5 \omega_p^{j}(3)  +5  \omega_p^{j}(4) ), &\text{if } j ~ \text{is even and } 10 \nmid j, \\
							5(- \omega^{j}(3)  +  \omega_p^{j}(4)) , &\text{if } j ~ \text{is odd},
						\end{cases}
					\end{align*}
				\end{small}%
				here $ \sim $ denotes up to multiplication by a $p$-adic unit. Taking $ j = 0,1, 2,\ldots, 9 $ we obtain 
				\begin{small}
					\begin{alignat*}{3}
						\mu_{p, f_{E}, \chi_{K}, 1}(0)& \sim  5 \zeta_{10}^3 + 5 \zeta_{10}^2 
						~~& \mu_{p, f_{E}, \chi_{K}, 2}(0) & \sim 5 \zeta_{10}^3 - 5 \zeta_{10}^2 - 3 \\
						\mu_{p, f_{E}, \chi_{K}, 3}(0)&\sim -5 \zeta_{10}^3 +5  \zeta_{10}^2 -10\zeta_{10} +5
						~~~&\mu_{p, f_{E}, \chi_{K}, 4}(0) &\sim -5 \zeta_{10}^3 + 5 \zeta_{10}^2 + 2 \\
						\mu_{p, f_{E}, \chi_{K}, 5}(0) & = 0 \qquad \qquad
						~~& \mu_{p, f_{E}, \chi_{K}, 6}(0) &\sim -5 \zeta_{10}^3 + 5 \zeta_{10}^2 + 2 \\ 
						\mu_{p, f_{E}, \chi_{K}, 7}(0) & \sim 5 \zeta_{10}^3 -5 \zeta_{10}^2 +10 \zeta_{10} -5   \quad
						~~&  \mu_{p, f_{E}, \chi_{K}, 8}(0) &\sim 5 \zeta_{10}^3 - 5 \zeta_{10}^2  - 3 \\
						\mu_{p, f_{E}, \chi_{K}, 9}(0) &\sim -5 \zeta_{10}^3 - 5  \zeta_{10}^2   
						~~& \mu_{p, f_{E}, \chi_{K}, 0}(0) & \sim 1.
					\end{alignat*}
				\end{small}%
				For $ j=5 $, we have $E \otimes \chi_K \omega_p^j$ is an elliptic curve with LMFDB label 64009.d2. Further $ \mu_{p, f_{E}, \chi_{K}, 5}(T) = T u(T) $, for some $ u(T) \in \mathcal{O}[[T]]^{\times} $ using  SAGE.  Thus $\mu^{\text{an}}(f_E \otimes \chi_K \omega_p^j) =0$ for $1 \leq j \leq 10$. Thus we have completed checking all the hypotheses of Theorem~\ref{thm:congruence of ideals}.
				
				Finally, we need to check the Iwasawa main conjecture for $f \otimes \xi_i \omega_p^j$ for $i=1,2$ to verify Theorem~\ref{congruence main conjecture }.  
				We have $ \prod_{j \neq 5} \mu_{p, f_{E}, \chi_{K}, j}(0) \sim 3003125$, which is a $ 11 $-adic unit. 
				This shows that  $\prod_{j=0}^{9}\mu_{p, f_{E}, \chi_{K}, j}(T)= Tu(T)$ for some unit  $ u(T) \in \mathcal{O}[[T]] $.  By \eqref{comparison mu invariants}, we get  $\mu(S_{\mathrm{Gr}}(A_{f_E \otimes \chi_K}(\omega_{p}^{-j})/\mathbb{Q}_{\cyc})^\vee) = \mu^{\text{an}}(f_E \otimes \chi_K \omega_p^{-j}) = 0 $. We now compute the $\lambda^{\mathrm{an}}(f_E \otimes \chi_K \omega_p^{-j})$ and $\lambda^{\mathrm{alg}}(f_E \otimes \chi_K \omega_p^{-j})$ namely the analytic and algebraic $\lambda$-invariants of $f_E \otimes \chi_K \omega_p^{-j}$.  Once again, it follows from the above computation that $\lambda^{\mathrm{an}}(f_E \otimes \chi_K \omega_p^{-j}) = 0$ for $j \neq 5$ and $\lambda^{\mathrm{an}}(f_E \otimes \chi_K \omega_p^{-j}) = 1$ for $j=5$. By a similar argument as in \eqref{comparison mu invariants}, we get $\lambda^{\mathrm{alg}}(f_E \otimes \chi_K \omega_p^{-j}) = \lambda^{\mathrm{an}}(f_E \otimes \chi_K \omega_p^{-j}) = 0$ for $1 \leq j \leq 10$ and $j \neq 5$. 
				From the above calculations, we have $\lambda^{\mathrm{an}}(f_E \otimes \chi_K \omega_p^{-5}) =1$. Since the $j$-invariant of the elliptic curve attached to elliptic curve associated to $f_E \otimes \chi_K \omega_{p}^5$ is not an integer, it follows from \cite[Theorem 1]{nekovar}(applied for $k,k_0, k'$ there equal to $\Q$) that $\lambda^{\mathrm{alg}}(f_E \otimes \chi_K \omega_p^{-5}) \equiv 1 \mod 2 $. By the result of Kato $\lambda^{\mathrm{alg}}(f_E \otimes \chi_K \omega_p^{-5}) \leq \lambda^{\mathrm{an}}(f_E \otimes \chi_K \omega_p^{-5})=1$. Hence  $\lambda^{\mathrm{alg}}(f_E \otimes \chi_K \omega_p^{-5}) = \lambda^{\mathrm{an}}(f_E \otimes \chi_K \omega_p^{-5})=1$ and the Greenberg Iwasawa main conjecture holds for $ f_{E} \otimes \chi_{K}\omega_p^{-5}$. 
				Thus the Greenberg Iwasawa main conjecture holds for $ f_{E} \otimes \chi_{K} $ (cf. \cite[Theorem 5.1.2]{epw}). By \cite[Corollary 5.1.4]{epw}, it follows that Iwasawa main conjecture also holds for $f \otimes \omega_p^j$. 
				
				Hence for $ f = \Delta \otimes \chi_{K} $ and $h \in S_{2}(\Gamma_0(23)) $ of LMFDB label 23.2.a, the hypotheses of Theorem~\ref{congruence main conjecture } are met.  Therefore by  same theorem, we conclude that the Iwasawa main conjecture holds $ f \otimes h $ modulo $ \pi $, that is,
				\begin{equation}\label{IMC mod p for example 1}
					(\mu_{p, f \times h, j}) = (C_{\mathcal{O}[[\Gamma]]}(S_{\mathrm{Gr}}(A_j)/\mathbb{Q}_{\cyc})^\vee) \mod \pi \quad \mathrm{for} ~ 1 \leq j \leq 10.
				\end{equation}
				
				Further $ \mu^{\mathrm{an}}(f\otimes \omega_{p}^{j}) = 0 $ and $\lambda^{\mathrm{an}}(f \otimes  \omega_p^j) = 0$ for $ j \neq 5 $, it follows that  $\mu_{p, f, \xi_i,j}(T)  $ is unit whenever $ j \neq 5$. Applying Theorem~\ref{analytic final},  we obtain 
				$ \mu_{p, f \otimes h, j} $ is a unit whenever $ j \neq 4, 5 $. Thus, it follows from \eqref{IMC mod p for example 1} that $ C_{\mathcal{O}[[\Gamma]]}(S_{\mathrm{Gr}} (A_j/\mathbb{Q}_{\cyc})^\vee) $ is a unit in $ \mathcal{O}[[\Gamma]] $ for $ j \neq 4,5 $ and 
				\begin{equation}\label{IMC example 1}
					(\mu_{p, f \times h, j}) = (C_{\mathcal{O}[[\Gamma]]}(S_{\mathrm{Gr}}(A_j)/\mathbb{Q}_{\cyc})^\vee)  = \mathcal{O}[[T]] \quad \mathrm{for} ~ 1 \leq j \leq 10 \text{ and } j \neq 4,5.
				\end{equation}
				Further, if $ j =5 $ we have $ \mu^{\mathrm{an}}(f\otimes \omega_{p}^{i}) = 0 $ and $\lambda^{\mathrm{an}}(f \otimes  \omega_p^i) = 1$. Thus  $ (\mu_{p, f \times h, j}) \equiv C_{\mathcal{O}[[\Gamma]]}(S_{\mathrm{Gr}} (A_j/\mathbb{Q}_{\cyc})^\vee)  \equiv (T) \mod \pi$, for some unit $ u(T) \in \mathcal{O}[[T]] $ for $j=4,5$.
			\end{example}
			
			\begin{example}\label{example 2}
				Let $p =5$.  Let $f(z) = \sum a(n,f) q^n  \in S_{6}(\Gamma_0(52))$ be the cusp form given by LMFDB label 52.6.a.a. We have $f(z) = q -5 q^3 -3 q^5 + 53 q^7 -218 q^9 -702 q^{11}+ \cdots$. Thus $f$ is $p$-ordinary and $p \nmid N_{f} $.

				Consider the Eisenstein series $g'(z) = E_2(z) - 11 E_2(11z)$, where $ E_{2}(z) = \frac{1}{24} + \sum_{n=1}^{\infty} \sigma(n) q^n $. Then by a result of Mazur  \cite[Proposition 5.12 (ii)]{Mazur}, there exists a normalised Hecke  eigenform $h \in S_2(\Gamma_{0}(11))$ (LMFDB label 11.2.a.a) such that 
				$
				a(n,h) \equiv a(n,g') \mod 5.
				$
				
				We verify Theorem~\ref{congruence main conjecture } for $f\otimes h$.
				
				By [DS, Theorem 9.6.6], we have $ \bar{\rho}_{h} \simeq \bar{\rho}_{g'} \simeq \bar{\omega}_{p} \oplus 1 $. Thus $\bar{\xi}_{1}=\bar{\omega}_{p}$ and $\bar{\xi}_{2}=1$. By our construction, $\xi_{1}= \omega_{p} $ and $\xi_{2} = 1$. Thus $g(z)=E (1_p ,1)(z)= E_2(z) - pE_2(pz)$, where $1_p$ is the trivial character modulo $p = 5$. As $ \mathrm{cond}(\rho_{h}) = 11 $ and $ \mathrm{cond}(\bar{\rho}_{h}) = 5  $, we obtain $ m = 11 $ by \eqref{definition m}. Thus $K =\Q$, $\Sigma=\{ 5, 11, \infty \} $ and $\Sigma_0 = \{ 11 \}$.
				
				Let  $  \mathcal{F}  $ be the $\Lambda$-adic newform passing through $f$. Specializing $\mathcal{F}$ at $(2, \zeta_{5^{0}})$ gives a  weight $2$ cusp form, say, $f_2$.
				Then it known that $f_2 \in S_{2}(\Gamma_0(260), \iota_{260})$ is congruent to $f$ modulo $5$ and $f_2$ is either a newform or $5$-stabilisation of a newform in $S_2(\Gamma_0(52),\iota_{52})$. From the LMFDB data, the  space of newforms of level $260$, weight $2$ and trivial nebentypus $S^{\mathrm{new}}_2(\Gamma_0(260),\iota_{260})$  has dimension $4$. Also, every newform in $S^{\mathrm{new}}_2(\Gamma_0(260),\iota_{260})$ has $5^{\mathrm{th}}$ Fourier coefficient equal to $1$. However $  a(5,f) = -3 \not\equiv 1 \mod 5$.  Thus $f_2$ must be the $5$-stabilization of a newform in $S_2(\Gamma_0(52),\iota_{52})$. From LMFDB, $S^{\mathrm{new}}_2(\Gamma_0(52),\iota_{52})$  has dimension $1$ and is spanned by cusp form given by LMFDB label $52.2.a.a$. Further this cusp form corresponds to the elliptic curve $E$ given by $y^2 = x^3 + x - 10$ and we denote it by $f_E$. Note  that the residual representation $ \bar{\rho}_{E} $ is an irreducible $ G_{\mathbb{Q}} $-module and so is $\bar{\rho}_{f}$.  Thus  $ f  $ satisfies \ref{irr-f}. It follows from Theorem~\ref{rhof} (iii) that $f$ is  $p$-distinguished. Hence the hypotheses of  Theorem~\ref{analytic final} hold for $f$.
				
				Next we check  the hypotheses of Theorem~\ref{thm:congruence of ideals}. We have $ N = N_f = 52 $, $ M_0 =1 $, so $ (N,M_0) =1 $ and the nebentypus of $ h $ i.e. $\psi$ is trivial.  Note that $ \mathrm{Frob}_{p} $ acts on $ A_{f}^{-}(\xi_{2}) $ by the scalar $ a(p,f) \xi_{2}(p) =  a(p,f) = -3 $. As $ a(p,f) \xi_{2} \not \equiv 1 \mod 5$, it follows that $H^0(G_{\Q_{\cyc,w}}, A_f^{-}(\xi_2)) = 0$ and  we need not check \eqref{assumption}. 
				
				As explained in Example~\ref{example 1}, it suffices to check that $ \mu^{\text{an}}(f_{E} \otimes \xi_{i}\omega_{p}^{-j}) = 0 $ for $i=1,2$ and $1 \leq j \leq 4$ to conclude all the conditions in Theorem~\ref{thm:congruence of ideals}  are met. Since $\xi_1 = \omega_p$ and $\xi_2 =1$, it is enough to show $ \mu^{\text{an}}(f_{E} \otimes \omega_{p}^{-j}) = 0$ for $1 \leq j \leq 4$.  
				To do this we proceed as in Example~\ref{example 1}. Using computations on SAGE, we get
				\begin{small}
					\begin{center}
						\begin{tabular}{| c | c |  c |  c  | c | c | c | c | c | c | c |}
							\hline
							$\frac{b}{5}$  & $1/5$ & $ 2/5 $ & $ 3/5 $ & $4/5$  \\ 
							\hline
							$ \text{x}^{+} $ &  $1$ &  $1$ & $1$ & $1$  \\  
							\hline
							$ \text{x}^{-} $  & $1$ &  $1$ & $-1$ & $-1$, \\ 
							\hline  
						\end{tabular}
					\end{center}
				\end{small}%
				where $ \text{x}^{\pm} $ are the $ \pm $ modular symbols associated to $ f_{E}  $, as defined by Manin.
				
				Let  $\zeta_{4}$ be a primitive root of unity and $ \omega_{p}(2) = \zeta_{4} $ and $ \omega_{p}(-1) = -1 $. Then it follows $ \omega_{p}(3)  =  -  \zeta_{4} $. Using the analogue of \eqref{eq: modular symbol and L-value} in this case,  we have 
				\begin{small}
					\begin{align*}
						\mu_{p,f_{E}, 1, j} (0) \sim \begin{cases}
							(2 + \omega_p^{j}(2)  +  \omega_p^{j}(3) ), &\text{if } j =2, \\
							(2 + \omega_p^{j}(2)  -  \omega_p^{j}(3)) , &\text{if } j = 1,3,
						\end{cases}
					\end{align*}
				\end{small}%
				here $ \sim $ denotes equality up to multiplication by a $p$-adic unit. Taking $ j = 0,1, 2,3 $ we obtain 
				\begin{small}
					\begin{alignat*}{3}
						\mu_{p, f_{E}, 1, 1}(0)& \sim  2(1 +  \zeta_{4}) 
						\qquad \qquad& \mu_{p, f_{E}, 1, 2}(0) &= 0 \\
						\mu_{p, f_{E}, 1, 3}(0)&\sim  2(1- \zeta_{4}) 
						\qquad \qquad & \mu_{p, f_{E}, 1, 0}(0) & \sim 1.
					\end{alignat*}
				\end{small}%
				For $ j=2 $,  $ E \otimes \omega_p^2$ is the elliptic curve given by LMFDB label 1300.d1 with additive reduction at $p=5$. From the LMFDB, we observe that $ E \otimes \omega_p^2$ has analytic rank $1$ i.e. $L(s,  E \otimes \omega_p^2)$ has a simple zero at $s=1$. Using SAGE, we get  $\mu^{\text{an}}(f_E \otimes  \omega_p^2) =0$ and $\lambda^{\text{an}}(f_E \otimes  \omega_p^j) =1$. 
				Thus $\mu^{\text{an}}(f_E \otimes  \omega_p^j) =0$ for $1 \leq j \leq 4$ and all the hypotheses of Theorem~\ref{thm:congruence of ideals} are satisfies.

				Again by \cite[Theorem 17.4]{kato}, we get $\mu^{\text{alg}}(f_E \otimes  \omega_p^j) =\mu^{\text{an}}(f_E \otimes  \omega_p^j) =0$ and $\lambda^{\text{alg}}(f_E \otimes  \omega_p^j) =\lambda^{\text{an}}(f_E \otimes  \omega_p^j) =0$ for $1 \leq j \leq 4$ and $j \neq 2$. As explained in Example~\ref{example 1}, it follows from \cite[Theorem 17.4]{kato} and $p$-parity conjecture (See \cite[Theorem 1]{nekovar})  that  $\lambda^{\text{alg}}(f_E \otimes  \omega_p^2) =\lambda^{\text{an}}(f_E \otimes  \omega_p^2) =1$. Hence Iwasawa main conjecture holds for $f_E \otimes \omega_p^j$ for $1 \leq j \leq 4$. By \cite[Corollary 5.1.4]{epw} it follows that Iwasawa main conjecture also holds for $f \otimes \omega_p^j$. 
				
				Hence the hypotheses of Theorem~\ref{congruence main conjecture } are met and the Iwasawa main conjecture holds $ f \otimes h $ modulo $ 5 $. Further $ \mu^{\mathrm{an}}(f\otimes \omega_{p}^{j}) = 0 $ and $\lambda^{\mathrm{an}}(f \otimes  \omega_p^j) = 0$ for $ j \neq 2 $, it follows that  $\mu_{p, f, \xi_i,j}(T)  $ is unit whenever $ j \neq 2$. Applying Theorem~\ref{analytic final},  we obtain 
				$ \mu_{p, f \times h, j} $ is a unit whenever $ j \neq 1,2 $. Thus, it follows from \eqref{IMC mod p for example 1} that $ C_{\mathcal{O}[[\Gamma]]}(S_{\mathrm{Gr}} (A_j/\mathbb{Q}_{\cyc})^\vee) $ is a unit in $ \mathcal{O}[[\Gamma]] $ for $ j \neq 1,2 $ and 
				\begin{equation}\label{IMC example 2}
					(\mu_{p, f \times h, j}) = (C_{\mathcal{O}[[\Gamma]]}(S_{\mathrm{Gr}}(A_j)/\mathbb{Q}_{\cyc})^\vee)  = \mathcal{O}[[T]] \quad \mathrm{for} ~  j = 3,4.
				\end{equation}
				
				Also  $ (\mu_{p, f \times h, j}) \equiv C_{\mathcal{O}[[\Gamma]]}(S_{\mathrm{Gr}} (A_j/\mathbb{Q}_{\cyc})^\vee)   \equiv (T) \mod \pi $ for  $j = 1,2$. 
			\end{example}
			
			\begin{example}\label{example 3}
				Let $p =5$.  Let $E$ be the elliptic curve given by $y^2+y=x^3+x^2-9x-15$. By modularity theorem, $E$ corresponds to a weight $2$ cusp form $f_E = q - 2q^3 - 2q^4 + 3q^5 - q^7 + q^9 + 3q^{11} + \cdots$ (LMFDB label 19.2.a.a). Then $f_E$ is $5$-ordinary and $\bar{\rho}_{f_E}$ is irreducible. Let  $  \mathcal{F}  $ be the $\Lambda$-adic cusp form passing through $f_E$. Specializing $\mathcal{F}$ at $(6, \zeta_{5^{0}})$ gives a  weight $6$ cusp form of the family, say $f_6$. Then it is known that $f_6 \in S_{6}(\Gamma_0(85), \iota_{85})$ is congruent to $f_E$ modulo $5$. From the LMFDB data,  $S^{\mathrm{new}}_2(\Gamma_0(85),\iota_{85})$   doesn't contain any $5$-ordinary cusp form. Hence $f_6$ is  $5$-stabilisation of a newform $f$ in $S_6(\Gamma_0(19),\iota_{19})$. This newform $f$ belongs to newform orbit given by  LMFDB label 19.6.a.d. 
				
				Let $f$ be as above and take $h \in S_2(\Gamma_{0}(11))$ (LMFDB label 11.2.a.a)  as described in Example~\ref{example 2}.  We verify Theorem~\ref{congruence main conjecture } for $f\otimes h$.
				
				Since $f \equiv f_E$ it follows that $f$ satisfies \ref{irr-f} and \ref{p-dist}. Again in this case, the hypotheses of Theorem~\ref{analytic final} hold for $f$.
				
				Next we check  the hypotheses of Theorem~\ref{thm:congruence of ideals}. In this case $ N = N_f = 19$, $ M_0 =1 $, $ (N,M_0) =1 $ and the nebentypus of $ h $ i.e. $\psi$ is trivial.   Again we compute  $a(p,f_E) \not \equiv 1 \mod 5$ and as explained in the previous examples, it follows that $H^0(I_{\cyc,v}, A_f^{-}(\xi_2)) = 0$, so we need not check \eqref{assumption}.
				
				Also, as in Example~\ref{example 1} and \ref{example 2}, it suffices to check that $ \mu^{\mathrm{an}}(f_E \otimes \xi_{i}\omega_{p}^{-j}) = 0 $ for $i=1,2$ to conclude all the conditions in Theorem~\ref{thm:congruence of ideals}  are met. Since $\xi_1 = \omega_p$ and $\xi_2 =1$, it is enough to show $ \mu(S^{\Sigma_0}_{\mathrm{Gr}}(A_{f_{E}}(\omega_{p}^{-j})/\mathbb{Q}_{\cyc})^\vee) = 0$ for $1 \leq j \leq 4$. Again take $\Omega^{\pm}_{f_E}$ be as in Theorem~\ref{choice of period and petterson innerproduct}. Proceeding as  in the  previous two examples, we compute via SAGE  the values of the  modular symbols $\text{x}^{\pm}$:  
				\begin{small}
					\begin{center}
						\begin{tabular}{| c | c |  c |  c  | c | c | c | c | c | c | c |}
							\hline
							$\frac{b}{5}$  & $1/5$ & $ 2/5 $ & $ 3/5 $ & $4/5$  \\ 
							\hline
							$ \text{x}^{+} $ &  $-1/2$ &  $1$ & $1$ & $-1/2$  \\  
							\hline
							$ \text{x}^{-} $  & $1/2$ &  $0$ & $0$ & $-1/2$ \\ 
							\hline  
						\end{tabular}
					\end{center}
				\end{small}%
				Let  $\zeta_{4}$ be a primitive root of unity such that $ \omega_{p}(2) = \zeta_{4} $. Then it follows $ \omega_{p}(3)  =  -  \zeta_{4} $. Thus 
				\begin{small}
					\begin{align*}
						\mu_{p,f_{E}, 1, j} (0) \sim \begin{cases}
							(-1 + \omega_p^{j}(2)  +  \omega_p^{j}(3) ), &\text{if } j ~ \text{is even and } 4 \nmid j, \\
							1 , &\text{if } j ~ \text{is odd},
						\end{cases}
					\end{align*}
				\end{small}%
				here $ \sim $ denotes equality up to multiplication by a $p$-adic unit. Taking $ j = 0,1, 2,3 $ we obtain 
				\begin{small}
					\begin{alignat*}{3}
						\mu_{p, f_{E}, 1, 1}(0)& \sim  1 
						\qquad \qquad& \mu_{p, f_{E}, 1, 2}(0) &\sim  -3 \\
						\mu_{p, f_{E}, 1, 3}(0)&\sim  1 
						\qquad \qquad & \mu_{p, f_{E}, 1, 0}(0) & \sim 1/3.
					\end{alignat*}
				\end{small}%
				Hence $\mu^{\text{an}}(f_E \otimes  \omega_p^j) =0$ and  $\lambda^{\text{an}}(f_E \otimes \omega_p^j) =0$ for $1 \leq j \leq 4$. Thus we have verified all the hypotheses of Theorem~\ref{thm:congruence of ideals}.

				Once again by \cite[Theorem 17.4]{kato},  it follows that $\mu^{\text{alg}}(f_E \otimes  \omega_p^j) =\mu^{\text{an}}(f_E \otimes  \omega_p^j) =0$ and $\lambda^{\text{alg}}(f_E \otimes  \omega_p^j) =\lambda^{\text{an}}(f_E \otimes  \omega_p^j) =0$ for $1 \leq j \leq 4$. It follows that the Iwasawa main conjecture holds for $f \otimes \omega_p^j$ and $\mu_{p,f, 1,j}$ is a $p$-adic unit for all $1 \leq j \leq 4$. 
				
				Hence the hypotheses of Theorem~\ref{congruence main conjecture } are met and the Iwasawa main conjecture holds $ f \otimes h $ modulo $\pi$. Further,  applying Theorem~\ref{analytic final},  we obtain 
				$ \mu_{p, f \times h, j} $ is a unit for  all $ j $. Thus, it follows from \eqref{IMC mod p for example 1} that $ C_{\mathcal{O}[[\Gamma]]}(S_{\mathrm{Gr}} (A_j/\mathbb{Q}_{\cyc})^\vee) $ is a unit in $ \mathcal{O}[[\Gamma]] $ for all $ j $ and
				\begin{equation}\label{IMC example 3}
					(\mu_{p, f \times h, j}) = (C_{\mathcal{O}[[\Gamma]]}(S_{\mathrm{Gr}}(A_j)/\mathbb{Q}_{\cyc})^\vee)  = \mathcal{O}[[T]] \quad \mathrm{for} ~ 1 \leq j \leq 4.
				\end{equation}
			\end{example}

		\end{document}